\documentclass[reqno,11pt]{article}
\usepackage{graphicx, amsmath, amsthm,amssymb,color}  
\usepackage{tikz}
\usetikzlibrary{matrix,arrows}
\usepgflibrary{fpu}

\input{stochbif.sty}

\input{regstruct.sty}

\def\unit{\vone}

\def\bPi{\mathbf{\Pi}}
\def\Mhat{\hat M}
\def\hlinespace{\vrule height 11pt depth 5pt width 0pt}
\def\Rplus{\boldsymbol{R^+}}
\def\normDgamma#1{|\!|\!|#1|\!|\!|}
\def\seminormff#1#2{|\!|\!|#1\,\mathord{;}\,#2|\!|\!|}
\def\seminormbox#1{[\hspace{-.1em}]#1[\hspace{-.1em}]}

\def\KQ{K^Q}

\def\chiQ{\chi^Q}

\def\JQhat{\hat\cJ^Q}

\DeclareMathOperator{\vspan}{span}

\DeclareMathOperator{\Id}{Id}

\def\fraks{\mathfrak{s}}
\def\fraK{\mathfrak{K}}
\def\fraKbar{\overline{\mathfrak{K}}}
\def\abss#1{\abs{#1}_\mathfrak{s}}

\def\Tbar{\overline T}
\def\Zbar{\overbar Z}
\def\Gammabar{\overline\Gamma}
\def\Pibar{\overline\Pi}
\def\Zhat{\widehat Z}
\def\Gammahat{\widehat\Gamma}
\def\Pihat{\widehat\Pi}

\begin{document}

%%%%%%%%%%%%%%%%%%%%%%%%%%%%%%%%%%%%%%%%%%%%%%%%%%%%%%%%%%%%%%%%%%%%%%%%%%%%%%

\title{Regularity structures and renormalisation of\\
FitzHugh--Nagumo SPDEs in three space dimensions}
\author{Nils Berglund and Christian Kuehn}
\date{}   
%\date{\today}   

\maketitle

\begin{abstract}
\noindent
We prove local existence of solutions for a class of suitably renormalised
coupled SPDE--ODE systems driven by space-time white noise, where the space
dimension is equal to $2$ or $3$. This class includes in particular the
FitzHugh--Nagumo system describing the evolution of action potentials of a large
population of neurons, as well as models with multidimensional gating variables.
The proof relies on the theory of regularity structures recently developed by
M.\ Hairer, which is extended to include situations with semigroups that are not
regularising in space. We also provide explicit expressions for the
renormalisation constants, for a large class of cubic nonlinearities.
\end{abstract}

\leftline{\small{\it June 16, 2015.\/} }
\leftline{\small{\bf Submitted version.}}
%\leftline{\small{\bf Preprint version:} Comments and suggestions welcome.}
\leftline{\small 2010 {\it Mathematical Subject Classification.\/} 
% 60 Probability theory and stochastic processes
60H15, %Stochastic partial differential equations
% 35 Partial differential equations
35K57 (primary),    % Reaction-diffusion equations
% 81 Quantum theory
81S20,   %Stochastic quantization
% 82 Statistical mechanics, structure of matter
82C28 (secondary).  %Dynamic re normalisation group methods 
}

%%%%%%%%%%%%%%%%%%%%%%%%%%%%%%%%%%%%%%%%%%%%%%%%%%%%%%%%%%%%%%%%%%%%%%%%%%%%%%

\section{Introduction}
\label{sec_intro}

Many spatially extended systems can be described by a reaction-diffusion
equation coupled to an ordinary differential equation, of the form 
\begin{align}
\nonumber
\partial_t u &= \mathfrak{D} \Delta_x u + F(u,v)\;, \\
\partial_t v &= G(u,v)\;,
\label{eq:intro01} 
\end{align}
where $u=u(t,x)$ and $v=v(t,x)$ are functions of time $t\geqs0$ and space
$x\in\R^d$ and take values in $\R^m$ and $\R^n$ respectively, $\Delta_x$ denotes
the usual Laplacian, $\mathfrak{D}\in\R^{m\times m}$ is a positive definite
matrix of diffusion coefficients, and $F:\R^m\times \R^n\rightarrow \R^m$ and
$G:\R^m\times \R^n\rightarrow \R^n$ are given sufficiently smooth functions. A
well-studied situation of this kind arises in neuroscience, where $x\in\R$
measures the position along a neuron's axon, $u(t,x)\in\R$ is the membrane
potential at time $t$ and position $x$, and the so-called gating variables
$v\in\R^n$ describe the number of open ion channels of $n$ different
types. Classical examples are the Hodgkin--Huxley model~\cite{HodgkinHuxley52}
and its simplifications such as the Morris--Lecar model~\cite{MorrisLecar81} and
the FitzHugh--Nagumo model~\cite{FitzHugh61,Nagumo62}.  However, there are many
more interesting situations outside neuroscience that can be described by
coupled systems of the form~\eqref{eq:intro01}, for instance in chemical
kinetics~\cite{TysonKeener}, in population dynamics~\cite{CantrellCosner}, and
in pattern formation~\cite{Barkley2}. Furthermore, many other classical
reaction-diffusion equations from mathematical biology, such as the Keller-Segel
\cite{KellerSegel} or the Gray-Scott model~\cite{GiererMeinhardt} are of the
form~\eqref{eq:intro01} if one regards the model parameters $v$ not as constants
via $\partial_t v = 0$ but as (slowly-)changing parameters via $\partial_t
v=\tilde{\varepsilon} \widetilde{G}(u,v)$ for some small parameter
$0<\tilde\varepsilon \ll1$ \cite{KuehnBook} and some smooth vector field
$\widetilde{G}$. In fact, the number of applications of systems of the
form~\eqref{eq:intro01} goes far beyond the areas we mentioned here.

However, in some cases, a deterministic model of the form~\eqref{eq:intro01} is
not sufficient to describe the dynamics, and one has to add a stochastic term to
capture the effect of thermal fluctuations and other unresolved dynamical
processes affecting the system. In the case of neuron dynamics, see for instance
the book and recent survey by Tuckwell~\cite{Tuckwell,Tuckwell2} and the recent
reviews by Bressloff~\cite{Bressloff_2012,Bressloff_2014}. In particular,
stochastic versions of the FitzHugh-Nagumo PDE have attracted considerable
recent attention~\cite{Tuckwell1,Tuckwell_Jost_2011,SauerStannat1} as well as
related models of stochastic neural
fields~\cite{KuehnRiedler,Kruger_Stannat_2014}. In the derivation of these
models and in related numerical studies, one frequently uses
correlated~\cite{LordPowellShardlow} as well as space-time white
noise~\cite{AlzubaidiShardlow} stochastic forcing.

Before studying properties of solutions to a stochastic partial differential
equation (SPDE) with space-time white noise, one has to ensure that such an
equation is well-defined, i.e., one has to attach a meaning of \lq\lq
solution\rq\rq\ to a given SPDE. Depending on the space dimension $d$, this
problem may in fact be extremely hard. Consider the case of a one-component
system ($m=1$, $n=0$, $\mathfrak{D}\equiv 1$) of the form 
\begin{equation}
 \label{eq:intro02}
 \partial_t u = \Delta_x u + F(u) + \xi\;, 
\end{equation} 
where $\xi$ denotes space-time white noise (precise mathematical definitions
will be given below). If $d=1$, it is possible to define a notion of mild
solution via the Duhamel principle in a quite general setting (see for
instance~\cite{DaPrato_Jabczyk_92}). In higher space dimension, the problem
is much more difficult, owing to the fact that space-time white noise is
extremely singular. 

Consider for instance the case of the Allen--Cahn equation given
by~\eqref{eq:intro02} with $F(u)=au-u^3$ (if $a=0$ this is also known in
quantum field theory as the dynamical $\Phi_d^4$ model, cf~\cite{Feldman74}).
In the case $d=1$, the proof of existence of a unique solution goes back to
Faris and Jona-Lasinio~\cite{Faris_JonaLasinio82}, and many quantitative
properties of this solution are known (see for
instance~\cite{Cerrai_1996,BG12a,Barret_2015}). The two-dimensional case was
solved by Da Prato and Debussche using a particular class of Besov
spaces~\cite{DaPrato_Debussche_2003}. The case $d=3$, however, was only solved
very recently by Hairer, using his theory of regularity
structures~\cite{Hairer2014}. One of the difficulties is that a renormalisation
procedure has to be used to properly define solutions. This is achieved by
replacing space-time white noise $\xi$ by a mollified version $\xi^\eps$,
solving the resulting regularised equation, and passing to the limit of
vanishing regularisation (for an alternative approach,
see~\cite{Kupiainen_2014}). 

The theory of regularity structures provides a framework to study SPDEs with
very singular noise, including but not limited to parabolic equations of the
form~\eqref{eq:intro02}. The basic idea is to construct an abstract space in
which one can define algebraic operations on distributions (in the sense of
generalised functions), such as multiplication, composition with a smooth
function, and convolution with a kernel. The fixed-point equation obtained by
applying Duhamel's principle to the SPDE with mollified noise is then lifted to
the abstract space, solved in that space, and finally projected down to a
distribution in \lq\lq physical\rq\rq\ space. In addition, the theory allows to
incorporate the renormalisation procedure that is needed in most cases; see for
instance~\cite{Hairer_Ln_2013} for an introduction. Furthermore, there are
already several recent applications of regularity structures, e.g., to large
deviation theory of the Allen-Cahn equation~\cite{Hairer_Weber_15}, to the
dynamical sine-Gordon model~\cite{HairerShen14}, to Wong-Zakai approximation of
nonlinear parabolic SPDEs~\cite{HairerPardoux14}, to the parabolic Anderson
model on bounded~\cite{Hairer2014} as well as unbounded~\cite{HairerLabbe15}
domains, and to the KPZ equation~\cite{Hairer2014,Hairer2}. 

The purpose of the present work is to extend the theory of regularity structures
to multicomponent systems of the form
\begin{align}
\nonumber
\partial_t u &= \Delta_x u + F(u,v)+\xi\;, \\
\partial_t v &= G(u,v)\;.
\label{eq:intro03} 
\end{align}
The main result is the proof of existence, in space dimensions $d=2$ and $3$, of
local solutions to the system~\eqref{eq:intro03} when $F$ is a cubic
polynomial and $G$ is linear. This includes the case of the standard
FitzHugh--Nagumo model, but also other equations such as the Koper
model~\cite{Koper} with diffusion in the fast component, which features
vectorial variables~$v$. We will mostly focus on the case $d=3$, first because
this is the physically relevant case in many applications, but also because it
is the technically more challenging case (the Allen--Cahn equation is known not
to be renormalisable for $d\geqs4$, so the same will hold \textit{a fortiori}
for multicomponent systems with cubic nonlinearities). 

The main technical difficulty that has to be overcome to obtain these results
comes from the fact that the semigroup associated with the second equation
in~\eqref{eq:intro03} is not at all regularising in space. Though one
expects that this loss of regularisation is somehow compensated by the fact that
no singular noise term acts on the equation for $v$, the general theory
in~\cite{Hairer2014} cannot be applied directly, because it uses in an essential
way the assumption that the heat kernel is smooth everywhere except at the
origin. Therefore, we have to extend the regularity structure constructed for
the Allen--Cahn equation with new abstract symbols representing integration with
respect to a singular kernel, and to prove suitable bounds involving these
symbols. Furthermore, one has to analyse which role new symbols play in the
renormalisation. We find that certain terms involving the singular kernel do not
have to be renormalised, while on the other hand, general cubic
nonlinearities yield renormalisation terms which do not appear in the
Allen--Cahn case.

The remainder of this work is organised as follows. Section~\ref{sec_res}
contains all main local existence results. Section~\ref{sec_rs} gives a summary
of the most important aspects of the theory of regularity structures contained
in~\cite{Hairer2014}, illustrated in the case of the Allen--Cahn equation. In
Section~\ref{sec_extend}, we present our results allowing to represent the
operation of integration with respect to a singular kernel.
Section~\ref{sec_fix} contains the fixed-point argument proving local existence
and uniqueness of solutions for the SPDE with mollified noise, and
Section~\ref{sec_renorm} deals with the renormalisation procedure which is
necessary to pass to the limit of vanishing mollification. Finally, in
Section~\ref{sec_proof} we complete the proofs of the main results. 

\medskip

\noindent
\textbf{Notations:} We write $\abs{x}$ to denote either the absolute value of
$x\in\R$ or the $\ell^1$-norm of $x\in\R^{d}$, while $\norm{x}$ denotes the
Euclidean norm of $x\in\R^{d}$. If $a,b\in\R$ we use $a\wedge b:=\min\{a,b\}$
and $a\vee b:=\max\{a,b\}$. If $f, g$ are two real-valued functions depending on
small parameters $\eps, \delta, \dots$ (which will be clear from the context),
we write $f\lesssim g$ to indicate that there exists a constant $C>0$ such that
$f \leqs Cg$ holds uniformly in the small parameters. We use the notation
$f\asymp g$ to indicate that we have both $f\lesssim g$ and $g\lesssim f$. 

\medskip

\noindent
\textbf{Acknowledgements:} {N.B.}~would like to thank the Institute for Analysis
and Scientific Computing at the Technical University in Vienna for kind
hospitality and financial support. {C.K.}~would like to thank the Austrian
Academy of Sciences ({\"{O}AW}) for support via an APART fellowship. {C.K.}~also
acknowledges support of the European Commission (EC/REA) via a Marie-Curie
International Re-integration Grant.

%%%%%%%%%%%%%%%%%%%%%%%%%%%%%%%%%%%%%%%%%%%%%%%%%%%%%%%%%%%%%%%%%%%%%%%%%%%%%%

\section{Results}
\label{sec_res}

%%%%%%%%%%%%%%%%%%%%%%%%%%%%%%%%%%%%%%%%%%%%%%%%%%%%%%%%%%%%%%%%%%%%%%%%%%%%%%

\subsection{The standard FitzHugh--Nagumo equation}
\label{ssec_rfhn}

We start by considering the particular case of the standard
FitzHugh--Nagumo SPDE, given by 
\begin{align}
\nonumber
\partial_t u &= \Delta_x u + u - u^3 + v + \xi\;, \\
\partial_t v &= a_1 u + a_2 v\;. 
\label{eq:FHN01} 
\end{align}
Here $\Delta_x=\sum_{j=1}^d \partial_{x_j}^2$ is the usual Laplacian, 
$u$ and $v$ are functions of time $t\geqs0$ and space $x\in\T^d$ (the torus
in dimension $d=2$ or $d=3$), $\xi$ stands for space-time white noise on
$\R\times\T^d$, and $a_1, a_2$ a real parameters. As such, the
system~\eqref{eq:FHN01} is not well-posed, and a renormalisation procedure is
required to define a notion of solution. To do this, we first choose a rescaled
mollifier 
\begin{equation}
 \label{eq:FHN02}
 \varrho_\eps(t,x) = \frac{1}{\eps^{d+2}}
\varrho\biggpar{\frac{t}{\eps^2},\frac{x}{\eps}}\;,
\end{equation} 
where $\varrho:\R^{d+1}\to\R$ is a smooth compactly supported function of
integral $1$. We set $\xi^\eps=\varrho_\eps*\xi$, where the star stands for
space-time convolution, and consider the sequence of equations 
\begin{align}
\nonumber
\partial_t u^\eps &= \Delta_x u^\eps + \bigbrak{1+C(\eps)}u^\eps - (u^\eps)^3
+ v^\eps + \xi^\eps\;, \\
\partial_t v^\eps &= a_1 u^\eps + a_2 v^\eps\;, 
\label{eq:FHN03} 
\end{align}
where $C(\eps)\in\R$. 
Then our first result is the following, which is close in spirit
to~\cite[Thm.~1.15]{Hairer2014}. 

\begin{theorem}[Standard FitzHugh--Nagumo SPDE]
\label{thm:FHN}
Assume $u_0$ belongs to the H\"older space $\cC^\eta$ for some $\eta>-\frac23$,
and $v_0$ belongs to the H\"older space $\cC^\gamma$ for some $\gamma>1$. Then
there exists a choice of the constant $C(\eps)$ such that the
system~\eqref{eq:FHN03} with initial condition $(u_0,v_0)$ admits a sequence of
local solutions $(u^\eps,v^\eps)$, converging in probability to a limit $(u,v)$
as $\eps\to0$. This limit is independent of the choice of mollifier $\varrho$.  
\end{theorem}

The proof of this theorem and the two next ones is given in
Section~\ref{sec_proof}. For the precise definition of H\"older spaces
$\cC^\eta$ with negative index, we refer to Definition~\ref{def:C_alpha} below.
By \emph{local solution} we mean that for any cut-off $L>0$, the solution is
defined up to the random time when the $\cC^\eta$-norm of $(u^\eps,v^\eps)$ 
first reaches $L$. In other words, we cannot exclude the possibility of
finite-time blow-up. 

The constant $C(\eps)$ admits an explicit expression in terms of the heat
kernel\footnote{Below, the renormalisation constants will rather be defined in
terms of a truncated version of the heat kernel. This does however not affect
their singular parts (terms of order $\eps^{-1}$ and $\log(\eps^{-1})$).}
\begin{equation}
 \label{eq:FHN04}
 G^{(d)}(t,x) = \frac{1}{\abs{4\pi t}^{d/2}} \e^{-\norm{x}^2/4t}
\indexfct{t>0}\;.
\end{equation}
In dimension $d=3$, $C(\eps)$ is the same as for the dynamical $\Phi_3^4$ model
considered in~\cite{Hairer2014}. Namely, setting $G^{(d)}_\eps =
G^{(d)}*\varrho_\eps$, one
has 
\begin{equation}
 \label{eq:FHN05}
 C(\eps) = 3\int_{\R^4}G^{(3)}_\eps(z)^2\6z
 - 18 \int_{\R^4}G^{(3)}(z)\cQ^\eps_0(z)^2\6z\;,
\end{equation} 
where $\cQ^\eps_0(z) = \int G^{(3)}_\eps(z_1)G^{(3)}_\eps(z_1-z)\6z_1$. In
particular, it has been known since the works of Feldman and
Osterwalder~\cite{Feldman74,Feldman_Osterwalder_76} that 
\begin{equation}
 \label{eq:FHN06}
 C(\eps) = \frac{C_1}{\eps} + C_2 \log(\eps^{-1}) + C_3
\end{equation} 
for some constants $C_1, C_2, C_3\in\R$. The constants $C_1$ and $C_3$ depend on
the choice of mollifier  $\varrho$, while $C_2$ is independent of $\varrho$
(cf.~\cite[Rem.~6.2]{Hairer_Ln_2013}). 

In dimension $d=2$, one can show (cf.~\cite[Rem.~2.14]{Hairer_Weber_15}) that
the
renormalisation constant is given by 
\begin{equation}
 \label{eq:FHN07}
 C(\eps) = 3\int_{\R^3}G^{(2)}_\eps(z)^2\6z = \frac{3}{4\pi} \log(\eps^{-1}) +
C_3\;,
\end{equation} 
for some constant $C_3\in\R$, depending again on the choice of $\varrho$. 

The proof of Theorem~\ref{thm:FHN} will also provide some information on the
structure of the solutions. Indeed, we always have 
\begin{equation}
\label{eq:FHN08} 
u^\eps(t,x) = \chi_\eps(t,x) + \varphi^\eps(t,x)\;, \\
\end{equation}
where $\chi_\eps = G^{(d)}*\xi^\eps$ is the stochastic convolution of heat
kernel and noise, and $\varphi^\eps$ converges to a function (as opposed to a
distribution). In other words, the only singular term in the limit $\eps\to0$ is
given by $\lim_{\eps\to0}\chi_\eps=G^{(d)}*\xi$, which is independent of the
nonlinear term in the equation. The function $v^\eps(t,x)$ has the same
structure, as it can be represented in terms of $u^\eps$ by solving a linear
inhomogeneous equation. 

%%%%%%%%%%%%%%%%%%%%%%%%%%%%%%%%%%%%%%%%%%%%%%%%%%%%%%%%%%%%%%%%%%%%%%%%%%%%%%

\subsection{More general cubic nonlinearities}
\label{ssec_rcubic}

We can now extend the above results to a more general class of systems, of the
form 
\begin{align}
\nonumber
\partial_t u &= \Delta_x u + F(u,v) + \xi\;, \\
\partial_t v &= a_1 u + a_2 v\;,
\label{eq:rcubic01} 
\end{align}
where $F$ is a cubic polynomial of the form
\begin{equation}
 \label{eq:rcubic02}
  F(u,v) = \alpha_1 u + \alpha_2 v + \beta_1 u^2 + \beta_2 uv + \beta_3 v^2
 + \gamma_1 u^3 + \gamma_2 u^2 v + \gamma_3 uv^2 + \gamma_4 v^3\;.
\end{equation} 
In that case, the renormalised equations take the form 
\begin{align}
\nonumber
\partial_t u^\eps &= \Delta_x u^\eps + \bigbrak{F(u^\eps,v^\eps) +
C_0(\eps) + C_1(\eps) u^\eps + C_2(\eps) v^\eps} +
\xi^\eps\;, \\
\partial_t v^\eps &= a_1 u^\eps + a_2 v^\eps\;, 
\label{eq:rcubic03} 
\end{align}
where $C_0(\eps), C_1(\eps), C_2(\eps) \in \R$, and we have the following
result:

\begin{theorem}[General cubic nonlinearities]
\label{thm:cubic}
Assume $u_0$ and $v_0$ satisfy the same assumptions as in Theorem~\ref{thm:FHN}.
Assume further that either $d=2$, or $d=3$ and $\gamma_2=0$. Then there exists a
choice of constants $C_0(\eps)$, $C_1(\eps)$ and $C_2(\eps)$ such that the
system~\eqref{eq:rcubic03} with initial condition $(u_0,v_0)$ admits a sequence
of local solutions $(u^\eps,v^\eps)$, converging in probability to a limit
$(u,v)$ as $\eps\to0$. This limit is independent of the choice of
mollifier~$\varrho$.  
\end{theorem}

Note that in dimension $d=3$, we assume that $F$ contains no term of the form
$\gamma_2 u^2v$. This is because the method we use does not yield a simple form
for the renormalised equation if such a term is present (see
Section~\ref{ssec_renorm_compute}), which is an indirect consequence of the fact
that the equation for $v$ is not regularising in space
(cf.~Section~\ref{ssec_Schauder}). It is not clear at this stage whether this is
just a technical artefact, or whether it has a deeper meaning.

The renormalisation constants can again be computed explicitly to leading order.
They are given by 
\begin{align}
\nonumber
C_0(\eps) &= -\frac13 \beta_1 C(\eps)\;, \\
\label{eq:rcubic04} 
C_1(\eps) &= -\gamma_1 C(\eps)\;, \\
C_2(\eps) &= -\frac13 \gamma_2 C(\eps)\;,
\nonumber
\end{align}
where $C(\eps)$ is given by 
\begin{equation}
 \label{eq:rcubic05}
 C(\eps) = 3\int_{\R^4}G^{(3)}_\eps(z)^2\6z
 + 18\gamma_1 \int_{\R^4}G^{(3)}(z)\cQ^\eps_0(z)^2\6z
\end{equation}
for $d=3$ and by~\eqref{eq:FHN07} for $d=2$. Note in particular that these
constants depend only on the coefficients of $u^2$, $u^2 v$ and $u^3$ in $F$. 

%%%%%%%%%%%%%%%%%%%%%%%%%%%%%%%%%%%%%%%%%%%%%%%%%%%%%%%%%%%%%%%%%%%%%%%%%%%%%%

\subsection{Vectorial gating variables}
\label{ssec_rvect}

Another generalisation of interest is to systems with vectorial variables $v$,
of the form 
\begin{align}
\nonumber
\partial_t u &= \Delta_x u + F(u,v) + \xi\;, \\
\partial_t v &= u A_1 + A_2 v\;,
\label{eq:rvect01} 
\end{align}
where $v(t,x)$ takes values in $\R^n$. Here $A_1\in\R^n$ is a constant vector,
and $A_2\in\R^{n\times n}$ is a square matrix, while $F$ is again a cubic
polynomial in $u$ and the $v_i$. This allows for instance to consider the Koper
model~\cite{Koper} with spatial diffusion and space-time white noise in the fast
variable; for the stochastic Koper model without the Laplacian representing
spatial diffusion see \cite{BerglundGentzKuehn1}. The reaction terms of one
version of the Koper model can be written in the form~\eqref{eq:rvect01} with 
\begin{equation}
 \label{eq:rvect02}
 F(u,v) = 3u + v_1 - u^3\;, 
 \qquad
 A_1 = 
 \begin{pmatrix}
 \epsilon_1k \\ 0 
 \end{pmatrix}\;, 
 \qquad
 A_2 = 
 \begin{pmatrix}
 -2\epsilon_1 & \epsilon_1 \\ \epsilon_1 & -\epsilon_1
 \end{pmatrix}\;,
\end{equation} 
where $k, \epsilon_1\in\R$ are model parameters. 

The natural candidate for the renormalised system associated to
\eqref{eq:rvect01} is given by 
\begin{align}
\nonumber
\partial_t u^\eps &= \Delta_x u^\eps + \biggbrak{F(u^\eps,v^\eps) +
C_0(\eps) + C_1(\eps) u^\eps + \sum_{i=1}^nC_{2,i}(\eps) v_i^\eps} +
\xi^\eps\;, \\
\partial_t v^\eps &= u^\eps A_1 + A_2 v^\eps\;, 
\label{eq:rvect03} 
\end{align}
for constants $C_0(\eps), C_1(\eps), C_{2,i}(\eps) \in\R$. 
Indeed, we have the following result.

\begin{theorem}[Vectorial variables $v$]
\label{thm:vectorial}
Assume $u_0$ and the components of $v_0$ satisfy the same assumptions as $u_0$
and $v_0$ in Theorem~\ref{thm:FHN}. Assume further that either $d=2$, or $d=3$
and $F(u,v)$ has no terms in $u^2v_i$. Then there exists a choice of constants
$C_0(\eps)$, $C_1(\eps)$, $C_{2,i}(\eps)$ such that the
system~\eqref{eq:rvect03} with initial condition $(u_0,v_0)$ admits a sequence
of local solutions $(u^\eps,v^\eps)$, converging in probability to a limit
$(u,v)$ as $\eps\to0$. This limit is independent of the choice of mollifier
$\varrho$.  
\end{theorem}

The renormalisation constants can again be expressed in terms of the
coefficients of the initial equation~\eqref{eq:rvect01}. Writing 
\begin{equation}
 \label{eq:rvect04}
 F(u,v) = \alpha_1 u + \beta_1 u^2 + \gamma_1 u^3 
 + \sum_{i=1}^n \bigbrak{\alpha_{2,i}v_i + \beta_{2,i}uv_i +
\gamma_{2,i}u^2v_i} + R(u,v)
\end{equation} 
with $\abs{R(u,v)} \leqs C\norm{v}^2$, we have 
\begin{align}
\nonumber
C_0(\eps) &= -\frac13 \beta_1 C(\eps)\;, \\
\label{eq:rvect05} 
C_1(\eps) &= -\gamma_1 C(\eps)\;, \\
C_{2,i}(\eps) &= -\frac13 \gamma_{2,i} C(\eps)\;,
\qquad i=1,\dots, n\;, 
\nonumber
\end{align}
where $C(\eps)$ is again the constant defined in~\eqref{eq:rcubic05} for $d=3$
and in~\eqref{eq:FHN07} for $d=2$. Note that in the particular case of the Koper
model, $C_1(\eps)$ is equal to the constant $C(\eps)$ obtained for the
FitzHugh--Nagumo equation, while all other constants vanish. This is not
surprising, since both models have the same nonlinearity. 

\begin{remark}
\label{rem:subcritical}
The notion of \emph{local subcriticality} given
in~\cite[Assumption~8.3]{Hairer2014} suggests that the class of renormalisable
SPDE--ODE models of the form~\eqref{eq:intro03} is larger than the one
considered here. More precisely, in dimension $d=3$, one would expect models
with quartic $F$ and linear $G$ to be renormalisable, as well as models with
$F$ quartic in $u$ and quadratic in $v$ and $G$ quadratic in $u$ and linear in
$v$. The reason we did not include them in our analysis
is a technical one: quartic $F$ would produce solutions whose singular component
is not homogeneous in space and time, as in~\eqref{eq:FHN08}. The way in which
we lift the singular kernel to the regularity structure is not able to deal with
such situations  (see Section~\ref{ssec_Schauder} below). We plan to further
investigate this issue in future work.~$\lozenge$
\end{remark}

%%%%%%%%%%%%%%%%%%%%%%%%%%%%%%%%%%%%%%%%%%%%%%%%%%%%%%%%%%%%%%%%%%%%%%%%%%%%%%

%%%%%%%%%%%%%%%%%%%%%%%%%%%%%%%%%%%%%%%%%%%%%%%%%%%%%%%%%%%%%%%%%%%%%%%%%%%%%%

\section{Regularity structure for the Allen--Cahn equation}
\label{sec_rs}

This section serves the double purpose of giving a very brief account of the
theory of regularity structures contained in~\cite{Hairer2014}, and of
describing a regularity structure for the Allen--Cahn equation
\begin{equation}
 \label{eq:Allen-Cahn}
 \partial_t u = \Delta_x u + u - u^3 + \xi\;,
\end{equation} 
where $\Delta_x=\sum_{j=1}^d \partial_{x_j}^2$ is the usual Laplacian,
$\xi=\xi(t,x)$ denotes space-time white noise on $\R\times\T^d$ where $\T^d$
denotes the $d$-dimensional torus with $d=2$ or $d=3$, and we seek a
solution $u:[0,T]\times\T^d\to\R$ for a given initial condition $u_0=u(0,x)$ in
a suitable function space. The regularity structure for \eqref{eq:Allen-Cahn}
will serve as our starting point to build a larger structure allowing to
represent the FitzHugh--Nagumo equation.

By Duhamel's principle, one possible solution concept for~\eqref{eq:Allen-Cahn}
is to consider the following integral equation
\begin{equation}
 \label{eq:AC-fpe}
 u_t = \int_0^t S(t-s) \bigbrak{u_s - u_s^3 + \xi_s}\6s + S(t)u_0\;,
\end{equation} 
where $S(t)=\e^{t\Delta_x}$ denotes the semigroup of the heat equation
compatible with the boundary conditions and $u_s:=u(s,\cdot)$,
$\xi_s:=\xi(s,\cdot)$. The purpose of a regularity structure is to provide an
abstract space in which one can construct a fixed point of~\eqref{eq:AC-fpe}
when $\xi$ is replaced by a mollified version $\xi^\eps$. Then the idea is to
take the limit $\eps\to0$, and to project the solution to a distribution on
$[0,T]\times\T^d$, where \lq\lq distribution\rq\rq\ is understood in the sense
of a \lq\lq generalized function\rq\rq\ representing a sample path.

%%%%%%%%%%%%%%%%%%%%%%%%%%%%%%%%%%%%%%%%%%%%%%%%%%%%%%%%%%%%%%%%%%%%%%%%%%%%%%

\subsection{Regularity structures}
\label{ssec_rs}

\begin{definition}[{\cite[Def.~2.1]{Hairer2014}}]
\label{def:regularity_structure}
A \emph{regularity structure} is a triple $(A,T,\cG)$ consisting of 
\begin{enumH}
\item[(R1)] an \emph{index set} $A\subset\R$, containing $0$, which is bounded
from
below and locally finite;
\item[(R2)] a \emph{model space} $T$, which is a graded vector space 
$T=\bigoplus_{\alpha\in A}T_\alpha$, where each $T_\alpha$ is a Banach space;
the
space $T_0$ is isomorphic to $\R$ and its unit is denoted $\unit$;
\item[(R3)] a \emph{structure group} $\cG$ of linear operators
acting on $T$, such
that 
\begin{equation}
 \label{eq:rs01}
 \Gamma\tau - \tau \in \bigoplus_{\beta<\alpha} T_\beta =: T_\alpha^- 
\end{equation} 
holds for every $\Gamma\in \cG$, every $\alpha\in A$ and every $\tau\in
T_\alpha$;
furthermore, $\Gamma\unit = \unit$ for every $\Gamma\in \cG$. 
\end{enumH}
\end{definition}

\begin{example}
\label{ex:rs_poly}
A simple but important example of regularity structure is the polynomial
regularity structure for $d$ variables. In that case, $A=\N_0$ is the set of
non-negative integers. For $\ell\in \N_0$, $T_\ell$ is the space of homogeneous
polynomials in $d$ variables of degree $\ell$. It is spanned by the monomials
$X^k = X_1^{k_1}\dots X_d^{k_d}$ for which $\abs{k} = k_1+\dots+k_d = \ell$. 
Finally, the structure group $\cG$ is defined by 
\begin{equation}
 \label{eq:rs02}
 \Gamma_h(X^k) = (X-h)^k\;, 
 \qquad h\in\R^d\;.
\end{equation}   
This group is isomorphic to $\R^d$, and one easily sees that $\Gamma_h$
satisfies the requirement~\eqref{eq:rs01}. The interpretation of $\cG$ is that
it allows to convert a Taylor expansion around a point $x\in\R^d$ into the
expansion around another point $x+h.~\blacklozenge$ 
\end{example}

We will henceforth denote by $\Tbar$ the polynomial regularity structure with
$d+1$ variables $X_0,\dots,X_d$, where $X_0$ represents the time variable. Since
the linear part of the Allen--Cahn equation is given by a parabolic operator, it
turns out to be useful to make the time variable \lq\lq count double\rq\rq. This
is done by defining the \emph{parabolic scaling}
\begin{equation}
 \label{eq:rs03}
 \fraks = (2,1,\dots,1) \in \N^{d+1}\;,
\end{equation} 
acting on $\R^{d+1}$ and declaring that a monomial $X^k$ has homogeneity
$\abss{X^k} = \abss{k}$ where the \emph{scaled degree} is defined by
$\abss{k} = 2k_0 + \sum_{i=1}^dk_i$. 

The regularity structure of the Allen--Cahn equation~\eqref{eq:rs01} in $\R^d$
is built by enlarging $\Tbar$, i.e., by adding new symbols other than the 
polynomial symbols $X^k$ to $\Tbar$ (see~\cite{Hairer_Weber_15}). The noise
is represented by a symbol $\Xi$. In order to account for the fact that
space-time white noise has H\"older regularity $\alpha$ for any $\alpha <
-(d+2)/2$, we set 
\begin{equation}
 \label{eq:rs04}
 \alpha_0 = -\frac{d+2}{2} - \kappa
\end{equation}
where $\kappa>0$ will be chosen sufficiently small in the sequel, and declare
that $\Xi$ has homogeneity $\abss{\Xi} = \alpha_0$. 

The set of symbols is equipped by a product, which by definition is commutative
and associative with unit $\unit$. The product of two elements $\tau,\sigma \in
T$ has homogeneity $\abss{\tau\sigma} := \abss{\tau} + \abss{\sigma}$.
Furthermore, integration against the heat kernel is represented by a map
$\cI:T\rightarrow T$, which by definition satisfies $\abss{\cI(\tau)} :=
\abss{\tau} + 2$, in order to account for the regularizing effect of the heat
kernel. 

\begin{table}
\begin{center}
\begin{tabular}{|l|c|l|l|l|l|}
\hline
\hlinespace
$\tau$ & Symb & $\abss{\tau}$ & $d=3$ & $d=2$ & $\Delta(\tau)$ \\
\hline
\hlinespace
$\Xi$   & $\Xi$   & $\alpha_0$ & $-\frac52-\kappa$ & \!\!$-2-\kappa$ & 
$\Xi\otimes\unit$ \\
\hlinespace
$\cI(\Xi)^3$ & \RSW & $3\alpha_0+6$ & $-\frac32 -3\kappa$ & $0-3\kappa$ &
$\RSW\otimes\unit$ \\
\hlinespace
$\cI(\Xi)^2$ & \RSV & $2\alpha_0+4$ & $-1-2\kappa$ & $0-2\kappa$ &
$\RSV\otimes\unit$ \\
\hlinespace
$\cI(\cI(\Xi)^3)\cI(\Xi)^2$ & \RSWW & $5\alpha_0+12$ & $-\frac12 -5\kappa$  &
$2-5\kappa$ & $\RSWW\otimes\unit + \RSV\otimes\cJ(\RSW)$ \\
\hlinespace
$\cI(\Xi)$   & \RSI & $\alpha_0+2$          & $-\frac12 -\kappa$  & $0-\kappa$ &
$\RSI\otimes\unit$ \\
\hlinespace
$\cI(\cI(\Xi)^3)\cI(\Xi)$ & \RSVW   & $4\alpha_0+10$ & $0-4\kappa$ & $2-4\kappa$
&
$\RSVW\otimes\unit + \RSI\otimes\cJ(\RSW)$ \\
\hlinespace
$\cI(\cI(\Xi)^2)\cI(\Xi)^2$ & \RSWV   & $4\alpha_0+10$ & $0-4\kappa$ &
$2-4\kappa$ & $\RSWV\otimes\unit + \RSV\otimes\cJ(\RSV)$
\\
\hlinespace 
$\cI(\Xi)^2 X_i$ & \RSV$X_i$ & $2\alpha_0+5$ & $0-2\kappa$ & $1-2\kappa$ &
$\RSV\otimes X_i + \RSV X_i \otimes \unit$ \\
\hlinespace
$\unit$ & $\unit$ & $0$ & $0$ & $0$ & $\unit\otimes\unit$ \\
\hlinespace 
$\cI(\cI(\Xi)^3)$ & \RSIW & $3\alpha_0+8$ & $\frac12 - 3\kappa$ &
$2-3\kappa$ & $\RSIW \otimes \unit + \unit \otimes \cJ(\RSW)$ \\
\hlinespace 
$\cI(\cI(\Xi)^2)\cI(\Xi)$ & \RSVV & $3\alpha_0+8$ & $\frac12 - 3\kappa$ &
$2-3\kappa$ & $\RSVV\otimes\unit + \RSI\otimes\cJ(\RSV)$ \\
\hlinespace 
$\cI(\cI(\Xi))\cI(\Xi)^2$ & \RSWI & $3\alpha_0+8$ & $\frac12 - 3\kappa$ &
$2-3\kappa$ & $\RSWI\otimes\unit + \RSV\otimes\cJ(\RSI)+$ \\
\hlinespace
&&&&& $\RSV X_i\otimes\cJ_i(\RSI) + \RSV\otimes X_i\cJ_i(\RSI)$\!\!
\\
\hlinespace
$\cI(\cI(\Xi)^2)$ & \RSY & $2\alpha_0+6$ & $1-2\kappa$ & $2-2\kappa$ &
$\RSY\otimes\unit + \unit\otimes\cJ(\RSV)$ \\
\hlinespace
$\cI(\cI(\Xi))\cI(\Xi)$ & \RSVI & $2\alpha_0+6$ & $1-2\kappa$ & $2-2\kappa$ &
$\RSVI\otimes\unit + \RSI\otimes\cJ(\RSI)+$\\
\hlinespace
&&&&& $\RSI X_i\otimes\cJ_i(\RSI) + \RSI\otimes X_i\cJ_i(\RSI)$ \\
\hlinespace 
$X_i$ & $X_i$ & $1$ & $1$ & $1$ & $X_i\otimes\unit + \unit\otimes X_i$ \\
\hlinespace
$\cI(\cI(\Xi))$ & \RSII & $\alpha_0+4$ & $\frac32-\kappa$ & $2-\kappa$ & 
$\RSII\otimes\unit + \unit\otimes\cJ(\RSI) + $ \\
\hlinespace
&&&&& $X_i\otimes\cJ_i(\RSI) + \unit\otimes X_i\cJ_i(\RSI)$ \\
\hline
\end{tabular}
\end{center}
\vspace{2mm}
\caption[]{Elements of $\cF_F$ of lowest homogeneity for the Allen--Cahn
equation. They have been ordered by increasing homogeneity for the
case $d=3$. The expressions for $\Delta(\tau)$ are shown for the case $d=3$.
We have written $\cJ$ instead of $\cJ_0$ and $\cJ_i$ instead of $\cJ_{e_i}$,
where the $e_i$ are canonical basis vectors of $\Z_+^{d+1}$. Summation over the
index $i$ is understood.}
\label{tab:FF_AllenCahn}
\end{table}

Let $\cF$ be the set containing all possible products of symbols
$\unit, \Xi,
X_i$ and their images by $\cI$, and denote by $\cH = \vspan(\cF)$ the vector
space spanned by all these symbols. This is an infinite-dimensional space but
in practice only a finite-dimensional subspace of $\cH$ will be needed. 
Let $\cU$ be the smallest set containing $\unit, X_i$ and $\cI(\Xi)$, and
such that 
\begin{equation}
 \label{eq:rs05}
 \tau_1, \tau_2, \tau_3 \in\cU 
 \quad \Rightarrow \quad
 \cI(\tau_1\tau_2\tau_3) \in \cU\;. 
\end{equation}
We then set 
\begin{equation}
 \label{eq:rs06}
 \cF_F = \set{\Xi} \cup \setsuch{\tau_1\tau_2\tau_3}{\tau_i\in\cU}\;,
\end{equation}
and define the model space as being the vector space $T=\vspan(\cF_F)$
spanned by $\cF_F$. The index set is then defined as 
\begin{equation}
 \label{eq:rs07}
 A = \bigsetsuch{\abss{\tau}}{\tau\in\cF_F}\;.
\end{equation}
The model space $T$ admits a natural grading 
\begin{equation}
 \label{eq:rs08}
 T = \bigoplus_{\gamma\in A} T_\gamma
\end{equation}
obtained by letting $T_\gamma$ be the vector space spanned by elements of
homogeneity $\gamma$. We equip each $T_\gamma$ with a norm
$\norm{\cdot}_\gamma$; the choice of norm is irrelevant
since~\cite[Lemma~8.10]{Hairer2014} shows the (nontrivial) fact that all
$T_\gamma$ are finite-dimensional. 

Table~\ref{tab:FF_AllenCahn} shows the elements in $\cF_F$ of lowest
homogeneity, using a graphical representation introduced by Hairer. Each symbol 
$\Xi$ is denoted by a dot, and integration with respect to $\cI$ is denoted by
a vertical line pointing downwards. For instance, $\cI(\Xi)=\RSI$.
Multiplication of symbols is represented by joining them at the base, so for
instance $\cI(\Xi)^2 = \RSV$.

Finally, we have to extend the definition of the structure group $\cG$, which is
the most involved part of the construction. The first step is to introduce a
set $\cF_+ \subset \cF$, which contains all formal expressions of the type 
\begin{equation}
 \label{eq:rs09}
 X^k \prod_{j} \cJ_{k_j}\tau_j\;,
\end{equation}
where $\tau_j\in\cF$ and the multiindices $k_j$ are such that
$\abss{\tau_j}+2-\abss{k_j}>0$. By definition, an expression of the form
\eqref{eq:rs09} has homogeneity $\abss{k} + \sum_j(\abss{\tau_j}+2-\abss{k_j})$,
which is always strictly positive, except for the element $\unit$, which has
homogeneity $0$. The set $\cF_F^+\subset\cF_+$ is defined similarly, but with
$\tau_j\in\cF_F$. We set $\cH_+=\vspan(\cF_+)$ and $T^+=\vspan(\cF_F^+)$. 

The last ingredient is given by a bilinear map %two bilinear maps
$\Delta:\cH\to\cH\otimes\cH_+$ defined by setting 
\begin{equation}
 \label{eq:rs10}
 \Delta\unit = \unit\otimes\unit\;, \qquad 
 \Delta X_i = X_i\otimes\unit + \unit\otimes X_i\;, \qquad 
 \Delta\Xi = \Xi\otimes\unit\;,
\end{equation} 
and extended inductively to all of $\cH$ by the rules 
\begin{align}
\nonumber
\Delta(\tau\sigma) &= (\Delta\tau)(\Delta\sigma)\;, \\
\Delta(\cI\tau) &= (\cI\otimes \Id)\Delta(\tau) 
+ \sum_{\ell, m} \frac{X^\ell}{\ell!} \otimes \frac{X^m}{m!} \cJ_{\ell+m}\tau\;,
\label{eq:rs11} 
\end{align}
where $\textnormal{Id}$ denotes the identity map and we impose that 
\begin{equation}
\label{eq:Jiszero}
\cJ_k\tau=0\qquad \text{if $\abss{k}\geqs \abss{\tau}+2$\;,}
\end{equation}
while $\cJ_k\tau$ is a new formal symbol otherwise. Table~\ref{tab:FF_AllenCahn}
also shows $\Delta(\tau)$ for the first few elements of $\cF_F$. 

Denote by $\cH_+^\star$ the dual of $\cH_+$, that is, the set of linear maps
$h:\cH_+\to\R$, that we will write as $\tau\mapsto\pscal{h}{\tau}$. Let
$\cG_\star$ be the set of grouplike elements $g\in\cH_+^\star$, that is, those
satisfying $\pscal{g}{\tau\sigma} = \pscal{g}{\tau} \pscal{g}{\sigma}$ for all
$\tau,\sigma\in\cH_+$. Then $\cG$ is obtained by identifying elements of
$\cG_\star$ acting the same way on $T^+$. The action of $\cG$ on $T$ is defined
by 
\begin{equation}
\label{eq:def_g} 
 (g,\tau) \mapsto \Gamma_g\tau = (\Id\otimes g)\Delta\tau\;.
\end{equation}
The fact that $(A,T,\cG)$ constructed in this way is indeed a regularity
structure is proved in~\cite[Thm.~8.24]{Hairer2014}. 

%%%%%%%%%%%%%%%%%%%%%%%%%%%%%%%%%%%%%%%%%%%%%%%%%%%%%%%%%%%%%%%%%%%%%%%%%%%%%%

\subsection{Models and reconstruction theorem}
\label{ssec_model}

A regularity structure can be used to represent an SPDE as a fixed-point problem
in an abstract space. One also has to provide a connection between the abstract
space and the \lq\lq physical\rq\rq\ space the solution lives in. 
First, we need to introduce a version of the classical H\"older spaces
$\cC^\alpha$ on $\R^{d+1}$. In order to take the parabolic scaling into account,
the Euclidean norm is replaced by 
\begin{equation}
\label{eq:parabolic_norm} 
\norm{(t,x)}_{\fraks} = \abs{t}^{1/2} + \sum_{i=1}^d \abs{x_i}\;.
\end{equation}
Denote by $\pscal{\unit}{\cdot}$ the element of the dual of $T$ defined by
$\pscal{\unit}{\unit} = 1$ and $\pscal{\unit}{\tau} = 0$ for all $\tau\in
\bigoplus_{\gamma\neq0}T_\gamma$. If $\alpha>0$, the space $\cC^\alpha_\fraks$
is defined as the set of all
functions $\varphi:\R^{d+1} \rightarrow \R$ such that there exists a function
$\hat{\varphi}: \R^{d+1}\rightarrow T_\alpha^-$ such that $\pscal{\unit}
{\hat{\varphi}(z)}=\varphi(z)$ for every $z=(t,x)\in\R^{d+1}$ and such that for
every compact set $\fraK\subset \R^{d+1}$ the estimate
\begin{equation}
\abs{\hat{\varphi}(z+h)-\Gamma_h\hat{\varphi}(z)}_\beta \lesssim
\abss{h}^{\alpha-\beta}
\end{equation}  
holds uniformly over $\beta<\alpha$, $\abss{h}\leqs 1$ and $z\in \fraK$
(cf~\cite[Def.~2.14]{Hairer2014}). Lemma~2.12 in~\cite{Hairer2014} shows
that in the case of the Euclidean scaling $\fraks=(1,\dots,1)$, this
definition coincides with the usual definition of H\"older spaces.

If $\alpha<0$, $\cC^\alpha_\fraks$ can be defined in a natural way as a subset
of Schwarz distributions $\cS'(\R^{d+1})$ as follows. Given
$r\in\N$, denote by $\cB^r_{\fraks,0}$ the space of functions of
class $\cC^r$ which are supported in the set
$\setsuch{z\in\R^{d+1}}{\norm{z}_{\fraks}\leqs1}$.

\begin{definition}[{\cite[Def.~3.7]{Hairer2014}}]
\label{def:C_alpha}
Assume $\alpha < 0$ and let $r=-\intpart{\alpha}$. A Schwarz distribution
$\upsilon\in\cS'(\R^{d+1})$ belongs to $\cC^\alpha_{\fraks}$ if it
belongs to
the dual of $\cC_0^r$ and for every compact set $\fraK\subset\R^{d+1}$,
there exists a constant $C$ such that 
\begin{equation}
 \label{eq:C_alpha}
 \pscal{\upsilon}{\cS^\delta_{\fraks,z}\eta} \leqs C\delta^\alpha
\end{equation} 
holds for all $\eta\in\cB^r_{\fraks,0}$ with $\norm{\eta}_{\cC^r}\leqs1$,
all $\delta\in(0,1]$ and all $z\in\fraK$. Here 
\begin{equation}
 \label{eq:Sdelta}
 (\cS^\delta_{\fraks,(t,x)}\eta)(\bar t,\bar x) :=
\delta^{-(d+2)}\eta\bigpar{\delta^{-2}(\bar t-t),\delta^{-1}(\bar x - x)}\;.
\end{equation} 
\end{definition}

Next we introduce the notion of a \emph{model} associated with a regularity
structure $(A,T,\cG)$. 

\begin{definition}[{\cite[Def.~2.17]{Hairer2014}}]
\label{def:model} 
A \defwd{model} for a regularity structure $(A,T,\cG)$ with scaling~$\fraks$ is
a pair $(\Pi,\Gamma)$, defined by a collection
$\set{\Pi_z:T\to\cS'(\R^{d+1})}_{z\in\R^{d+1}}$ of continuous linear maps and a
map $\Gamma:\R^{d+1}\times\R^{d+1}\to \cG$ with the following properties. 
\begin{enumH}
\item[(M1)] 	$\Gamma_{zz} = \Id$ is the identity of
$\cG$ and
$\Gamma_{zz'}\Gamma_{z'z''}=\Gamma_{zz''}$ for
all $z,z',z''\in\R^{d+1}$. 

\item[(M2)] 	$\Pi_{z'}=\Pi_z \Gamma_{zz'}$ for all $z,z'\in\R^{d+1}$. 

\item[(M3)] 	For any $\gamma\in\R$ and any compact set
$\fraK\subset{\R^{d+1}}$, one has 
\begin{equation}
 \label{eq:defmod1}
 \norm{\Pi}_{\gamma;\fraK} := 
 \sup_{z\in\fraK}~ \sup_{\alpha<\gamma} ~\sup_{\tau\in T_\alpha}~
\sup_{\eta\in\cB^r_{\fraks,0}}~\sup_{0<\delta\leqs1}
 \frac{\bigabs{\pscal{\Pi_z\tau}{\cS^\delta_{\fraks,z}\eta}}}{\delta^{
\alpha} \norm{\tau}} 
 < \infty\;.
\end{equation}

\item[(M4)] 	For any $\gamma\in\R$ and any compact set
$\fraK\subset{\R^{d+1}}$, one has 
\begin{equation}
 \label{eq:defmod2}
 \norm{\Gamma}_{\gamma;\fraK} := 
 \sup_{z,z'\in\fraK} ~\sup_{\alpha<\gamma}~
\sup_{\beta<\alpha}~\sup_{\tau\in T_\alpha}
 \frac{\norm{\Gamma_{zz'}\tau}_\beta}
 {\norm{\tau}\norm{z-z'}^{\alpha-\beta}_{\fraks}} < \infty\;.
\end{equation} 
\end{enumH}
\end{definition}

In the case of the Allen--Cahn equation with mollified noise $\xi^\eps:=\xi *
\varrho_\eps$, where $\varrho(t,x)$ is a mollifier and
$\varrho_\eps(t,x)=\eps^{-(d+2)}\varrho(t/\eps^2,x/\eps)$, a canonical way of
building a model $Z^\eps=(\Pi^\eps,\Gamma^\eps)$ proceeds as follows.
First, we can just define 
\begin{align}
\nonumber
(\Pi^\eps_z\Xi)(\bar z) &= \xi^\eps(\bar z)\;, \\
\nonumber
(\Pi^\eps_z X^k)(\bar z) &= (\bar z-z)^k\;, \\
\label{eq:Pi_x} 
(\Pi^\eps_z \tau\bar\tau)(\bar z) &= (\Pi^\eps_z \tau)(\bar z)(\Pi^\eps_z
\bar\tau)(\bar z)
\qquad\forall\tau,\bar\tau\in T\;. 
\end{align}
To also include the integration map in the fixed point problem let $G=G(t,x)$ be the heat kernel 
defined as the fundamental solution associated to $\partial_t u = \Delta_x u$. We split the heat
kernel $G$ as 
\begin{equation}
 \label{eq:splitG}
 G =R + K\;,
\end{equation}
where $K:(\R^{d+1}\setminus\{0\})\rightarrow \R$ is a singular part satisfying specific algebraic properties, while
$R:\R^{d+1}\rightarrow \R$ is a smooth part. The properties of $K$ are: 
\begin{itemizz}
\item 	$K$ is supported in $\set{\abs{x}^2 + \abs{t} \leqs 1}$ where 
$\abs{x}=\sum_{j=1}^d\abs{x_j}$;
\item 	$K(t,x)=0$ for $t\leqs0$ and $K(t,-x)=K(t,x)$ for all $(t,x)$;
\item we have	
\begin{equation}
 \label{eq:K01}
 K(t,x) = \frac{1}{\abs{4\pi t}^{d/2}} \e^{-\norm{x}^2/4t}
 \qquad
 \text{for $\abs{x}^2+\abs{t} \leqs \frac12$}
\end{equation}
and $K(t,x)$ is smooth for $\abs{x}^2+\abs{t} > \frac12$;
\item furthermore, one has vanishing moments
\begin{equation}
 \label{eq:K02}
 \int_{\R^{d+1}} K(t,x)P(t,x) \6x\6t = 0
\end{equation}
for all polynomials $P$ of parabolic degree less or equal some
fixed $\zeta\geqs2$.
\end{itemizz}

Lemma~5.5 in~\cite{Hairer2014} shows that such a splitting indeed exists and 
also the vanishing moments condition holds for $K$; both parts $K,R$ then
satisfy a number of derivative bounds. Furthermore, it follows from Lemma~5.5
in~\cite{Hairer2014} that there exists a decomposition
\begin{equation}
\label{eq:Kdecomp}
K(z-z')=\sum_{n\geqs 0} K_n(z-z')
\end{equation}
where $n\in\N_0$ and suitable derivative bounds and vanishing moment conditions hold for 
each of the kernels $K_n$. With this construction, the abstract integration 
map is represented by the Taylor-series-like expression
\begin{equation}
 \label{eq:Pi_xI} 
(\Pi^\eps_z \cI\tau)(\bar z) = \int K(\bar z-\bar z')(\Pi^\eps_z \tau)(\6\bar
z') 
+ \sum_\ell \frac{(\bar z-z)^\ell}{\ell!} \pscal{f^\eps_z}{\cJ_\ell\tau}\;,
\end{equation} 
where the linear forms $f^\eps_z\in \mathcal{G}\subset\cH_+^\star$
are constructed as follows:
\begin{align}
\nonumber
\pscal{f^\eps_z}{\unit} &= 1\;, \\
\nonumber
\pscal{f^\eps_z}{X_i} &= -z_i\;, \\
\pscal{f^\eps_z}{\tau\bar\tau} &= \pscal{f^\eps_z}{\tau}
\pscal{f^\eps_z}{\bar\tau} 
\nonumber
\qquad \forall\tau,\bar\tau\in T\;, \\
\pscal{f^\eps_z}{\cJ_\ell\tau} &= - \int D^\ell K(z-\bar z)
(\Pi^\eps_z\tau)(\6\bar
z)\;.
\label{eq:f_x} 
\end{align}
Note that in the more general kernel case $K=K(z,\bar z)$ one has to replace
$D^\ell$ by $D^\ell_1$ and the subscript indicates derivative with respect to
the first argument. The second term in the definition~\eqref{eq:Pi_xI} of
$\Pi^\eps_z \cI\tau$ may seem somewhat strange, and will be motivated in the
next section. Note that it should really be interpreted by applying both sides
to test functions, i.e., for all smooth compactly supported functions $\psi$ and
all $\tau \in T_\alpha$ one requires
\begin{equation}
\pscal{\Pi^\eps_z \cI\tau}{\psi} =\sum_{n\geqs 0}\int_{\R^{d+1}}\psi(z') 
\pscal{\Pi^\eps_z \tau}{K_{n;zz'}^\alpha}\6 z'  
\end{equation}
where $K_{n;zz'}^\alpha$ is defined via
\begin{equation}
K_{n;zz'}^\alpha(z'')=K_n(z'-z'')-\sum_{\abss{k}<\alpha+2}
\frac{(z'-z)^k}{k!}D^k K_n(z-z'')\;.
\end{equation}
The group elements $\Gamma^\eps_{z\bar z}$ are then simply defined by 
\begin{equation}
 \label{eq:Gamma_F}
 \Gamma^\eps_{z\bar z} = (F^\eps_z)^{-1} F^\eps_{\bar z}
 \qquad 
 \text{where $F^\eps_z = \Gamma_{f^\eps_z}$}
\end{equation} 
(recall that $F^\eps_z$ is invertible because $f^\eps_z$ is grouplike). In this
way, the algebraic property (M1) of Definition~\ref{def:model} is automatically
satisfied. Regarding the algebraic property (M2), we can note the following
relations. Using Sweedler's notation~\cite{Sweedler69}, we write $\Delta\tau =
\tau^{(1)}\otimes\tau^{(2)}$, although $\Delta\tau$ is usually a sum of such
terms. Let further $\gamma^\eps_{z\bar z}$ denote the element of the structure
group such that $\Gamma^\eps_{z\bar z} = \Gamma^\eps_{\gamma^\eps_{z\bar z}}$.
Then we have by~\eqref{eq:def_g}
\begin{equation}
 \label{eq:Gamma_F01}
 \Pi^\eps_z \Gamma^\eps_{z\bar z}\tau = \Pi^\eps_z (\Id\otimes\gamma^\eps_{z\bar
z})\Delta\tau
 = \Pi^\eps_z\tau^{(1)} \pscal{\gamma^\eps_{z\bar z}}{\tau^{(2)}}\;.
\end{equation} 
Property (M2) thus amounts to the relation
\begin{equation}
 \label{eq:Gamma_F02}
 \Pi^\eps_{\bar z}\tau = \Pi^\eps_z\tau^{(1)} \pscal{\gamma^\eps_{z\bar
z}}{\tau^{(2)}}\;,
\end{equation} 
which provides some intuition for the meaning of $\Delta$. 
It is easy to check that this relation holds for elements $\tau$ of the
polynomial regularity structure. In the general case, the fact that
$Z^\eps=(\Pi^\eps,\Gamma^\eps)$ is indeed a model is a nontrivial
fact, proved in~\cite[Prop.~8.27]{Hairer2014}. 

\begin{remark}
\label{rem:model_canonical}
It is very important to realise that the canonical model just built is not the
only possible model for a given regularity structure. This freedom in the
choice of model will be used when introducing the renormalisation procedure.
All models, however, will share many properties with the canonical model. The
only rule that will be modified is the product rule $\Pi^\eps_z(\tau\bar\tau) =
\Pi^\eps_z(\tau)\Pi^\eps_z(\bar \tau)$. 
\end{remark}

We can now introduce the spaces $\cD^\gamma$, which play an analogous role as
the $\cC^\alpha_\fraks$ on the level of the regularity structure. 

\begin{definition}[{\cite[Def.~3.1]{Hairer2014}}]
\label{def:D_gamma}
Let $\gamma\in\R$. Given a model $Z=(\Pi,\Gamma)$, the space $\cD^\gamma =
\cD^\gamma(Z)$ consists of all functions $f:\R^{d+1}\to T_\gamma^-$ such
that for every compact set $\fraK\subset\R^{d+1}$ one has 
\begin{equation}
 \label{eq:D_gamma}
 \normDgamma{f}_{\gamma;\fraK} := 
 \sup_{z\in\fraK}~ \sup_{\beta < \gamma} \norm{f(z)}_\beta 
 + \sup_{\substack{z,\bar z\in\fraK\\ 
 \norm{z-\bar z}_{\fraks}\leqs1}}
 \sup_{\beta<\gamma}\frac{\norm{f(z)-\Gamma_{z\bar z}f(\bar z)}_\beta}
 {\norm{z-\bar z}^{\gamma-\beta}_{\fraks}} < \infty\;.
\end{equation}
\end{definition}

In the particular case of the polynomial regularity structure, it is again quite
straightforward to check that the requirement~\eqref{eq:D_gamma} is equivalent
to $f$ being the Taylor expansion of an element of the H\"older space
$\cC^\gamma_\fraks$. The spaces $\cD^\gamma$ depend on the model via $\Gamma$,
but not on $\Pi$, cf~\cite[Remark~3.4]{Hairer2014}. In order to compare
elements
of $\cD^\gamma$ for different models, it is useful to introduce
\begin{equation}
 \label{eq:D_gamma_seminorm}
 \seminormff{f}{\bar f}_{\gamma;\fraK}
 = \norm{f-\bar f}_{\gamma;\fraK} 
 + \sup_{\substack{z,\bar z\in\fraK \\ \norm{z-\bar z}_{\fraks}\leqs1}}
 \sup_{\beta < \gamma}
 \frac{\norm{f(z) - \bar f(z) - \Gamma_{z\bar z}f(\bar z) +
\Gammabar_{z\bar z}\bar f(\bar z)}_\beta}{\norm{z-\bar
z}_{\fraks}^{\gamma-\beta}}
\end{equation}
(which is in general \emph{not} a function of $f-\bar f$), where 
\begin{equation}
 \label{eq:D_gamma_normf}
 \norm{f-\bar f}_{\gamma;\fraK} 
 = \sup_{z\in\fraK}~ \sup_{\beta < \gamma}
 ~\norm{f(z) - \bar f(z)}_\beta\;.
\end{equation}

The central result allowing to link elements in $\cD^\gamma$ and in
$\cC^\alpha_\fraks$ is the \emph{reconstruction
theorem}~\cite[Thm~3.10]{Hairer2014}. It states that if $\alpha_*=\min A$ and
given $r>\abs{\alpha_*}$ there exists, for any $\gamma\in\R$, a 
continuous linear map $\cR:\cD^\gamma\to\cC^{\alpha_*}_{\fraks}$ such
that 
\begin{equation}
 \label{eq:R01}
 \bigabs{\pscal{\cR f-\Pi_zf(z)}{\cS^\delta_{\fraks,z}\eta}} 
 \leqs C \delta^\gamma \norm{\Pi}_{\gamma;\bar{\fraK}}
 \normDgamma{f}_{\gamma;\bar{\fraK}}
\end{equation} 
holds uniformly over all test functions $\eta\in\cB^r_{\fraks,0}$, all
$\delta\in(0,1]$, all $f\in\cD^\gamma$  and all $z\in\fraK$. The constant $C>0$
depends only on $\gamma$ and on the regularity structure, and the set
$\bar{\fraK}$ is the $1$-fattening of $\fraK$ (i.e., the points at distance at
most $1$ from $\fraK$). If $\gamma>0$, then~\eqref{eq:R01} defines $\cR$
uniquely. The reconstruction theorem also provides bounds on the dependence of
$\cR$ on the model $(\Pi,\Gamma)$. Heuristically,~\eqref{eq:R01} states that
$\cR f$ locally looks like $\Pi_zf(z)$ near any point $z\in\R^{d+1}$, up to
terms of order $\gamma$. 

%%%%%%%%%%%%%%%%%%%%%%%%%%%%%%%%%%%%%%%%%%%%%%%%%%%%%%%%%%%%%%%%%%%%%%%%%%%%%%

\subsection{Lifting the convolution maps}
\label{ssec_convolution}

We can now lift the operation of convolution with the heat kernel to the space
$\cD^\gamma(Z)$. The requirement~\eqref{eq:K02} is for compatibility with
the condition that $\cI$ should define an abstract integration map of order $2$
in the sense of~\cite[Def.~5.7]{Hairer2014}, namely  
\begin{itemizz}
 \item 	$\cI: T_\alpha \to T_{\alpha+2}$ for every $\alpha\in A$;
 \item 	$\cI\tau = 0$ for every $\tau$ in $\overline{T}$;
 \item 	$\cI\,\Gamma\tau - \Gamma\cI\tau\in\overline{T}$ for every
$\tau\in T$ and every $\Gamma\in \cG$. 
\end{itemizz}

The central result is given by the so-called \emph{multilevel Schauder
estimates}~\cite[Thm~5.12]{Hairer2014}, which state in particular that for all
$\gamma\in\R$ such that $\gamma+2\not\in\N$, there exists a map $\cK_\gamma:
\cD^\gamma \to \cD^{\gamma+2}$ such that 
\begin{equation}
 \label{eq:K_gamma1}
 \cR\cK_\gamma f = K*\cR f
\end{equation} 
holds for all $f\in\cD^\gamma$. In other words, the following diagram commutes:

\begin{center}
\begin{tikzpicture}[>=stealth']
\matrix (m) [matrix of math nodes, row sep=3em,
column sep=3.5em, text height=1.75ex, text depth=0.5ex]
{ \cD^\gamma & \cD^{\gamma+2} \\
 \cC^{\alpha_*}_{\fraks} & \cC^{\alpha_*}_{\fraks} \\};
\path[->]
(m-1-1) edge node[auto] {$\cK_\gamma$} (m-1-2)
(m-1-1) edge node[left] {$\cR$} (m-2-1)
(m-1-2) edge node[auto] {$\cR$} (m-2-2)
(m-2-1) edge node[below] {$K*$} (m-2-2);
\end{tikzpicture}
\end{center}

The map $\cK_\gamma$ has the following expression. For any $f\in\cD^\gamma$, 
\begin{equation}
 \label{eq:K_gamma2}
 (\cK_\gamma f)(z) = \cI f(z) + \cJ(z)f(z) + (\cN_\gamma f)(z)\;,
\end{equation} 
where for each $\tau\in T_\alpha$, 
\begin{align}
\label{eq:K_gamma3}  
\cJ(z)\tau &= \sum_{\abss{k} < \alpha+2}
\frac{X^k}{k!} \int_{\R^{d+1}} D^k K(z-\bar z)(\Pi_z\tau)(\6\bar z)\;, \\
\label{eq:K_gamma4}  
(\cN_\gamma f)(z) &= \sum_{\abss{k} < \gamma+2}
\frac{X^k}{k!} \int_{\R^{d+1}} D^k K(z-\bar z)(\cR f - \Pi_z f(z))(\6\bar z)\;.
\end{align}
Note that the last two operators have values in $\overline T$, the polynomial
part of the regularity structure, and that the only nonlocal operator is
$\cN_\gamma$. 

The r\^ole of the operators $\cJ$ and $\cN_\gamma$ is to ensure that
$\cK_\gamma f$ has the properties required to belong to $\cD^{\gamma+2}$, which
would not be the case if one simply sets $\cK_\gamma f = \cI f$. The maps $\cJ$
are related to the coefficients $\cJ_{k_j}$ appearing in~\eqref{eq:rs09}; in
fact, the $\cJ_{k_j}\tau_j$ play the r\^ole of placeholders for the
$\cJ(z)\tau$. 

As for the smooth part $R$ of the heat kernel, it can be lifted as
in~\cite[(7.7)]{Hairer2014}. Namely, with a smooth kernel $R$ we associate
maps $R_\gamma: \cC^{\alpha_*}_{\fraks} \to \cD^\gamma$ given by 
\begin{equation}
 \label{eq:R_gamma}
 (R_\gamma \upsilon)(z) 
 = \sum_{\abss{k}<\gamma} \frac{X^k}{k!} 
 \int_{\R^{d+1}} D^k R(z-\bar z) \upsilon(\bar z)\6\bar z
 := \sum_{\abss{k}<\gamma} \frac{X^k}{k!} 
 \pscal{\upsilon}{D^k R(z-\cdot)}\;.
\end{equation} 
It follows from~\cite[Prop~3.28]{Hairer2014} that for
$\upsilon\in\cC^{\alpha_*}_{\fraks}$ one has 
\[
 (\cR R_\gamma \upsilon)(z) = \pscal{\upsilon}{R(z-\cdot)}\;,
\]
and thus 
\begin{equation}
 \label{eq:R_gamma2}
 \cR R_\gamma \cR f = R * \cR f\;.
\end{equation} 
In other words, the following diagram commutes:

\begin{center}
\begin{tikzpicture}[>=stealth']
\matrix (m) [matrix of math nodes, row sep=3em,
column sep=3.5em, text height=1.75ex, text depth=0.5ex]
{ \cD^\gamma & \cD^{\gamma} \\
 \cC^{\alpha_*}_{\fraks} & \cC^{\alpha_*}_{\fraks} \\};
\path[->]
(m-1-1) edge node[auto] {$R_\gamma\cR$} (m-1-2)
(m-1-1) edge node[left] {$\cR$} (m-2-1)
(m-1-2) edge node[auto] {$\cR$} (m-2-2)
(m-2-1) edge node[below] {$R*$} (m-2-2)
(m-2-1) edge node[above left] {$R_\gamma$} (m-1-2);
\end{tikzpicture}
\end{center}

%%%%%%%%%%%%%%%%%%%%%%%%%%%%%%%%%%%%%%%%%%%%%%%%%%%%%%%%%%%%%%%%%%%%%%%%%%%%%%

Assume that for an appropriate choice of $\gamma$, one can find an element
$U\in\cD^\gamma(Z^\eps)$ satisfying the fixed-point equation 
\begin{equation}
 \label{eq:AC-fpe01}
 U = (\cK_{\bar\gamma} + R_\gamma\cR) \Rplus (\Xi + U - U^3) + Gu_0\;,
\end{equation} 
for some $\bar\gamma\geqs\gamma-2$, where $\Rplus(t,x)=\indexfct{t>0}$ and
$Gu_0$ denotes a suitable lift of the convolution in space of heat kernel
and initial condition.
Applying the reconstruction operator $\cR$ to both sides of this equation, one
can show that $u=\cR U$ satisfies 
\begin{equation}
 \label{eq:AC-fpe02}
 u = (K+R)*(\Rplus[u-u^3+\xi^\eps]) + Gu_0\;,
\end{equation} 
which is equivalent to~\eqref{eq:AC-fpe}; in \eqref{eq:AC-fpe02} we also use
$Gu_0$ as a notation for the usual convolution in space of heat kernel and
initial condition. This is basically the strategy implemented in Theorem~7.8 and
Section~9.4 of~\cite{Hairer2014}, except that one has to deal with two
additional technical difficulties. The first one is that due to the singular
behaviour of the heat kernel as time goes to $0$, the definition of the space
$\cD^\gamma$ has to be modified. We will apply this modification to our case in
Section~\ref{ssec_D_gamma_eta}. The second difficulty is that although a fixed
point exists for every mollification parameter $\eps>0$, it cannot converge as
$\eps\to0$. This is why a renormalisation procedure is needed, which we will
adapt to our case in Section~\ref{sec_renorm}. 

%%%%%%%%%%%%%%%%%%%%%%%%%%%%%%%%%%%%%%%%%%%%%%%%%%%%%%%%%%%%%%%%%%%%%%%%%%%%%%

\section{Extension of the regularity structure}
\label{sec_extend}

Our aim is now to extend the regularity structure built for the Allen--Cahn
equation, in order to allow to represent a family of coupled SPDEs--ODEs of the
form 
\begin{align}
\nonumber
\partial_t u &= \Delta_x u + F(u,v) + \xi^\eps\;, \\
\partial_t v &=  u A_1 + A_2 v\;,
\label{eq:fhn01} 
\end{align} 
where $F(u,v)$ is a cubic polynomial, and $A_1\in \R^n$, $A_2\in \R^{n\times n}$
are either scalars as for the classical FitzHugh--Nagumo case with $n=1$, or a
vector and a matrix if $v$ has multiple components. Duhamel's formula allows us
to represent (mild) solutions of~\eqref{eq:fhn01} as 
\begin{align}
\nonumber
u_t &= \int_0^t S(t-s) \bigbrak{\xi^\eps_s + F(u_s,v_s)} \6s + S(t)u_0\;, \\
v_t &= \int_0^t u_s Q(t-s)  \6s + \e^{tA_2}v_0\;, 
\label{eq:fhn02} 
\end{align}
where $Q(t) := \e^{t A_2}A_1$. We thus have to lift to the regularity
structure the operation of time-integration with respect to $Q$, which has no
smoothing effect in space. 

In the case where $v$ has values in $\R^n$ with $n>1$, the kernel $Q(t)$ is in
fact a vector of dimension $n$. In what follows, we will mainly deal with the
case of scalar $Q$, since the generalisation to the vectorial case is rather
straightforward. 

%%%%%%%%%%%%%%%%%%%%%%%%%%%%%%%%%%%%%%%%%%%%%%%%%%%%%%%%%%%%%%%%%%%%%%%%%%%%%%

\subsection{Extension theorem}
\label{ssec_ext}

Let us fix a finite time horizon $T$. Then we can always assume that the
kernel $Q$ satisfies the following properties:
\begin{itemizz}
\item 	$Q(t)$ is supported on $[0,2T]$ and smooth for $t>0$;
\item 	$Q(t) = \e^{tA_2}A_1$ for all $t\in[0,T]$.
\end{itemizz}
The reason why this is allowed is that we will be interested in showing
existence of solutions on a sufficiently small interval $[0,T]$, so that the
behaviour of $Q$ outside this interval is not going to matter. Also, since
$Q$ is bounded, it will not be necessary to decompose it as a sum of $Q_n$
concentrated in sets of radius $2^{-n}$, as in the case of the heat kernel
in~\cite[Section~5]{Hairer2014}, cf~\eqref{eq:Kdecomp}.

In order to be able to represent the time-integration map, we will have to
extend our regularity structure and the associated model. We do this by adding
to $T$ new symbols denoted $\cE(\tau)$, $\tau\in T\setminus\overline{T}$, which
we represent by an open blue dot. Thus for instance we write $\cE(\cI(\Xi)) =
\cE(\RSI) = \RSoI$. In practice, we will only need to apply $\cE$ to elements
$\tau$ of homogeneity $\abss{\tau} \in (-2,0)$. We thus set $V =
\setsuch{\tau\in T}{-2<\abss{\tau}<0}$ ($V$ is called a \emph{sector} of $T$). 
If $\tau\not\in V$, we simply set $\cE(\tau)=0$. We also postulate that for
$\tau\in V$, $\cE(\tau)$ has the same homogeneity as $\tau$, and we choose the
norm on the vector space generated by the new symbols in such a way that 
\begin{equation}
 \label{eq:E_norm}
 \norm{\cE(\tau)}_\alpha = \norm{\tau}_\alpha
 \qquad 
 \forall \alpha\in V\;.
\end{equation}
The operator $\cE$ defined in this way is an abstract integration map of order
$0$ on $V$ in the sense of~\cite[Def.~5.7]{Hairer2014}, i.e.,
\begin{itemizz}
 \item 	$\cE: V\cap T_\alpha \to T_\alpha$ for every $\alpha\in A$;
 \item 	$\cE\tau = 0$ for every $\tau \in V\cap\overline{T}$;
 \item 	$\cE\Gamma\tau - \Gamma\cE\tau\in\overline{T}$ for every
$\tau\in V$ and every $\Gamma\in \cG$. 
\end{itemizz} 
The first two properties are obvious by definition of $\cE\tau$, in particular
we have $V\cap\overline{T} = \emptyset$. Regarding the
third property, we first extend the structure group by setting 
\begin{equation}
 \label{eq:ext05}
 \Delta(\cE\tau) = (\cE\otimes\Id)\Delta\tau\;.
\end{equation} 
The third required property of $\cE$ is indeed satisfied since 
\begin{equation}
 \label{eq:GammaE}
 \Gamma_g\cE\tau = (\Id\otimes g)\Delta(\cE\tau) 
 = (\cE\otimes g)\Delta\tau
 = \cE\Gamma_g\tau\;,
\end{equation} 
so that in fact $\Gamma_g\cE\tau - \cE\Gamma_g\tau = 0$ for all $\tau\in T$ and
all $g\in\cG$. 

Since $\cE$ is not regularity-increasing, there can be in principle infinitely
many symbols of given homogeneity. The trick, however, will be to build the
fixed-point map in such a way that only finitely many new symbols are needed. 
In practice, it will turn out that the only required new symbols are $\RSoI$,
and those obtained by applying $\cI$ to existing symbols and/or multiplying
them. However, since it is of independent interest, we are going to describe the
extension procedure in a more abstract, inductive way. Given a subset $W\subset
V$, on which the canonical model $Z^\eps=(\Pi^\eps,\Gamma^\eps)$ is defined, we
want to extend the model to a larger set
$\widehat{W}=W\cup\setsuch{\cE(\tau)}{\tau\in W}$. We can then apply the usual
extension theorem~\cite[Thm~5.14]{Hairer2014} to extend the model to
$\widehat{W}\cup\setsuch{\cI(\tau)}{\tau\in \widehat{W}}$, and so on as often as
needed. 

The inductive step from $W$ to $\widehat{W}$ goes as follows. Assume that
$Z^\eps$ already satisfies Definition~\ref{def:model} on a regularity structure
$(A,W,\cG)$. We define the extended model by setting  
\begin{equation}
\label{eq:E01} 
 (\Pi^\eps_{t,x}\cE\tau)(\bar t,\bar x) = 
  \int_{\bar t-2T}^{\bar t} Q(\bar t-s) (\Pi^\eps_{t,x}\tau)(s,\bar x)\6s
\end{equation} 
for all $\tau\in W$. 
Writing as before $F^\eps_z = \Gamma_{f^\eps_z}$, the new group elements are
defined by setting
\begin{equation}
 \label{eq:E03}
 \Gamma^\eps_{zz'}(\cE\tau) = (F^\eps_z)^{-1} F^\eps_{z'}(\cE\tau)
\end{equation} 
for all $z,z'\in\R^{d+1}$.

\begin{remark}
\label{rem:JQ} 
An alternative would be to define $\cE$ on the sector $V = \setsuch{\tau\in
T}{-2<\abss{\tau}<2}$ by the expression
\begin{equation}
\label{eq:E01a} 
 (\Pi^\eps_{t,x}\cE\tau)(\bar t,\bar x) = 
 \begin{cases}
 \displaystyle\vrule height 5pt depth 18pt width 0pt
  \int_{\bar t-2T}^{\bar t} Q(\bar t-s) (\Pi^\eps_{t,x}\tau)(s,\bar x)\6s 
  &\text{if $\abss{\tau}<0$\;,} \\
  \displaystyle\vrule height 18pt depth 5pt width 0pt
  \int_{\bar t-2T}^{\bar t} Q(\bar t-s) (\Pi^\eps_{t,x}\tau)(s,\bar x)\6s
  + \pscal{f^\eps_{t,x}}{\cJ^Q\tau}\quad
  &\text{if $0\leqs\abss{\tau}<2$\;,}
 \end{cases}
\end{equation} 
where 
\begin{equation}
\label{eq:E02a} 
 \pscal{f^\eps_{t,x}}{\cJ^Q\tau} = 
 - \int_{t-2T}^t Q(t-s) (\Pi^\eps_{t,x}\tau)(s,x)\6s\;.
\end{equation} 
The new symbol $\cJ^Q$ is needed to ensure the property
$(\Pi^\eps_{t,x}\cE\tau)(t,x) = 0$ when $\abs{\tau}_\fraks\geqs0$, which is
necessary when lifting the fixed-point equation. 
While this defines an extended model with the required properties, the fact
that $Q$ acts by convolution in time only limits the regularity of its lift to
the regularity structure (cf.\ Remark~\ref{rem:JQhat} below). We will see in the
next subsection why it is sufficient to introduce new symbols $\cE\tau$ only
when $\abs{\tau}_\fraks < 0$.~$\lozenge$ 
\end{remark}

\begin{remark}
\label{rem:E_vectorial} 
In cases where $v(t,x)$ takes values in $\R^n$ and $Q(t)\in\R^n$, we should 
in fact introduce $n$ commuting symbols $\cE_1,\dots,\cE_n$, and define the
extended model by the relations 
\begin{equation}
\label{eq:E04a} 
 (\Pi^\eps_{t,x}\cE_i\tau)(\bar t,\bar x) = 
  \int_{\bar t-2T}^{\bar t} Q_i(\bar t-s) (\Pi^\eps_{t,x}\tau)(s,\bar x)\6s 
\end{equation} 
for $i=1,\dots n$. The results that follow remain true when $\cE$ and $Q$ are
replaced by $\cE_i$ and $Q_i$.~$\lozenge$ 
\end{remark}

We now have to check that $(\Pi^\eps,\Gamma^\eps)$ indeed defines a model on
$\widehat{W}$. The following lemma contains a technical estimate preparing
the proof of that fact. 

\begin{lemma}
\label{lem:E_extension}
Assume that for any compact set $\fraK\subset\R^{d+1}$ there exists a constant
$C_\fraK$ such that 
\begin{equation}
\label{eq:E_05} 
 \abs{(\Pi^\eps_z\tau)(\bar z)} \leqs C_\fraK \norm{z-\bar
z}_\fraks^{\abss{\tau}} \norm{\tau}
\end{equation}
for all $\tau\in W$ and all $z,\bar z\in\fraK$. Then there exists a constant
$C_0$, depending only on $Q$, such that 
\begin{equation}
\label{eq:E_06} 
 \abs{(\Pi^\eps_z\cE\tau)(\bar z)} \leqs C_0C_{\fraKbar} \norm{z-\bar
z}_\fraks^{\abss{\tau}} \norm{\tau}
\end{equation}
holds for all $\tau\in W$ and all $z,\bar z\in\fraK$, where $\fraKbar =
\setsuch{(t,x)}{\exists \bar t\in\R\colon \abs{\bar t-t}\leqs 2T, (\bar
t,x)\in\fraK}$ (for brevity we shall call $\fraKbar$ the $2T$-fattening of
$\fraK$, although strictly speaking it is only a fattening in the time
direction). 
\end{lemma}
\begin{proof}
Using the definition~\eqref{eq:E01} of $\Pi^\eps_{t,x}\cE\tau$ and the
assumption~\eqref{eq:E_05}, we obtain 
\begin{align*}
 \abs{(\Pi^\eps_{t,x}\cE\tau)(\bar t,\bar x)}
 &\leqs C_{\fraKbar}\int_{\bar t-2T}^{\bar t} \abs{Q(\bar t-s)}
\norm{(t,x)-(s,\bar x)}^\alpha_\fraks \6s \, \norm{\tau} \\
 &\leqs C_{\fraKbar}\int_0^{2T} \abs{Q(\bar s)}\abs{t-\bar t+\bar s}^{\alpha/2}
\6\bar s\, \norm{\tau}
 + C_{\fraKbar} \int_0^{2T} \abs{Q(\bar s)}\,\6\bar s\,
\abs{x-\bar x}^\alpha \norm{\tau} \;,
\end{align*}
where $\alpha=\abss{\tau}$.
The required bound thus follows if we can show that 
\[
 \int_0^{2T} \abs{t-\bar t+\bar s}^{\alpha/2} \6\bar s 
 \lesssim \abs{t-\bar t\,}^{\alpha/2}
\]
holds for $\abs{t-\bar t\,}\lesssim 1$. By treating separately the cases $t>\bar
t$ and $t<\bar t$, one sees that the left-hand side is always bounded above by
a constant times $(\abs{t-\bar t\,}+2T)^{1+\alpha/2}$, which is bounded above
for $\alpha>-2$. Since on the other hand, the right-hand side is bounded below
by a positive constant for $\alpha<0$, the result follows. 
\end{proof}

Note that~\cite[Prop.~8.27]{Hairer2014} shows in particular that the
assumption~\eqref{eq:E_05} is satisfied by any canonical model for mollified
noise $\Pi^\eps$ built as in Section~\ref{ssec_model}. 

We can now state the main result of this subsection, which is an adaptation of
the extension theorem~\cite[Thm.~5.14]{Hairer2014} and of
\cite[Prop.~8.27]{Hairer2014} to our degenerate situation. 

\begin{prop}[Extension theorem for $\cE$]
\label{prop_extension_E} 
Let $Z^\eps=(\Pi^\eps,\Gamma^\eps)$ be a model for the regularity structure
$(A,W,\cG)$, where $W\subset V$, and such that $\bar z\mapsto\Pi^\eps_z\tau(\bar
z)$ is continuous and satisfies~\eqref{eq:E_05} for any $\eps>0$. Let $\widehat
W=W\cup\setsuch{\cE\tau}{\tau\in W}$. Then $\Zhat^\eps =
(\Pihat^\eps,\Gammahat^\eps)$ obtained by extending $Z^\eps$ in the above way is
a model for $(A,\widehat W,\cG)$, which satisfies~\eqref{eq:E_06} for any
$\eps>0$. Furthermore, 
\begin{equation}
 \label{eq:E_08}
 \norm{\Gammahat^\eps}_{\gamma;\fraK} = \norm{\Gamma^\eps}_{\gamma;\fraK}
\end{equation} 
holds for any $\gamma\in\R$ and any compact $\fraK\subset\R^{d+1}$. 
\end{prop}
\begin{proof}
We have to prove that $\Zhat^\eps$ satisfies the assumptions (M1)--(M4) of
Definition~\ref{def:model}. Property (M1) is automatically satisfied owing 
to~\eqref{eq:E03}. Regarding property (M2), it is known
\cite[Sec.~8.3]{Hairer2014} that there exists a linear map
$\bPi^\eps:T\rightarrow \cS'(\R^{d+1})$ such that
$\bPi^\eps\tau=\Pi^\eps_z(F^\eps_z)^{-1}\tau$ is independent of $z$ for any
$\tau\in T$. Property (M2) holds if we can find for any $\tau\in W$ a
distribution $\bPi^\eps(\cE\tau)$, independent of $(t,x)$, such that 
\[
 \Pi^\eps_{t,x}(\cE\tau) = \bPi^\eps F^\eps_{t,x}(\cE\tau) 
 \qquad \forall z=(t,x)\in \R^{d+1}\;.
\]
To achieve this, we simply define $\bPi^\eps$ on the extended structure by  
\[
 (\bPi^\eps\cE\tau)(\bar t,\bar x) 
 = \int_{\bar t-2T}^{\bar t} Q(\bar t-s) (\bPi^\eps\tau)(s,\bar x)\6s\;. 
\]
Indeed, we have (writing as usual $\Delta\tau=\tau^{(1)}\otimes\tau^{(2)}$) 
\begin{align*}
(\bPi^\eps F^\eps_z\cE\tau)(\bar t,\bar x) 
&= \bPi^\eps (\cE\otimes f^\eps_z)\Delta\tau
(\bar t,\bar x) \\
&= (\bPi^\eps\cE\tau^{(1)})(\bar t,\bar x) \pscal{f^\eps_z}{\tau^{(2)}} \\
&= \int_{\bar t-2T}^{\bar t} Q(\bar t-s) 
(\bPi^\eps\tau^{(1)})(s,\bar x)\6s \, \pscal{f^\eps_z}{\tau^{(2)}} \\
&= \int_{\bar t-2T}^{\bar t} Q(\bar t-s) (\Pi^\eps_z\tau)(s,\bar x)\6s \\ 
&= (\Pi^\eps_z\cE\tau)(\bar t,\bar x)\;.
\end{align*}
To obtain the fourth line, we have used the fact that $\Pi^\eps_z = \bPi^\eps
F^\eps_z$ holds for the original model, and thus 
\[
 (\Pi^\eps_z\tau)(s,\bar x) 
 = (\bPi^\eps_zF^\eps_z\tau)(s,\bar x)
 = (\bPi^\eps_z(\Id\otimes f^\eps_z)\Delta\tau)(s,\bar x)
 = (\bPi^\eps\tau^{(1)})(s,\bar x) \pscal{f^\eps_z}{\tau^{(2)}}\;.
\]
Property (M3) is a direct consequence of~\eqref{eq:E_06}. Indeed, this bound
implies that for any localised scaled test function $\eta$ of integral $1$, 
\begin{align*}
\bigabs{\pscal{\Pi^\eps_{t,x}\cE\tau}{\eta^\delta_{t,x}}}
&= \biggabs{\iint (\Pi^\eps_{t,x}\cE\tau)(\bar t,\bar x)
\frac{1}{\delta^{d+2}}
\eta\biggpar{\frac{\bar t-t}{\delta^2},\frac{\bar x-x}{\delta}}
\6\bar t\6\bar x} \\
&\leqs \iint \bigabs{(\Pi^\eps_{t,x}\cE\tau)(t+\delta^2s,x+\delta y)}
\bigabs{\eta(s,y)}\6s\6y \\
&\lesssim \delta^{\abss{\tau}} \norm{\tau}
\end{align*}
as required by~\eqref{eq:defmod1}. 
Finally, in order to prove Property (M4), we recall from~\eqref{eq:GammaE} that 
\[
\Gamma^\eps_{z\bar z}\cE\tau 
= \cE\Gamma^\eps_{z\bar z}\tau\;.
\]
If $\abs{\tau}_\fraks\geqs 0$, there is nothing to prove. If $\beta< \alpha =
\abs{\tau}_\fraks < 0$, then 
\[
 \norm{\Gamma^\eps_{z\bar z}\cE\tau}_\beta 
 = \norm{\cE\Gamma^\eps_{z\bar z}\tau}_\beta 
 = \norm{\Gamma^\eps_{z\bar z}\tau}_\beta
 \leqs \norm{\Gamma^\eps}_{\gamma;\fraK}\norm{\tau}\norm{z-\bar
z}^{\alpha-\beta}_\fraks
\]
for all $z, \bar z\in\fraK$. This completes the proof that $\Zhat^\eps$ is a
model, and also proves~\eqref{eq:E_08}. 
\end{proof}

%%%%%%%%%%%%%%%%%%%%%%%%%%%%%%%%%%%%%%%%%%%%%%%%%%%%%%%%%%%%%%%%%%%%%%%%%%%%%%

\subsection{Multilevel Schauder estimates}
\label{ssec_Schauder}

We now would like to construct an operator $\cK^Q_\gamma: \cD^\gamma \to
\cD^\gamma$ which lifts the operation of integration against $Q$ to the
regularity structure. Formally, this means that we should have 
\begin{equation}
 \label{eq:fix_t03}
 \cR\cK^Q_\gamma f = K_v * \cR f\;,
 \qquad
 K_v(t,x,\bar t,\bar x) = Q(t-\bar t\,)\delta(x-\bar x) 
\end{equation} 
for all $f\in\cD^\gamma$. In fact, this should really be interpreted as 
\begin{equation}
 \label{eq:fix_t03a}
 (\cR\cK^Q_\gamma f)(t,x) = \int_{t-2T}^t Q(t-\bar t\,) (\cR f)(\bar t,x)
\6\bar t\;, 
\end{equation} 
or in terms of test functions  
\begin{equation}
 \label{eq:fix_t03b}
 \pscal{\cR\cK^Q_\gamma f}{\psi} = \pscal{\cR f}{\hat\psi}\;,
\end{equation} 
where $\hat\psi$ is defined by 
\begin{equation}
 \label{eq:ext04}
 \hat\psi(\bar t,y) := \int_\R Q(s-\bar t\,) \psi(s,y)\6s
= \int_{\bar t\,}^{\bar t+2T} Q(s-\bar t) \psi(s,y)\6s\;.
\end{equation} 
The problem with such a plan is that the kernel $K_v$ being singular in space,
we cannot apply an expansion as in~\eqref{eq:K_gamma3} and~\eqref{eq:K_gamma4}.
However, for the class
of equations we are interested in, we do not actually need to define
$\cK^Q_\gamma$ on \emph{all} of $\cD^\gamma$. It will be quite sufficient to
define it on a subset of $\cD^\gamma$, which is given by the functions whose
components with negative homogeneity do not depend on $(t,x)$. This motivates
the following definition.

\begin{definition}
\label{def_Dgamma0}
Let $Z=(\Pi,\Gamma)$ be a model. 
The space $\cD^\gamma_0(Z)=\cD^\gamma_0(\Gamma)$
is the space of functions $f \in \cD^\gamma(\Gamma)$ of the form 
\begin{equation}
 \label{eq:def_Dgamma0}
 f(z) = \sum_{\tau\in T\colon -2 < \abss{\tau} < 0} c_\tau \tau +
\sum_{\tau\in T\colon\abss{\tau}\geqs0} \hat c_\tau(z)\tau =: f_- +
f_+(z)\;,
\end{equation}
where the $c_\tau$ do not depend on $z$. 
\end{definition}

\begin{remark}
\label{rem:Dgamma0}
Not all constant functions belong to $\cD^\gamma_0$, as they still have to
satisfy the analytical bound \eqref{eq:D_gamma}, which includes the
requirement 
$\norm{f(z)-\Gamma_{z\bar z}f(\bar z)}_\beta = \Order{\norm{z-\bar z}^{\gamma-\beta}}$
as $z\rightarrow \bar z$ for all $\beta < \gamma$. In fact, since
$\Gamma_{z\bar z}\tau = (\Id\otimes\gamma_{z\bar z})\Delta\tau$, we see that a
sufficient condition for having $\Gamma_{z\bar z}\tau = \tau$ is 
\begin{equation}
 \label{eq:tau_tensor_unit}
 \Delta\tau = \tau\otimes\unit\;.
\end{equation}
Thus the sum defining $f_-$ should involve only terms $c_\tau \tau$ such that
$\Delta\tau = \tau\otimes\unit$. As can be seen in
Table~\ref{tab:FF_AllenCahn}, in our case this property is satisfied by the
symbols $\Xi$, $\RSI$, $\RSV$, $\RSW$, and will also hold e.g.\ for $\RSoI$, but
it does \emph{not} hold for $\RSWW$ for instance. Note that owing
to~\eqref{eq:def_g} and~\eqref{eq:Gamma_F02}, \eqref{eq:tau_tensor_unit}
implies $\Pi_{\bar z}\tau = \Pi_z\tau$ for all $z, \bar z\in\R^{d+1}$, i.e., the
model does not depend on the base point $z$ for these $\tau$.~$\lozenge$ 
\end{remark}

With this notation in place, we define the operator $\cK^Q_\gamma$ on
$\cD^\gamma_0$ by 
\begin{equation}
\label{eq:def_KQ} 
  (\cK^Q_\gamma f)(z) = \cE f_- + \int_{t-2T}^t Q(t-s)f_+(s,x)\6s\;.
\end{equation}
Note that this definition indeed requires $\cE$ to be defined only on symbols
$\tau$ of strictly negative homogeneity. 

\begin{remark}
\label{rem:JQhat}
An alternative definition of $\cK^Q_\gamma$, closer in spirit
to~\eqref{eq:K_gamma2}, would be to set 
\[
 (\cK^Q_\gamma f)(z) = \cE f(z) + \JQhat(z)f_+(z) + (\cN_\gamma^Qf_+)(z)\;,  
\]
where $\cE$ is defined for positive-homogeneous terms as in Remark~\ref{rem:JQ}
and 
\begin{align*}
\JQhat(t,x)\tau &= \int_{t-2T}^t
Q(t-s)(\Pi^{(\eps)}_{t,x}\tau)(s,x)\6s\,\unit\;, \\
(\cN_\gamma^Qf_+)(z) &= \int_{t-2T}^t Q(t-s)(\cR f_+
- \Pi^{(\eps)}_{t,x}f_+(t,x))(s,x)\6s\,\unit\;.
\end{align*}
While it is easily checked that $\cK^Q_\gamma$ defined in this way indeed
satisfies~\eqref{eq:fix_t03a}, the problem is that it typically does not map
$\cD^\gamma$ into itself, but only into some $\cD^{\gamma'}$ with $\gamma'<1$.
This is due to the fact that we cannot Taylor-expand the kernel $K_v$ in space,
preventing us from subtracting higher-order polynomial terms as
in~\eqref{eq:K_gamma4}.
\end{remark}

The following lemma on translation invariance of the canonical model
$Z^\eps=(\Pi^\eps,\Gamma^\eps)$ will allow us to prove that $\cK^Q_\gamma$ given
by~\eqref{eq:def_KQ} satisfies the required properties. We give its proof in
Appendix~\ref{app_proof_translation}.

\begin{lemma}
\label{lem_trans_inv}
Let $(\Pi^\eps,\Gamma^\eps)$ be the canonical model defined
by~\eqref{eq:Pi_x}, \eqref{eq:Pi_xI} and~\eqref{eq:E01}.
For all $h, z, \bar z\in\R^{d+1}$, all $\tau\in T$, all $\bar\tau\in\cF_+$ and
all $\eps>0$, one has
\begin{align}
\nonumber
\Pi^\eps_{z+h}\tau(\bar z+h) &= \Pi^\eps_z\tau(\bar z)\;, \\
\nonumber
\Gamma^\eps_{z+h,\bar z+h} \tau &= \Gamma^\eps_{z\bar z}\tau\;, \\
\pscal{\gamma^\eps_{z+h,\bar z+h}}{\bar \tau} &= \pscal{\gamma^\eps_{z\bar
z}}{\bar \tau}\;.
\label{eq:translation}
\end{align}
\end{lemma}

The following result shows that the multilevel Schauder estimates contained
in~\cite[Thm.~5.12]{Hairer2014} also hold in the case of $\cE$, in a similar
form. 

\begin{prop}[Multilevel Schauder estimates on $\cK^Q_\gamma$]
\label{prop:Schauder_E} 
If $\gamma>0$, then the operator $\cK^Q_\gamma$ maps $\cD^\gamma_0$ into
$\cD^\gamma_0$, and satisfies 
\begin{equation}
 \label{eq:norm_E}
 \normDgamma{\cK^Q_\gamma f}_{\gamma;\fraK}
 \leqs (1\vee\norm{Q}_{L^1})\normDgamma{f}_{\gamma;\fraKbar}\;,
\end{equation} 
where $\norm{Q}_{L^1} = \displaystyle\int_\R
\abs{Q(t)}\6t=\displaystyle\int_0^{2T} \abs{Q(t)}\6t$, and $\fraKbar$
is the $2T$-fattening of $\fraK$. The identity 
\begin{equation}
 \label{eq:REf_KRf}
  (\cR\cK^Q_\gamma f)(t,x) = \int_{t-2T}^t Q(t-s)\cR f(s,x)\6s
\end{equation} 
holds for all $f\in\cD^\gamma_0$. Furthermore, if $(\Pibar,\Gammabar)$ is a
second translation-invariant model satisfying~\eqref{eq:Pi_xI}
and~\eqref{eq:E01}, and $\overline{\cK}^Q_\gamma$ is the associated lift, then 
\begin{equation}
 \label{eq:Schauder_bounds}
 \seminormff{\cK^Q_\gamma f}{\overline{\cK}^Q_\gamma \bar f}_{\gamma;\fraK} 
 \leqs (1\vee\norm{Q}_{L^1})\seminormff{f}{\bar f}_{\gamma;\fraKbar}
\end{equation} 
holds for all $f\in\cD^\gamma_0(\Gamma)$ and $\bar
f\in\cD^\gamma_0(\Gammabar)$. 
\end{prop}
\begin{proof}
We start by showing that relation~\eqref{eq:REf_KRf} holds, assuming $f,
\cK^Q_\gamma f\in\cD^\gamma_0$. It follows from the reconstruction theorem and
its
corollary~\cite[Prop.~3.28]{Hairer2014} (see also~\eqref{eq:E_05}) that
\[
 \cR f_+(z) = \pscal{\unit}{f_+(z)} = \bigpar{\Pi^\eps_z f_+(z)}(z)
\]
for all $z\in\R^{d+1}$. Furthermore, let 
\[
 (\widehat\cR f_-)(\bar z) := \Pi^\eps_z f_-(\bar z) 
 =\sum_{-2<\abss{\tau}<0} c_\tau (\Pi^\eps_z\tau)(\bar z)
\]
(where the right-hand side actually does not depend on $z$, cf.\
Remark~\ref{rem:Dgamma0}). Then we see that $\widehat\cR f_-$ trivially
satisfies the reconstruction theorem, so that by uniqueness of the
reconstruction operator if $\gamma>0$ we have 
\[
 \cR f_- = \widehat\cR f_- = \Pi^\eps_z f_-
\]
whenever $f_-\in\cD^\gamma_0$. Thus in fact 
\[
\cR f(z) = \bigpar{\Pi^\eps_z f(z)}(z) 
\]
holds for all $z\in\R^{d+1}$. It follows by~\eqref{eq:GammaE} and the same
argument as the one yielding  $\cR f_- = \Pi_z f_-$ that 
\begin{align*}
(\cR\cE f_-)(t,x) 
&= (\Pi^\eps_{t,x}\cE f_-)(t,x) \\
&= \int_{t-2T}^t Q(t-s) (\Pi^\eps_{t,x}f_-)(s,x)\6s \\
&= \int_{t-2T}^t Q(t-s) \cR f_-(s,x)\6s\;,
\end{align*}
which implies that we indeed have  
\begin{align*}
 (\cR\cK^Q_\gamma f)(t,x) 
 &= (\cR\cE f_-)(t,x) + \int_{t-2T}^t Q(t-s)\cR f_+(s,x)\6x \\
 &= \int_{t-2T}^t Q(t-s)\cR f(s,x)\6x\;.
\end{align*}
Next we check that $\cK^Q_\gamma$ indeed maps $\cD^\gamma_0$ into itself. Since
$\cE$ preserves homogeneity, it is in fact sufficient to show that
$\cK^Q_\gamma$ maps $\cD^\gamma_0$ into $\cD^\gamma$. 
First we note that for any $\beta<\gamma$, 
\begin{align*}
\norm{\cK^Q_\gamma f(z)}_\beta 
&\leqs \norm{\cE f_-(z)}_\beta + \int_{t-2T}^t \abs{Q(t-s)}
\norm{f_+(s,x)}_\beta \6s \\
&\leqs \norm{f_-(z)}_\beta + \norm{Q}_{L^1} \sup_{s\in[t-2T,t]}
\norm{f_+(s,x)}_\beta \\
&\leqs (1 \vee \norm{Q}_{L^1}) \normDgamma{f}_{\gamma;\fraKbar}\;,
\end{align*}
(note that depending on the sign of $\beta$, either one or the other term in the
sums contributes, but never both). 
Regarding terms evaluated in different locations, note that \eqref{eq:GammaE}
implies that 
\begin{align*}
\norm{\cE f_- - \Gamma^\eps_{z\bar z}\cE f_-}_\beta 
&= \norm{\cE(f_- - \Gamma^\eps_{z\bar z} f_-)}_\beta \\
&= \norm{f_- - \Gamma^\eps_{z\bar z} f_-}_\beta \\
&\leqs \normDgamma{f_-}_{\gamma;\fraK} \norm{z-\bar z}^{\gamma-\beta}_\fraks\;.
\end{align*}
As for the positive-homogeneous part $g=\cK^Q_\gamma f_+$, using
Lemma~\ref{lem_trans_inv} we get 
\begin{align*}
g(t,x) 
- \Gamma^\eps_{t,x;\bar t,\bar x}g(\bar t,\bar x) 
&= \int_0^{2T} Q(s) \bigbrak{f_+(t-s,x) - \Gamma^\eps_{t,x;\bar t,\bar x}
f_+(\bar t-s,\bar x)}\6s \\
&= \int_0^{2T} Q(s) \bigbrak{f_+(t-s,x) - \Gamma^\eps_{t-s,x;\bar t-s,\bar x}
f_+(\bar t-s,\bar x)}\6s\;, 
\end{align*}
and thus
\begin{align*}
\norm{g(t,x) 
- \Gamma^\eps_{t,x;\bar t,\bar x}g(\bar t,\bar x)}_\beta
&\leqs \int_0^{2T} \abs{Q(s)} \norm{f_+(t-s,x) - \Gamma^\eps_{t-s,x;\bar
t-s,\bar x} f_+(\bar t-s,\bar x)}_\beta\6s\\
&\leqs \norm{Q}_{L^1} \normDgamma{f_+}_{\gamma;\fraKbar} \norm{(t,x)-(\bar
t,\bar x)}^{\gamma-\beta}_\fraks\;.
\end{align*}
Collecting the estimates obtained and using the definition~\eqref{eq:D_gamma} of
$\normDgamma{\cdot}$, we see that $\cK^Q_\gamma f\in\cD^\gamma$ and that it
satisfies~\eqref{eq:norm_E}. 

Consider finally the case where we have a second function $\bar
f\in\cD^\gamma_0(\Gammabar)$. Using again~\eqref{eq:GammaE}, it is immediate to
see that 
\[
 \seminormff{\cE f_-}{\cE \bar f_-}_{\gamma;\fraK}
 \leqs \seminormff{f_-}{\bar f_-}_{\gamma;\fraKbar}\;.
\]
Writing $\bar g(z)=\overline\cK^Q_\gamma \bar f_+(z)$, it follows in the same
way as above that 
\[
 \norm{g-\bar g}_{\gamma;\fraK} 
 \leqs \norm{Q}_{L^1} \norm{f_+ - \bar f_+}_{\gamma,\fraKbar}\;.
\]
Finally, we have 
\begin{multline*}
g(t,x) - \bar g(t,x) - \Gamma^\eps_{t,x;\bar t,\bar x} g(\bar t,\bar x) 
+ \Gammabar_{t,x;\bar t,\bar x} \bar g(\bar t,\bar x)   \\
 = \int_0^{2T} Q(s) 
\bigbrak{f_+(t-s,x) - \bar f_+(t-s,x) - \Gamma^\eps_{t,x;\bar t,\bar x} f_+(\bar
t-s,\bar x) + \Gammabar_{t,x;\bar t,\bar x} \bar f_+(\bar t-s,\bar x)}\6s\;.  
\end{multline*}
Using translation invariance of $\Gamma^\eps$ and $\Gammabar$, we obtain for all
$\beta<\gamma$ 
\begin{align*}
&\norm{g(t,x) - \bar g(t,x) - \Gamma^\eps_{t,x;\bar t,\bar x} g(\bar t,\bar x) 
+ \Gammabar_{t,x;\bar t,\bar x} \bar g(\bar t,\bar x)}_\beta \\
&\qquad\qquad \leqs 
\norm{Q}_{L^1} \sup_{s\in[0,2T]} \norm{(t-s,x)-(\bar t-s,\bar
x)}^{\gamma-\beta}_\fraks \seminormff{f_+}{\bar f_+}_{\gamma;\fraKbar} \\
&\qquad\qquad =  
\norm{Q}_{L^1} \norm{(t,x)-(\bar t,\bar
x)}^{\gamma-\beta}_\fraks \seminormff{f_+}{\bar f_+}_{\gamma;\fraKbar}\;,
\end{align*}
from which~\eqref{eq:Schauder_bounds} follows easily. 
\end{proof}

%%%%%%%%%%%%%%%%%%%%%%%%%%%%%%%%%%%%%%%%%%%%%%%%%%%%%%%%%%%%%%%%%%%%%%%%%%%%%%

\subsection{Extension to the space $\cD^{\gamma,\eta}$}
\label{ssec_D_gamma_eta}

As already mentioned at the end of Section~\ref{ssec_convolution}, the
definition of the spaces $\cD^\gamma$ has to be modified in order to be able to
deal with the time-zero singularity of the heat kernel. In particular, we have
to check that the estimates on $\cK^Q_\gamma$ established in
Proposition~\ref{prop:Schauder_E} still hold in these modified spaces. 

We recall a few notations from~\cite[Section~6]{Hairer2014}. Let
$P=\setsuch{(t,x)\in\R^{d+1}}{t=0}$ denote the time-zero hyperplane, and
introduce the scaled norms 
\begin{equation}
 \label{eq:normP}
 \norm{z}_P = 1\wedge \abs{t}^{1/2}\;, \qquad
 \norm{z,\bar z}_P = \norm{z}_P \wedge \norm{\bar z}_P
 = 1 \wedge \abs{t}^{1/2} \wedge \abs{\bar t\,}^{1/2}\;.
\end{equation} 
For a compact $\fraK \subset \R^{d+1}$, let 
\begin{equation}
 \label{eq:KP} 
 \fraK_P = \bigsetsuch{(z,\bar z)\in(\fraK\setminus P)^2}
 {z\neq\bar z, \norm{z-\bar z}_\fraks \leqs \norm{z,\bar z}_P}\;.
\end{equation} 

\begin{definition}[{\cite[Def.~6.2]{Hairer2014}}]
\label{def:D_gamma_eta} 
Given a model $(\Pi,\Gamma)$ and constants $\gamma>0$ and $\eta\in\R$, set 
\begin{equation}
 \label{eq:norm_gamma_eta1}
 \norm{f}_{\gamma,\eta;\fraK} 
 := \sup_{z\in\fraK\setminus P}~ \sup_{\beta<\gamma}~
 \frac{\norm{f(z)}_\beta}{\norm{z}_P^{(\eta-\beta)\wedge0}}\;, 
 \qquad
 \seminormbox{f}_{\gamma,\eta;\fraK} 
 := \sup_{z\in\fraK\setminus P}~ \sup_{\beta<\gamma}~
 \frac{\norm{f(z)}_\beta} {\norm{z}_P^{\eta-\beta}}\;.
\end{equation} 
The space $\cD^{\gamma,\eta}$ consists of all functions $f:\R^{d+1}\setminus
P\to T_\gamma^-$ such that 
\begin{equation}
 \label{eq:norm_gamma_eta2}
 \normDgamma{f}_{\gamma,\eta;\fraK} := 
 \norm{f}_{\gamma,\eta;\fraK}  
 + \sup_{(z,\bar z)\in\fraK_P}
 \sup_{\beta<\gamma}~\frac{\norm{f(z)-\Gamma_{z\bar z}f(\bar z)}_\beta}
 {\norm{z-\bar z}^{\gamma-\beta}_{\fraks}\norm{z,\bar z}_P^{\eta-\gamma}}
< \infty\;.
\end{equation}
Given a second model $(\Pibar,\Gammabar)$ and $\bar
f\in\cD^{\gamma,\eta}(\Gammabar)$, we also set 
\begin{equation}
 \label{eq:norm_gamma_eta3}
 \seminormff{f}{\bar f}_{\gamma,\eta;\fraK}
 := \norm{f-\bar f}_{\gamma,\eta;\fraK} 
 + \sup_{(z,\bar z)\in\fraK_P}
 \sup_{\beta < \gamma}
 ~\frac{\norm{f(z) - \bar f(z) - \Gamma_{z\bar z}f(\bar z) +
\Gammabar_{z\bar z}\bar f(\bar z)}_\beta}{\norm{z-\bar
z}_{\fraks}^{\gamma-\beta}\norm{z,\bar z}_P^{\eta-\gamma}}\;.
\end{equation}
Finally, similarly to Definition~\ref{def_Dgamma0}, we define
$\cD_0^{\gamma,\eta}$ to be the set of $f\in\cD^{\gamma,\eta}$ whose components
with negative homogeneity do not depend on $z$. 
\end{definition}

Most of the results in~\cite[Section~6]{Hairer2014} are directly applicable to
the present setting. The only result that has to be adapted is
Proposition~6.16, which takes here the following form.

\begin{prop}[Multilevel Schauder estimates on $\cD^{\gamma,\eta}_0$]
\label{prop:Schauder_D_gamma_eta}
Let $f\in\cD_0^{\gamma,\eta}(\Gamma^\eps)$ with  
\begin{equation}
 \label{eq:Schauder_cond_eta_gamma}
 -2 < \eta < 0 < \gamma {}<{} \eta+2\;,
\end{equation} 
and assume that $f_+(t,x) = 0$ whenever $t<0$. 
Then $\cK^Q_\gamma f\in\cD_0^{\gamma,\eta}(\Gamma^\eps)$, and there exists a
constant $C_1$, depending only on $Q$, such that 
\begin{equation}
\label{eq:Schauder_D_gamma_eta} 
\normDgamma{\cK^Q_\gamma f}_{\gamma,\eta;\fraK} \leqs 
C_1 \bigbrak{1+\norm{\Gamma^\eps}_{\gamma;\fraKbar}}
\normDgamma{f}_{\gamma,\eta;\fraKbar}\;, 
\end{equation}
where $\fraKbar$ is the $2T$-fattening of $\fraK$. Furthermore, let $\bar
f\in\cD_0^{\gamma,\eta}(\Gammabar)$, where $(\Pibar,\Gammabar)$ is another
translation-invariant model satisfying~\eqref{eq:Pi_xI} and~\eqref{eq:E01}, and
assume that $\bar f_+(t,x) = 0$ whenever $t<0$. Then there exists a constant
$C_2$, depending only on $Q$, such that 
\begin{equation}
\label{eq:Schauder_D_gamma_eta2} 
\seminormff{\cK^Q_\gamma f}{\overline\cK^Q_\gamma \bar f}_{\gamma,\eta;\fraK}
\leqs C_1 \bigbrak{1+\norm{\Gamma^\eps}_{\gamma;\fraKbar}}
\seminormff{f}{\bar f}_{\gamma,\eta;\fraKbar}
+ C_2 \norm{f}_{\gamma,\eta;\fraKbar} \norm{\Gammabar -
\Gamma^\eps}_{\gamma;\fraKbar}\;.
\end{equation}
\end{prop}
\begin{proof}
It follows as in the proof of Proposition~\ref{prop:Schauder_E} that 
\[
 \normDgamma{\cE f_-}_{\gamma,\eta;\fraK} \leqs
\normDgamma{f_-}_{\gamma,\eta;\fraKbar}\;, \qquad 
 \seminormff{\cE f_-}{\cE \bar f_-}_{\gamma,\eta;\fraK} \leqs 
 \seminormff{f_-}{\bar f_-}_{\gamma,\eta;\fraKbar}\;,
\]
so that we only have to consider the positive-homogeneous part $f_+$ of $f$. 
For convenience, we again write $g(z)$ instead of $(\cK^Q_\gamma f_+)(z)$. Note
that 
\[
 \norm{f_+}_{\gamma,\eta;\fraK} = \sup_{z\in\fraK\setminus P}
\sup_{\beta<\gamma} 
 \frac{\norm{f_+(z)}_\beta}{\norm{z}_P^{(\eta-\beta)\wedge 0}}
 = \sup_{z=(t,x)\in\fraK\setminus P} \; \sup_{0\leqs\beta<\gamma} \Bigbrak{
 (1\wedge\abs{t})^{(\beta-\eta)/2} \; \norm{f_+(z)}_\beta}\;,
\]
since we may always assume $\beta\geqs0$ and thus $(\eta-\beta)\wedge 0 =
\eta-\beta$. Since $f_+(t,x)=0$ for $t<0$ we have 
\[
 \norm{g(t,x)}_\beta 
\leqs \int_{(t-2T)\vee0}^t \abs{Q(t-s)} \norm{f_+(s,x)}_\beta \6s \\
\leqs \int_0^t \frac{\abs{Q(t-s)}}{(1\wedge s)^{(\beta-\eta)/2}} \6s \, 
\norm{f_+}_{\gamma,\eta;\fraKbar}\;.
\]
Since $\beta<\gamma<\eta+2$, the integral is convergent and we have 
\[
 \norm{g}_{\gamma,\eta;\fraK} \leqs C_Q
\norm{f_+}_{\gamma,\eta;\fraKbar}\;,
\]
where $C_Q$ depends only on $Q$. In order to estimate
$\normDgamma{g}_{\gamma,\eta;\fraK}$, we also need to bound differences of the
form $g(z)-\Gamma^\eps_{z\bar z}g(\bar z)$. From now on we assume that $z=(t,x)$
and $\bar z=(\bar t,\bar x)$ with $t\leqs\bar t$, since the case $t>\bar t$
follows by symmetry. Using the fact that $f_+$ is supported on $\set{t\geqs0}$
and translation invariance, we get
\[
 \norm{g(z)-\Gamma^\eps_{z\bar z}g(\bar z)}_\beta 
\leqs \int_0^{2T\wedge\bar t} \abs{Q(s)} \norm{f_+(t-s,x)
-\Gamma^\eps_{t-s,x;\bar t-s,\bar x}f_+(\bar t-s,\bar x)}_\beta \6s\;.
\]
However, there is a difficulty because $(z,\bar z)\in\fraK_P$ does not
automatically imply that one has $((t-s,x),(\bar t-s,\bar x))\in\fraK_P$. In
fact, 
\begin{align*}
((t-s,x),(\bar t-s,\bar x))\in\fraK_P 
&\Leftrightarrow (1\wedge\abs{t-s}{}\wedge{}\abs{\bar t-s})^{1/2} \geqs
\norm{(t-s,x)-(\bar t-s,\bar x)}_\fraks \\
&\Leftrightarrow 1\wedge\abs{t-s}{}\wedge{}\abs{\bar t-s} \geqs
\norm{z-\bar z}^2_\fraks\;.
\end{align*}
By definition of $\fraK_P$, we can assume that $\norm{z-\bar z}_\fraks \leqs 1$,
so that $s\leqs t-\norm{z-\bar z}_\fraks^2$ is a sufficient condition for
having $((t-s,x),(\bar t-s,\bar x))\in\fraK_P$. We split the integration
interval into $[0,t-\norm{z-\bar z}^2_\fraks]$ and $[t-\norm{z-\bar
z}^2_\fraks,2T\wedge\bar t\,]$, and treat each interval separately. For the
first interval, we use the definition of $\normDgamma{f_+}$ to get 
\begin{align*}
\int_0^{t-\norm{z-\bar z}^2_\fraks} \abs{Q(s)} &\norm{f_+(t-s,x)
-\Gamma^\eps_{t-s,x;\bar t-s,\bar x}f_+(\bar t-s,\bar x)}_\beta \6s \\
&\leqs \int_0^{t-\norm{z-\bar z}^2_\fraks}
\frac{\abs{Q(s)}}{(1\wedge(t-s))^{(\gamma-\eta)/2}}\6s \; \norm{z-\bar
z}^{\gamma-\beta}_\fraks \normDgamma{f_+}_{\gamma,\eta;\fraKbar} \\
&\leqs \frac{C_Q}{(1\wedge t)^{(\gamma-\eta)/2}} \norm{z-\bar
z}^{\gamma-\beta}_\fraks \normDgamma{f_+}_{\gamma,\eta;\fraKbar}\;.
\end{align*}
For the second part of the integral, we bound $\norm{f_+(t-s,x)
-\Gamma^\eps_{z\bar z}f_+(\bar t-s,\bar x)}_\beta$ above by the sum
$\norm{f_+(t-s,x)}_\beta + \norm{\Gamma^\eps_{z\bar z}f_+(\bar t-s,\bar
x)}_\beta$, and treat each term separately. Using the assumption
$f_+(t,x)=0$ for $t<0$ and performing the change of variables $s=t-\bar s$, 
the first term becomes 
\begin{align*}
\int_0^{\norm{z-\bar z}_\fraks^2} \abs{Q(t-\bar s)} \norm{f_+(\bar s,x)}_\beta
\6\bar s 
&\leqs \int_0^{\norm{z-\bar z}_\fraks^2} \frac{\abs{Q(t-\bar s)}}
{(1\wedge\bar s)^{(\beta-\eta)/2}} \6\bar s
\;\norm{f_+}_{\gamma,\eta;\fraKbar}\\
&\lesssim \norm{z-\bar z}_\fraks^{2-\beta+\eta}
\;\norm{f_+}_{\gamma,\eta;\fraKbar}\;.
\end{align*}
For the second term, we use the definition of $\norm{\Gamma^\eps}$ and the
decomposition of $f_+$ as a linear combination of $\tau$ to obtain  
\begin{align*}
\norm{\Gamma^\eps_{z\bar z}f_+(\bar s,\bar x)}_\beta 
&\leqs \sum_\delta \norm{f_+(\bar s,\bar x)}_\delta \norm{z-\bar
z}_\fraks^{\delta-\beta} \norm{\Gamma^\eps}_{\gamma;\fraKbar} \\
&\leqs \sum_\delta \frac{\norm{f_+}_{\gamma,\eta;\fraKbar}}
{(1\wedge\bar s)^{(\delta-\eta)/2} } \norm{z-\bar
z}_\fraks^{\delta-\beta} \norm{\Gamma^\eps}_{\gamma;\fraKbar}\;,
\end{align*}
where the sum runs over finitely many values of $\delta\in[0,\gamma]$. 
Setting $s=\bar t-\bar s$, this yields a contribution 
\begin{align*}
\int_0^{\norm{z-\bar z}_\fraks^2+\bar t-t} \abs{Q(\bar t-\bar s)} 
&\norm{\Gamma^\eps_{z\bar z} f_+(\bar s,x)}_\beta \6\bar s \\
&\leqs \sum_\delta \int_0^{\norm{z-\bar z}_\fraks^2+\bar t-t}
\frac{\abs{Q(\bar t-\bar s)}}
{(1\wedge\bar s)^{(\delta-\eta)/2}} \6\bar s
\;\norm{z-\bar z}_\fraks^{\delta-\beta}\norm{f_+}_{\gamma,\eta;\fraKbar}
\norm{\Gamma^\eps}_{\gamma;\fraKbar} \\
&\lesssim \sum_\delta \norm{z-\bar z}_\fraks^{2-\delta+\eta}
\norm{z-\bar z}_\fraks^{\delta-\beta}
\norm{f_+}_{\gamma,\eta;\fraKbar}
\norm{\Gamma^\eps}_{\gamma;\fraKbar} \\
&\lesssim \norm{z-\bar z}_\fraks^{2+\eta-\beta}\norm{f_+}_{\gamma,\eta;\fraKbar}
\norm{\Gamma^\eps}_{\gamma;\fraKbar}\;.
\end{align*}
Combining the different bounds, we obtain~\eqref{eq:Schauder_D_gamma_eta}.

It remains to obtain a similar bound on $\seminormff{g}{\bar g}$, where $\bar
g=\overline\cK^Q_\gamma\bar f_+$. The same argument as before yields 
\[ 
\norm{g-\bar g}_{\gamma,\eta;\fraK} \leqs C_Q
\norm{f_+-\bar f_+}_{\gamma,\eta;\fraKbar}\;. 
\]
Regarding the second term in the definition of $\seminormff{g}{\bar g}$, we
again split the integral defining it into two parts. The first one is  
\begin{align*}
\int_0^{t-\norm{z-\bar z}^2_\fraks} \abs{Q(s)} &\norm{f_+(t-s,x)
-\bar f_+(t-s,x)-\Gamma^\eps_{z\bar z}f_+(\bar t-s,\bar x)
+\Gammabar_{z\bar z}\bar f_+(\bar t-s,\bar x)}_\beta \6s \\
&\leqs \frac{C_Q}{(1\wedge t)^{(\gamma-\eta)/2}} \norm{z-\bar
z}^{\gamma-\beta}_\fraks \normDgamma{f_+;\bar f_+}_{\gamma,\eta;\fraKbar}\;.
\end{align*}
For the second part, we use the decomposition 
\begin{align*}
\norm{f_+(t-s,x)
&-\bar f_+(t-s,x)-\Gamma^\eps_{z\bar z}f_+(\bar t-s,\bar x)
+\Gammabar_{z\bar z}\bar f_+(\bar t-s,\bar x)}_\beta \\
\leqs{}& \norm{f_+(t-s,x)-\bar f_+(t-s,x)}_\beta \\
&{}+\norm{\Gamma^\eps_{z\bar z}(f_+(\bar t-s,\bar x)-\bar f_+(\bar t-s,\bar
x))}_\beta \\
&{}+\norm{(\Gamma^\eps_{z\bar z}-\Gammabar_{z\bar z})\bar f_+(\bar t-s,\bar
x)}_\beta\;,
\end{align*}
and estimate the integral of each term separately, in the same way as above. 
Combining the different bounds, we obtain~\eqref{eq:Schauder_D_gamma_eta2}. 
\end{proof}

\begin{remark}
\label{rem:Schauder_integer}
Note that the condition~\eqref{eq:Schauder_cond_eta_gamma} on $\eta$ and
$\gamma$ excludes integer values for these parameters, except that we may have 
either $\eta=-1$ or $\gamma=1$. The reason while these integer values are
allowed is that owing to the parabolic scaling, the exponents such as
$(\gamma-\eta)/2$ occurring in integrals over time can never take integer
values that could make them divergent. 
\end{remark}

%%%%%%%%%%%%%%%%%%%%%%%%%%%%%%%%%%%%%%%%%%%%%%%%%%%%%%%%%%%%%%%%%%%%%%%%%%%%%%

\section{The fixed-point equation}
\label{sec_fix}

In this section, we construct a fixed-point equation lifting~\eqref{eq:fhn02} to
an appropriate space $\cD^{\gamma,\eta}$, and prove that it admits a unique
fixed point, depending Lipschitz-continuously on the lifted initial data and
noise. The procedure is very similar to the one for the Allen--Cahn equation or
the $\Phi_3^4$ model carried out in~\cite[Section~9]{Hairer2014}, with only a
few adjustments required by the additional variable $v$. 

%%%%%%%%%%%%%%%%%%%%%%%%%%%%%%%%%%%%%%%%%%%%%%%%%%%%%%%%%%%%%%%%%%%%%%%%%%%%%%

\subsection{The set-up}
\label{ssec_fix_setup}

From now on, we assume that the space dimension is $d=3$, since this is the
most difficult case which is still renormalisable. 
Consider the fixed-point equation on $\cD^{\gamma,\eta}\times\cD^{\gamma,\eta}$
given by 
\begin{align}
\nonumber
U &= (\cK_{\bar\gamma}+R_\gamma\cR)\Rplus \bigbrak{\Xi+F(U,V)} + Gu_0\;, \\
V &= \cK^Q_\gamma\Rplus U + \widehat Qv_0\;, 
\label{eq:fix01} 
\end{align}
where $\Rplus(t,x)=\indexfct{t>0}$ and $\bar\gamma\geqs\gamma-2$ as before, 
and by definition, $\cR\Rplus\Xi$ is given by the distribution $\xi^\eps\indexfct{t>0}$.
Furthermore, 
\begin{align}
\nonumber
(Gu_0)(t,x) &= (S(t)u_0)(x) = \int_{\R^3} G(t,x-y)u_0(y)\6y\;, \\
(\widehat Qv_0)(t,x) &= \e^{tA_2}v_0(x) \indexfct{t>0}\;.
\label{ew:fix01b} 
\end{align}
In fact, these distributions have to be lifted to $\cD^{\gamma,\eta}$, which is
always possible by construction, cf~\cite[Def.~2.14]{Hairer2014}. In order to 
understand the structure of the solution of the
fixed-point equation~\eqref{eq:fix01}, let us rewrite it in the form 
\begin{align}
\nonumber
U &= \cI\Rplus \bigbrak{\Xi+F(U,V)} + \widehat R(U,V) + Gu_0\;, \\
V &= \cE\Rplus U_- + \cQ \Rplus U_+ + \widehat Qv_0\;, 
\label{eq:fix01a} 
\end{align}
where $U_-$ and $U_+$ denote respectively the strictly negative homogeneous
and positive homogeneous parts of $U$ (cf.~\eqref{eq:def_Dgamma0}),
$\widehat R(U,V)$ contains the polynomial part of
$\cK_{\bar\gamma}\bigbrak{\Xi+F(U,V)}$ as well as the term
$R_\gamma\cR\bigbrak{\Xi+F(U,V)}$ (which always belongs to $\Tbar$), and
\begin{equation}
 \label{eq:def_cQ}
 \cQ f_+(t,x) = \int_{t-2T}^t Q(t-s)f_+(s,x)\6s\;.
\end{equation} 
Assume $F$ is a cubic polynomial of the form 
\begin{equation}
 \label{eq:fix02}
  F(u,v) = \alpha_1 u + \alpha_2 v + \beta_1 u^2 + \beta_2 uv + \beta_3 v^2
 + \gamma_1 u^3 + \gamma_2 u^2v + \gamma_3 uv^2 + \gamma_4 v^3\;.
\end{equation}
Further assume that the initial conditions $u_0$ and $v_0$ are such that $Gu_0$
and $\widehat Qv_0$ are sufficiently regular. Iterating the
map~\eqref{eq:fix01a}, starting for instance with identically zero functions, it
is not difficult to see that any fixed point of~\eqref{eq:fix01} must have the
form 
\begin{alignat}{4}
\nonumber
U &= \RSI + \varphi\unit 
&&{}+{} \bigbrak{a_1\RSIW + a_2\RSIWo + a_3\RSIWoo + a_4\RSIWooo} 
&&{}+{} \bigbrak{b_1\RSY + b_2\RSYo + b_3\RSYoo} + \pscal{\nabla\varphi}{X} 
&&{}+{}\varrho_U\;,
\\
V &= \RSoI + \psi\unit
&&{}+{} \bigbrak{\hat a_1\RSIW + \hat a_2\RSIWo + \hat a_3\RSIWoo +
\hat a_4\RSIWooo}
&&{}+{} \bigbrak{\hat b_1\RSY + \hat b_2\RSYo + \hat b_3\RSYoo} +
\pscal{\nabla\psi}{X}
&&{}+{}\varrho_V\;,
\label{eq:fix03} 
\end{alignat}
where the coefficients satisfy $a_i=\gamma_i$ for $i=1,2,3,4$,
\begin{align}
\nonumber
b_1 &= \beta_1 + 3\varphi\gamma_1 +  \psi\gamma_2\;, \\
\label{eq:fix04} 
b_2 &= \beta_2 + 2\varphi\gamma_2 + 2\psi\gamma_3\;, \\
b_3 &= \beta_3 +  \varphi\gamma_3 + 3\psi\gamma_4\;,
\nonumber
\end{align}
and $\hat a_i\tau = \cQ(a_i\tau)$, $\hat b_i\tau = \cQ(b_i\tau)$. 
In~\eqref{eq:fix03}, $\varrho_U, \varrho_V$ denote terms of homogeneity at least
$\frac32 - \kappa$, and $\varphi(z)$, $\psi(z)$, $\nabla\varphi(z)$,
$\nabla\psi(z)$ are independent functions --- the notation is not supposed to
suggest that $\varphi$ and $\psi$ are differentiable. The scalar product
notation $\pscal{\nabla\varphi}{X}$ is a shorthand for $\sum_{i=0}^3
(\nabla\varphi)_i X_i$. 

\begin{prop}[Equivalence of fixed-point equations]
\label{prop:fixed_point_equiv}
Assume that $\eps>0$ and that~\eqref{eq:fix01} admits a fixed point $(U,V)$
where $U, V\in\cD^{\gamma,\eta}$ for some $\gamma, \eta\in\R$. Then $(u,v)=(\cR
U, \cR V)$ satisfies the fixed-point equation~\eqref{eq:fhn02}.
\end{prop}
\begin{proof}
The first step is to note that~\eqref{eq:Pi_xI} and~\eqref{eq:E01} imply
\begin{align*}
 (\Pi_{\bar z}^\eps \RSI)(z) &= 
 \int K(z-z_1)\xi^\eps(z_1)\6z_1 =: \chi_\eps(z)\;, \\
 (\Pi_{\bar z}^\eps \RSoI)(z) &= 
 \int \KQ(z-z_1)\xi^\eps(z_1)\6z_1 =: \chi^Q_\eps(z)\;, 
\end{align*} 
independently of $\bar z$, where we have used the notation 
\[
 \KQ(t,x) = \int_0^{2T} K(t-s,x) Q(s)\6s\;.
\]
Note that both in dimensions $d=3$ and $d=2$, the sum in~\eqref{eq:Pi_xI}
is empty, because in both cases $\RSI$ has strictly negative homogeneity,
cf.\ Table~\ref{tab:FF_AllenCahn}.
Applying the reconstruction operator $\cR$ to~\eqref{eq:fix03}, since all
functions are smooth for $\eps>0$, we obtain
\begin{align*}
u(z) &= \chi_\eps(z) + \varphi(z)\;, \\
v(z) &= \chi^Q_\eps(z) + \psi(z)\;. 
\end{align*}
Here we have used the fact that if $f_+\in\cD^{\gamma,\eta}$ has only components
of strictly positive homogeneity, one has $\cR f_+(z) = (\Pi^\eps_z f_+(z))(z) =
0$ (cf.~\eqref{eq:E_05} as well as~\cite[Prop.~3.28]{Hairer2014}). 
On the other hand, using~\eqref{eq:K_gamma1} and
\eqref{eq:R_gamma2}, we obtain for all $f\in\cD^{\gamma,\eta}$ 
\[
 \bigpar{\cR(\cK_{\bar\gamma} + R_\gamma\cR)f}(z) 
 = (G*\cR f)(z)\;.
\]
Applying $\cR$ to the equations~\eqref{eq:fix01} and using~\eqref{eq:REf_KRf},
we thus obtain 
\begin{align*}
u(t,x) &= G * \cR\Rplus \bigbrak{\Xi+F(U,V)}(t,x) + Gu_0(t,x)\;, \\
v(t,x) &= \int_0^t (\cR\Rplus U)(s,x) Q(t-s)\6s + \widehat Qv_0(t,x)\;.  
\end{align*}
It is not difficult to check that $\cR$ and $\Rplus$ commute. Thus the result
follows if we are able to show that 
\[
 \cR F(U,V) = F(\cR U, \cR V)\;.
\]
Computing the part of nonpositive homogeneity of the difference $\cR F(U,V) =
F(\cR U, \cR V)$, we see that it is a linear combination of terms of the form 
\[
 (\Pi^\eps_z \RSWV)(z)\;, \quad
 (\Pi^\eps_z \RSVW)(z)\;, \quad
 (\Pi^\eps_z \RSWW)(z) \quad
 \text{and} \quad 
 (\Pi^\eps_z \RSI)(z)^2 \pscal{\nabla\varphi}{(\Pi^\eps_z X)(z)}\;,
\]
and all possible terms obtained from these by substitutions of the form
$\RSI\mapsto\RSoI$. It follows directly from~\eqref{eq:Pi_x} that 
$(\Pi^\eps_z X)(z) = 0$. It thus remains to show that all terms of the three
other types vanish as well. Applying~\eqref{eq:Pi_xI}, \eqref{eq:f_x}, and 
recalling \eqref{eq:Jiszero} we obtain 
\[
(\Pi^\eps_{\bar z} \RSY)(z) = 
\int \bigbrak{K(z-z_1)-K(\bar z-z_1)}(\Pi^\eps_{\bar z} \RSV)(z_1)\6z_1\;, 
\]
which vanishes in $z=\bar z$. It follows that 
$(\Pi^\eps_z \RSWV)(z) = (\Pi^\eps_z \RSY)(z)(\Pi^\eps_z \RSV)(z) = 0$. 
All other terms can be treated similarly, and the result follows. 
\end{proof}

\begin{remark}
\label{rem:d=2}
In dimension $d=2$, the symbols $\RSWV, \RSVW, \RSWW$ have strictly positive
homogeneity, so that the last part of the proof is not needed. 
\end{remark}

The expansion~\eqref{eq:fix03} shows in particular that both $U$ and $V$ have
regularity $-\frac12-\kappa$, and that their components of negative homogeneity
do not depend on $(t,x)$. We will thus look for a fixed point
of~\eqref{eq:fix03} in a space $\cD^{\gamma,\eta}_0$ for appropriate values of
$\gamma$ and $\eta$.

%%%%%%%%%%%%%%%%%%%%%%%%%%%%%%%%%%%%%%%%%%%%%%%%%%%%%%%%%%%%%%%%%%%%%%%%%%%%%%

\subsection{Local existence and uniqueness of the fixed point}
\label{ssec_fix_exist}

In order to describe the effect of $F$, it will be useful to introduce the
following slight generalisation of Definition~\ref{def_Dgamma0}. Given a model
$(\Pi,\Gamma)$, the space $\cD^{\gamma,\eta}_{\beta,\alpha}(\Gamma)$ is the
space of functions $f \in \cD^{\gamma,\eta}(\Gamma)$ of the form 
\begin{equation}
 \label{eq:fex01}
 f(z) = \sum_{\tau\in T\colon \alpha \leqs \abss{\tau} < \beta} c_\tau \tau +
\sum_{\tau\in T\colon\abss{\tau}\geqs\beta} \hat c_\tau(z)\tau \;,
\end{equation}
where the $c_\tau$ do not depend on $z$. 

The following result establishes the strong local Lipschitz continuity of $F$
in the sense of~\cite[Section~7.3]{Hairer2014}. 

\begin{prop}[Strong local Lipschitz continuity]
\label{prop:f_Lipschitz}
Fix $\alpha < 0 < \gamma$ and set $\bar\gamma=\gamma+2\alpha$ and 
$\bar\eta = 3\eta \wedge(\eta+2\alpha)$. Let $\fraK\subset\R^{d+1}$ be
compact, and assume $U,V\in\cD^{\gamma,\eta}_{0,\alpha}(\Gamma)$ are such that
\begin{equation}
 \label{eq:fex02}
 \normDgamma{U}_{\gamma,\eta;\fraK} + \normDgamma{V}_{\gamma,\eta;\fraK} 
 \leqs R\;.
\end{equation} 
Then $F(U,V)\in\cD^{\bar\gamma,\bar\eta}_{2\alpha,3\alpha}(\Gamma)$, and there
exists a constant $C(R)$ such that 
\begin{equation}
 \label{eq:fex03}
 \normDgamma{F(U,V)}_{\bar\gamma,\bar\eta;\fraK} \leqs C(R)\;.
\end{equation} 
Furthermore, if $(\Pibar,\Gammabar)$ is a possibly different model and
$\overline U,\overline V\in\cD^{\gamma,\eta}_{0,\alpha}(\Gammabar)$ satisfy 
the bound $\normDgamma{\overline U}_{\gamma,\eta;\fraK} +
\normDgamma{\overline V}_{\gamma,\eta;\fraK} \leqs R$, then
\begin{equation}
 \label{eq:fex04}
 \seminormff{F(U,V)}{F(\overline U,\overline V)}_{\bar\gamma,\bar\eta;\fraK}
\leqs 
C(R) \Bigbrak{\seminormff{U}{\overline U}_{\gamma,\eta;\fraK} +
\seminormff{V}{\overline V}_{\gamma,\eta;\fraK}
+ \norm{\Gamma - \Gammabar}_{2\gamma+\alpha;\fraK}}\;.
\end{equation} 
\end{prop}
\begin{proof}
Writing $U(z)=U_-+U_+(z)$ where $U_-$ contains all terms of negative
homogeneity and applying~\cite[Prop.~6.12]{Hairer2014}, we see that 
\begin{alignat*}{3}
U(z)^2 &= U_-^2 + 2 U_- U_+(z) + U_+(z)^2 &&\in
\cD^{\gamma+\alpha,2\eta\wedge(\eta+\alpha)}_{\alpha,2\alpha}\;, \\
U(z)^3 &= U_-^3 + 3 U_-^2 U_+(z) + 3 U_- U(z)^2 + U_+(z)^3 &&\in 
\cD^{\gamma+2\alpha,3\eta\wedge(\eta+2\alpha)}_{2\alpha,3\alpha}\;,
\end{alignat*}
where the multiplication of elements in $\cD^{\gamma,\eta}_{\cdot,\cdot}$ is 
well-defined due to the results in \cite[Sec.~4]{Hairer2014}. 
Similar relations hold for powers of $V$ and cross-terms,
proving~\eqref{eq:fex03}. The proof of~\eqref{eq:fex04} is similar, using the
bound on $\seminormff{f}{g}$ given in~\cite[Prop.~6.12]{Hairer2014}.
\end{proof}

\begin{remark}
\label{rem:fix01} 
In cases where the two models $(\Pi,\Gamma)$ and $(\Pibar,\Gammabar)$ are
identical, \eqref{eq:fex04} automatically provides the bound
\begin{equation}
 \label{eq:fex_remark} 
\normDgamma{F(U,V)-F(\overline U,\overline V)}_{\bar\gamma,\bar\eta;\fraK}
\leqs C(R) \Bigbrak{\normDgamma{U-\overline U}_{\gamma,\eta;\fraK} +
\normDgamma{V-\overline V}_{\gamma,\eta;\fraK}}\;.
\end{equation} 
\end{remark}

In order to proceed, it is convenient to rewrite the fixed-point
equation~\eqref{eq:fix01} in a slightly different form. First we set 
\begin{equation}
 \label{eq:fix05}
W := (\cK_{\bar\gamma} + R_\gamma\cR)\Rplus \Xi\;.
\end{equation} 
For the time being, let us \emph{assume} that 
$W\in\cD^{\gamma,\eta}_{0,\alpha_0+2}$ for some $\gamma\in\R$ and all
$\eta < \alpha_0+2$ (recall that $\alpha_0 = -\frac52 - \kappa$ in dimension
$d=3$); we come back to this point in Remark~\ref{rem:fix02} below. Then
we define the map
\begin{equation}
 \label{eq:fix06}
 \cM \colon U \mapsto 
 (\cK_{\bar\gamma}+R_\gamma\cR)\Rplus F\bigpar{U,\cK^Q_\gamma\Rplus U+\widehat
Qv_0} + W + Gu_0\;. \\
\end{equation}
If $U$ is a fixed point of $\cM$ and $V=\cK^Q_\gamma\Rplus U+\widehat Qv_0$ then
$(U,V)$ is a fixed point of~\eqref{eq:fix01}. 

Following~\cite[Section~7.1]{Hairer2014}, we write $O=[-1,2]\times\R^3$ and
$O_T=(-\infty,T]\times\R^3$, and introduce the shorthand
$\normDgamma{\cdot}_{\gamma,\eta;T}$ for $\normDgamma{\cdot}_{\gamma,\eta;O_T}$,
and define $\seminormff{\cdot}{\cdot}_{\gamma,\eta; T}$ similarly. 
Finally, we use the notation $\seminormff{Z}{\bar Z}_{\gamma;\fraK} = 
\norm{\Pi-\Pibar}_{\gamma;\fraK} + \norm{\Gamma-\Gammabar}_{\gamma;\fraK}$ to
quantify the difference between two models.
The following result is an adaptation of~\cite[Thm.~7.8]{Hairer2014}. 

\begin{prop}[Existence and uniqueness of the fixed point]
\label{prop:fixed_point}
Assume $-\frac23 < \eta < -\frac12$ and $\alpha > -1$ are such that
$\eta+2\alpha > -2$, and assume $\gamma > -2\alpha$.  
For any $u_0, v_0, W$ such that $W+Gu_0, \widehat
Qv_0\in\cD^{\gamma,\eta}_{0,\alpha}$, there exists a time $T>0$ such that $\cM$
admits a unique fixed point
$U^*\in\cD^{\gamma,\eta}$ on $(0,T)$.
Furthermore, the solution map $\cS_T:(u_0,v_0,W,Z) \mapsto U^*$ is
jointly Lipschitz continuous in the sense that if $Z$ and $\bar Z$ are two
models and $(u_0,v_0,W)$ and $(\bar u_0,\bar v_0,\overline W)$ are two sets of
initial conditions and forcing terms such that 
\begin{equation}
 \label{eq:prop_fixed_point01}
 \seminormff{Gu_0}{G\bar u_0}_{\gamma,\eta; T} +  
 \seminormff{\widehat Q v_0}{\widehat Q \bar v_0}_{\gamma,\eta; T} + 
 \seminormff{W}{\overline W}_{\gamma,\eta;T} + 
 \seminormff{Z}{\bar Z}_{2\gamma+\alpha;O} \leqs \delta\;,
\end{equation} 
then one has
\begin{equation}
 \label{eq:prop_fixed_point02}
 \seminormff{U^*}{\overline U^*}_{\gamma,\eta;T} 
 \leqs C_0 \delta
\end{equation} 
for some constant $C_0>0$. 
\end{prop}
\begin{proof}
Given $R>0$ let $\cB(R) = \setsuch{f\in\cD^{\gamma,\eta}_{0,\alpha}}
{\normDgamma{f}_{\gamma,\eta;T}\leqs R}$ and pick $U\in\cB(R)$. 
By Proposition~\ref{prop:Schauder_D_gamma_eta},
we have 
\[
V := \cK^Q_\gamma U + \widehat Qv_0 
\in\cB(C_1'R+\normDgamma{\widehat Qv_0}_{\gamma,\eta;T}) 
\]
for a constant $C_1'$ depending only on $Q$ and on the model. 
Applying Proposition~\ref{prop:f_Lipschitz}, we see that 
$F(U,V)\in\cD^{\bar\gamma,\bar\eta}_{2\alpha,3\alpha}$ and satisfies 
\[
 \normDgamma{F(U,V)}_{\bar\gamma,\bar\eta;T}
 \leqs C\bigpar{R + C_1'R+\normDgamma{\widehat Qv_0}_{\gamma,\eta;T}}\;.
\]
The assumptions on $\gamma, \eta$ and $\alpha$ guarantee that $\bar\gamma>0$ and
$\bar\eta>-2$. We can thus apply Theorem~7.1 and Lemma~7.3 in~\cite{Hairer2014}
to obtain the existence of constants $C_2, \kappa>0$ such that 
\[
 \normDgamma{\cM(U)}_{\gamma,\eta;T} \leqs C_2T^{\kappa/2}
C\bigpar{R + C_1'R+\normDgamma{\widehat Qv_0}_{\gamma,\eta;T}} +
\normDgamma{W}_{\gamma,\eta;T} + 
\normDgamma{Gu_0}_{\gamma,\eta;T}
\]
for every $T\in(0,1]$. Choosing 
\[
R=1+\bigbrak{\normDgamma{W}_{\gamma,\eta;T} +
\normDgamma{Gu_0}_{\gamma,\eta;T}} 
\]
and $T$ sufficiently small, we see that
$\cM$ maps the set $\cB(R)$ into itself. 

Next we show that $\cM$ is a contraction in $\cB(R)$. For
$U,\overline U\in\cB(R)$, $V$ as above and $\overline{V}=\cK^Q_\gamma \overline
U + \widehat Qv_0$, we have 
\[
 \cM(U) - \cM(\overline U) = 
 (\cK_{\bar\gamma}+R_\gamma\cR)\Rplus 
 \bigbrak{F(U,V) - F(\overline U,\overline V)}\;.
\]
Using again Theorem~7.1 and Lemma~7.3 in~\cite{Hairer2014} and
applying~\eqref{eq:fex_remark} to bound the term $\bigbrak{F(U,V) - F(\overline
U,\overline V)}$, we get 
\[
 \normDgamma{\cM(U) - \cM(\overline U)}_{\gamma,\eta;T} \leqs
C_2T^{\kappa/2} C\bigpar{R + C_1'R+\normDgamma{\widehat Qv_0}_{\gamma,\eta;T}}
\bigbrak{\normDgamma{U - \overline U}_{\gamma,\eta;T}}\;.
\]
Taking again $T$ small enough, we obtain that $\cM$ is indeed a contraction, and
the existence of a unique fixed point in a ball in $\cD^{\gamma,\eta}$ follows
from Banach's fixed-point theorem.

Finally, in order to prove the bound~\eqref{eq:prop_fixed_point02},  
we set $U_1 = (\cK_{\bar\gamma}+R_\gamma\cR) \Rplus F(U,V)$ and $\overline
U_1 = (\cK_{\bar\gamma}+R_\gamma\cR) \Rplus F(\overline U,\overline V)$. 
Using the inequality 
\[
\seminormff{f+g}{\bar f+\bar g}_{\gamma,\eta;T} \leqs
\seminormff{f}{\bar f}_{\gamma,\eta;T} + \seminormff{g}{\bar g}_{\gamma,\eta;T}
\] 
we obtain
\[
 \seminormff{\cM(U)}{\cM(\overline U)}_{\gamma,\eta;T} 
 \leqs \seminormff{U_1}{\overline U_1}_{\gamma,\eta;T} 
 + \seminormff{W}{\overline W}_{\gamma,\eta;T} 
 + \seminormff{Gu_0}{G\bar u_0}_{\gamma,\eta;T}\;,
\]
where the first term on the right-hand side can be estimated with the help
of~\eqref{eq:fex04}. The required bound thus follows from~\cite[Thm.~7.1 and
Lemma~7.3]{Hairer2014}.
\end{proof}

\begin{remark}
\label{rem:fix02}
In fact, Proposition~\ref{prop:fixed_point} can be applied as soon as we choose
$\gamma > 1$. Indeed, let $\alpha=-\frac12-\kappa$ and assume $\gamma = 1+\bar
\kappa$ for some $\bar\kappa>0$. 
\begin{enum}
\item 	The conditions on $\gamma$ and $\eta$ are satisfied provided $\kappa <
\bar \kappa$ and $\kappa \leqs \frac16$. 

\item 	Lemma~7.5 in~\cite{Hairer2014} shows that if $u_0\in\cC^\eta$ (with the
Euclidean scaling and recalling that we take $-\frac23 < \eta < -\frac12$) then 
$Gu_0\in\cD^{\gamma,\eta}$ for any $\gamma > (\eta\vee0)$.

\item 	Definition~2.14 in~\cite{Hairer2014} shows that if $v_0\in\cC^\gamma$,
then it can be lifted to a an element $v_0\in\cD^\gamma$, and then $\widehat Q
v_0$ will belong to $\cD^{\gamma,\eta}$ as required. 

\item 	As discussed in Section~9.4 of~\cite{Hairer2014} (see Equation (9.18)
and below), the assumption $W\in\cD^{\gamma,\eta}_{0,\alpha_0+2}$ is indeed
satisfied for $\gamma, \eta$ as in Proposition~\ref{prop:fixed_point}, provided
we define $\cR\Rplus\Xi$ as the distribution $\xi^\eps\indexfct{t\geqs0}$.~$\lozenge$  
\end{enum}
\end{remark}

Let $Z^\eps = \Psi(\xi^\eps)$ denote the canonical model for mollified noise
$\xi^\eps$ constructed in Sections~\ref{ssec_model} and~\ref{ssec_ext}. In
exactly the same way as in~\cite[Prop.~9.8 and 9.10]{Hairer2014}, we can now
extend the solution map $\cS_T$ to a maximal solution map $\cS^L$, defined until
the first time $\norm{(\cR U)(t,\cdot)}_\eta + \norm{(\cR V)(t,\cdot)}_\eta$
reaches a (large) cut-off value $L$. This map satisfies 
\begin{equation}
\label{eq:fix07} 
 \cR \cS^L(u_0,v_0,Z^\eps) = \bar\cS^L(u_0,v_0,\xi^\eps)\;,
\end{equation}
where $\bar\cS^L$ is the solution map of the original SPDE--ODE
system~\eqref{eq:fhn01} with mollified noise $\xi^\eps$. In other words, the
following diagram commutes: 

\begin{center}
\begin{tikzpicture}[>=stealth']
\matrix (m) [matrix of math nodes, row sep=3em,
column sep=3.5em, text height=1.75ex, text depth=0.5ex]
{ (u_0,v_0,Z^\eps) & (U^*,V^*) \\
 (u_0,v_0,\xi^\eps) & (u^\eps,v^\eps) \\};
\path[|->]
(m-1-1) edge node[auto] {$\cS^L$} (m-1-2)
(m-2-1) edge node[auto] {$\Psi$} (m-1-1)
(m-1-2) edge node[auto] {$\cR$} (m-2-2)
(m-2-1) edge node[below] {$\bar\cS^L$} (m-2-2);
\end{tikzpicture}
\end{center}

%%%%%%%%%%%%%%%%%%%%%%%%%%%%%%%%%%%%%%%%%%%%%%%%%%%%%%%%%%%%%%%%%%%%%%%%%%%%%%

\section{Renormalisation of the equation}
\label{sec_renorm}

%%%%%%%%%%%%%%%%%%%%%%%%%%%%%%%%%%%%%%%%%%%%%%%%%%%%%%%%%%%%%%%%%%%%%%%%%%%%%%

\subsection{The renormalisation group}
\label{ssec_renorm_gen}

Consider a general model $Z=(\Pi,\Gamma)$ for our regularity structure, which is
admissible in the sense that is satisfies~\eqref{eq:Pi_xI}, the last identity
in~\eqref{eq:f_x} and~\eqref{eq:E01}. The model may also be specified by
the pair $(\bPi,f)$ such that for all $z, \bar z\in\R^{d+1}$, one has 
\begin{equation}
 \label{eq:rgen01}
 \Pi_z = \bPi \,\Gamma_{f_z}
 \qquad \text{and} \qquad
 \Gamma_{z\bar z} = 
 (\Gamma_{f_z})^{-1} \Gamma_{f_{\bar z}}
 =: \Gamma_{\gamma_{z\bar z}}\;.
\end{equation} 
To build a renormalisation transformation, one first needs to define sets
$\cF_\star\subset\cF_0\subset\cF_F$ such that 
\begin{itemiz}
\item 	$\cF_0$ contains all $\tau\in\cF_F$ of strictly negative homogeneity;
\item 	for all $\tau\in\cF_0$, one has $\Delta\tau\in
\vspan(\cF_0)\otimes\vspan(\cF_0^+)$, where
$\cF_0^+\subset\cF_F^+$ contains all symbols of the
form~\eqref{eq:rs09} with $\tau_j\in\cF_\star$ and $\abss{\cJ_{k_j}\tau_j}>0$
(including the empty product). 
\end{itemiz}
See~\cite[Sect.~8.3]{Hairer2014}, in particular Remark~8.37, for a proof
that it is always possible to find sets $\cF_\star$ and $\cF_0$ satisfying these
two properties.
For the Allen--Cahn model, a possible choice is 
\begin{align}
\nonumber
\cF_\star &= \set{\RSW,\RSV,\RSI}\;, \\
\cF_0 &= \set{\Xi,\RSW,\RSV,\RSI,\RSWW,\RSV
X_i,\RSVW,\RSWV,\unit,\RSWI,\RSY,\RSVI,X_i}\;.
\label{eq:rgen02} 
\end{align}
For FitzHugh--Nagumo type equations, we choose the same sets, enriched by all
symbols obtained by substitutions $\RSI \to \RSoI$. 

Let $\cH_0 = \vspan(\cF_0)$ and $\cH_0^+ = \vspan(\cF_0^+)$. By definition, a
\emph{renormalisation transformation} is a linear map $M:\cH_0\to\cH_0$ such
that $MX^k = X^k$, $M\cI\tau = \cI M\tau$ for all $\tau\in\cF_0$ such that
$\cI\tau\in\cF_0$ and  $M\cE\tau = \cE M\tau$ for all $\tau\in\cF_0$. With $M$,
we would like to associate another admissible model $(\bPi^M,f^M)$ by
setting 
\begin{equation}
 \label{eq:rgen03}
 \bPi^M\tau = \bPi M \tau
 \qquad
 \forall \tau\in\cF_0\;.
\end{equation} 
As shown in~\cite[Section~8.3]{Hairer2014}, this is indeed possible if one
is able to find linear maps $\Delta^M:\cH_0\to\cH_0\otimes\cH_0^+$, and $\hat
M:\cH_0^+\to\cH_0^+$ such that 
\begin{align}
\nonumber
\Pi^M_z \tau &= (\Pi_z\otimes f_z)\Delta^M \tau\;, \\
\label{eq:rgen04} 
\pscal{f^M_z}{\bar\tau} &= \pscal{f_z}{\hat M\bar\tau}
\end{align}
holds for all $\tau\in\cF_0$ and all $\bar\tau\in\cF_0^+$, where $\Pi^M_z$ and
$f^M_z$ satisfy relations analogous to~\eqref{eq:rgen01}. 
The model $(\bPi^M,f^M)$ is admissible provided the conditions 
\begin{align}
\nonumber
\hat M \cJ_k &= \cM(\cJ_k\otimes\Id)\Delta^M\;, \\
(\Id\otimes\cM)(\Delta\otimes\Id)\Delta^M &= (M\otimes\hat M)\Delta 
\label{eq:rgen04b} 
\end{align}
hold, where $\cM$ is the multiplication map, defined by
$\cM(\tau\otimes\sigma)=\tau\sigma$.
The map $\hat M$
should also be a multiplicative morphism leaving the $X^k$ invariant, that is,  
\begin{equation}
 \label{eq:rgen05}
 \Mhat(\tau_1\tau_2) = (\Mhat\tau_1)(\Mhat\tau_2)\;, 
 \qquad
 \Mhat(X^k) = X^k\;.
\end{equation}
It is shown in~\cite[Prop.~8.36]{Hairer2014} that given a linear map $M$, there
is a unique choice of maps $\Delta^M$ and $\hat M$ such
that~\eqref{eq:rgen04b} and~\eqref{eq:rgen05} hold, making $(\bPi^M,f^M)$ an
admissible model. In addition, the elements $\Gamma^M_{z\bar z}$ can be
built using a map $\hat\Delta^M:\cH_0^+\to\cH_0^+\otimes\cH_0^+$ such that  
\begin{equation}
\label{eq:rgen05b} 
\gamma^M_{z\bar z} = (\gamma_{z\bar z}\otimes f_{\bar z}) \hat\Delta^M\;.
\end{equation}

\begin{definition}[{\cite[Def.~8.41]{Hairer2014}}]
\label{def:renorm_group}
The \emph{renormalisation group} $\mathfrak{R}$ is the set of linear maps
$M:\cH_0\to\cH_0$ satisfying $M\cI=\cI M$, $M\cE=\cE M$, $M(X^k)=X^k$, and such
that for all $\tau\in\cF_0$ and all $\bar\tau\in\cF_0^+$ one has 
\begin{align}
\nonumber
\Delta^M\tau - \tau\otimes\unit 
&\in \bigoplus_{\beta>\alpha}(T_\beta\cap\vspan(\cF_0))\otimes T^+\;, \\
\label{eq:def_RG} 
\hat\Delta^M\bar\tau - \bar\tau\otimes\unit 
&\in \bigoplus_{\beta>\alpha}(T_\beta\cap\vspan(\cF_0^+))\otimes T^+\;,
\end{align}
for all $\tau\in T_\alpha$ and all $\bar\tau\in T_\alpha\cap \vspan(\cF_0^+)$.
\end{definition}

It is shown in~\cite[Lemma~8.43]{Hairer2014} that $\mathfrak{R}$ is indeed
a group. The main result~\cite[Thm.~8.44]{Hairer2014} states that given
$M\in\mathfrak{R}$, and an admissible model $(\bPi,f)$, the model
$(\bPi^M,f^M)$ constructed as above is indeed an admissible
model. In particular, it satisfies the required analytic bounds and behaves well
under the extension theorem~\cite[Thm.~5.14]{Hairer2014}. 

Let us now introduce a particular subgroup of $\mathfrak{R}$ suited to our
problem. In the case of the Allen--Cahn equation
(cf~\cite[Sections~9.2 and~10.5]{Hairer2014}), a two-parameter group of
transformations $M=\exp\set{-C_1L_1 -C_2L_2}$ is sufficient, where the
generators $L_1$ and $L_2$ are defined by the substitution rules 
\begin{equation}
 L_1: \RSV \mapsto \unit\;, 
 \qquad
 L_2: \RSWV \mapsto \unit\;. 
 \label{eq:rn_substitutions} 
\end{equation}
More precisely, $L_1\tau$ is in general a sum of terms, where each term
corresponds to one occurrence of $\RSV=\cI(\Xi)^2$ in
$\tau$. Thus for instance $L_1(\RSW) = 3\RSI$, because there are $3$ ways of
picking a pair of $\RSI$ in $\RSW$, and replacing any of these pairs by $\unit$
yields $\RSI$.

In our case, we may \emph{a priori} have to apply similar substitutions for all
elements obtained from $\RSV$ and $\RSWV$ by replacing one or several $\cI(\Xi)$
by $\cE\cI(\Xi)$. We thus introduce the sets of symbols 
\begin{align}
\nonumber
\cF_1 &= \set{\RSV,\RSVo,\RSVoo}\;, \\
\cF_2 &= \setsuch{\cI(\tau_1)\tau_2}{\tau_1,\tau_2\in\cF_1}\;.
\label{eq:rgen07} 
\end{align}
Note that $\cF_2$ has $9$ elements. This leads us to define 
a $12$-parameter subgroup of $\mathfrak{R}$ given by 
\begin{equation}
 \label{eq:rgen08}
 M = \exp\Bigset{-\sum_{\tau\in\cF_1}C_1(\tau)L_\tau
 -\sum_{\tau\in\cF_2}C_2(\tau)L_\tau}\;,
\end{equation} 
where $L_\tau$ denotes the substitution $\tau\mapsto\unit$, applied as
often as possible. The exponential is then defined by its formal series, which
is always finite because any application of $L_\tau$ strictly decreases the
number of occurrences of $\cI(\Xi)$ or $\cE(\cI(\Xi))$. In practice,
however, we will see that only $2$ of these parameters are really needed
to renormalise the model. 

We introduce the shorthands $C_1=C_1(\RSV)$, $C_1'=C_1(\RSVo)$ and
$C_1''=C_1(\RSVoo)$. The action of $M$ on a few elements that will occur in the
computations below is given in the following list:
\begin{subequations}
\label{eq:M-AC} 
\begin{align}
\label{eq:M-AC-I}
M(\RSI) &= \RSI\;, \\
\label{eq:M-AC-V}
M(\RSV) &= \RSV - C_1\unit\;, \\
\label{eq:M-AC-W}
M(\RSW) &= \RSW - 3C_1\RSI\;, \\
M(\RSWV) &= \RSWV - C_2(\RSWV)\unit - C_1\RSY\;, %\\
\label{eq:M-AC-WV} \\
M(\RSWW) &= \RSWW - 3C_2(\RSWV)\RSI - 3C_1\RSWI - C_1\RSIW + 3C_1^2\RSII\;.
\label{eq:M-AC-WW}
\end{align}
\end{subequations}
The action of $M$ on elements such as $\RSVo$ is obtained by obvious
substitutions, e.g.\ we have $M(\RSWoo) = \RSWoo - C_1''\RSI - 2C_1'\RSoI$. As
in~\cite[Section~9.2]{Hairer2014}, one can determine the map
$\Delta^M$ and thus, via~\eqref{eq:rgen04}, the renormalised model $\Pi^M$.
Again, we just list a few expressions that will play a r\^ole in the
computations below:
\begin{subequations}
\label{eq:PiM-AC} 
\begin{align}
\label{eq:PiM-AC-I} 
\Pi^M_z(\RSI) &= \Pi_z(\RSI)\;, \\
\label{eq:PiM-AC-V} 
\Pi^M_z(\RSV) &= \Pi_z(\RSI)^2 - C_1\;, \\
\label{eq:PiM-AC-W} 
\Pi^M_z(\RSW) &= \Pi_z(\RSI)^3 - 3C_1\Pi_z(\RSI)\;, \\
\label{eq:PiM-AC-Y} 
\Pi^M_z(\RSY) &= \Pi_z(\RSY)\;, \\
\label{eq:PiM-AC-WV} 
\Pi^M_z(\RSWV) &= \Pi_z(\RSY)\Pi^M_z(\RSV) - C_2(\RSWV)\;, \\
\label{eq:PiM-AC-WW} 
\Pi^M_z(\RSWW) &= \Pi^M_z(\RSIW)\Pi^M_z(\RSV) - 3 C_2(\RSWV) \Pi^M_z(\RSI) + 
3C_1 \Pi_z(\RSV X_i) \pscal{f_z}{\cJ_{e_i}(\RSI)}\;.
\end{align}
\end{subequations}
Similar relations with obvious substitutions hold for the elements
obtained by having $\cE$ act on some $\cI(\Xi)$.

%%%%%%%%%%%%%%%%%%%%%%%%%%%%%%%%%%%%%%%%%%%%%%%%%%%%%%%%%%%%%%%%%%%%%%%%%%%%%%

\subsection{Convergence of the renormalised models}
\label{ssec_renorm_FHN}

From now on, we assume that the driving noise $\xi$ is Gaussian white noise on
$D=\R\times\T^d$, where $\T^d=\R^d/\Z^d$ is the $d$-dimensional torus .
This means that we are given a probability space $(\Omega,\cF,\fP)$, the Hilbert
space $H=L^2(D)$, and a collection of centred jointly Gaussian random
variables $\set{W_h}_{h\in H}$ which are assumed to satisfy 
\begin{equation}
 \label{eq:white_noise}
 \expec{W_hW_{\bar h}} = \pscal{h}{\bar h}
\end{equation} 
for all $h, \bar h\in H$. 
Then $\xi$ is the distribution defined by
$\pscal{\xi}{\varphi}=W_{\pi\varphi}$ for every test function $\varphi$, where
$\pi$ is the canonical projection from $L^2(\R^{d+1})$ to $L^2(D)$.

Let $Z^\eps=(\Pi^\eps,\Gamma^\eps)$ be the canonical model built for mollified
noise $\xi^\eps$. Our aim is now to find a specific sequence of renormalisation
maps $M_\eps$ in $\mathfrak{R}$ such that the sequence of models $\widehat
Z^\eps=M_\eps Z^\eps$ defined by 
\begin{equation}
 \label{eq:rFHN01}
 \widehat\Pi_z^{(\eps)} = (\Pi_z^\eps)^{M_\eps}
\end{equation} 
converges in a suitable sense to a limiting model $\widehat Z$. 

The relevant result to achieve this is Theorem~10.7 in~\cite{Hairer2014}. It
states that if for all localised scaled test functions $\eta^\delta_z =
\cS^\delta_{\fraks,z}\eta$ and all $\tau\in\cF_- =
\setsuch{\tau\in\cF_F}{\abss{\tau}} < 0$, one can find random
variables $\widehat \Pi_z\tau$ such that
\begin{align}
\nonumber
\E\bigabs{\pscal{\widehat \Pi_z\tau}{\eta^\delta_z}}^2 
&\lesssim \delta^{2\abss{\tau}+\kappa}\;, \\
\E\bigabs{\pscal{\widehat \Pi_z\tau - \widehat
\Pi_z^{(\eps)}\tau}{\eta^\delta_z}}^2 
&\lesssim \eps^{2\theta}\delta^{2\abss{\tau}+\kappa}
 \label{eq:rFHN02} 
\end{align}
holds for some $\kappa, \theta>0$, then there exists a unique admissible model
$\widehat Z$ such that for all compact sets $\fraK\subset\R^{d+1}$ and all
$p\geqs 1$ one has 
\begin{equation}
\E \normDgamma{\widehat Z}^p_{\gamma;\fraK} \lesssim 1\;, 
\qquad 
\E \seminormff{\widehat Z}{\widehat Z^\eps}^p_{\gamma;\fraK} \lesssim
\eps^{\theta p}\;.
 \label{eq:rFHN03} 
\end{equation} 
Here $\normDgamma{Z}_{\gamma;\fraK} = \norm{\Pi}_{\gamma;\fraK} +
\norm{\Gamma}_{\gamma;\fraK}$, and $\seminormff{Z}{\Zbar}_{\gamma;\fraK}$ 
is defined in Section \ref{ssec_fix_exist}. 
In~\eqref{eq:rFHN02}, each random variable $\widehat \Pi_z\tau$ should in
addition belong to the inhomogeneous Wiener chaos of order $\norm{\tau}$, where
$\norm{\tau}$ is the number of occurrences of $\Xi$ in $\tau$
(cf~\cite[Section~10.1]{Hairer2014}). More precisely, Wiener's isometry
allows one to define linear maps $I_k:H^{\otimes k}\to L^2$, as explained
in~\cite[Section~10.1]{Hairer2014}. Then the Wiener chaos expansion of
$\Pihat_0^{(\eps)}\tau$ is obtained by finding distributions
$(\widehat\cW^{(\eps;k)}\tau)(z)\in H^{\otimes k}$ such that for every
test function $\varphi$, 
\begin{equation}
 \label{eq:rFHN04}
 \pscal{\Pihat_0^{(\eps)}\tau}{\varphi}
 = \sum_{k\leqs\norm{\tau}}
 I_k \biggpar{\int_{\R^{d+1}} \varphi(z)(\widehat\cW^{(\eps;k)}\tau)(z)\6z}\;.
\end{equation} 
Let further $\cW^{(\eps;k)}$ be the distributions defined in a similar way from
the bare model $\Pi^\eps$.
As usual, we will express~\eqref{eq:rFHN04} by the slightly more informal
notation 
\begin{equation}
 \label{eq:rFHN04b}
 (\Pihat_0^{(\eps)}\tau)(z)
 = \sum_{k\leqs\norm{\tau}}
 I_k \bigpar{(\widehat\cW^{(\eps;k)}\tau)(z)}\;.
\end{equation} 
Proposition~10.11 in~\cite{Hairer2014} provides conditions on the
$\widehat\cW^{(\eps;k)}$ from which the required conditions~\eqref{eq:rFHN02}
follow. Namely, for each $\tau\in\cF_-$, one should find functions
$\widehat\cW^{(k)}$ such that 
\begin{align}
\nonumber
\Bigabs{\bigpscal{(\widehat\cW^{(k)}\tau)(z)}{(\widehat\cW^{(k)}\tau)(\bar z)}}
&\leqs C \sum_{\lambda\in B_{k,\tau}} \bigpar{\norm{z}_{\fraks} +
\norm{\bar z}_{\fraks}}^\lambda \norm{z-\bar
z}_{\fraks}^{\bar\kappa+2\abss{\tau}-\lambda}\;, \\
\Bigabs{\bigpscal{(\delta\widehat\cW^{(\eps;k)}\tau)(z)}
{(\delta\widehat\cW^{(\eps;k)} \tau)(\bar z)}}
&\leqs C\eps^{2\theta} \sum_{\lambda\in B_{k,\tau}}
\bigpar{\norm{z}_{\fraks}
+ \norm{\bar z}_{\fraks}}^\lambda \norm{z-\bar
z}_{\fraks}^{\bar\kappa+2\abss{\tau}-\lambda}
\label{eq:rFHN04c}
\end{align}
for some $\bar\kappa,\theta>0$, where
$\delta\widehat\cW^{(\eps;k)}=\widehat\cW^{(\eps;k)}-\widehat\cW^{(k)}$ and 
each $B_{k,\tau}$ is a finite set of indices
$\lambda\in[0,2\abss{\tau} +
\kappa + \abs{\fraks})$, where $\abs{\fraks}=d+2=5$. Furthermore, the scalar
product in~\eqref{eq:rFHN04c} is the canonical scalar product on $H^{\otimes
k}$ induced by the scalar product on $H$.

In order to obtain such bounds for our system, we can to a large extent
adapt the proof of~\cite[Thm.~10.22]{Hairer2014}. As a matter of fact, it is
immediate to see that this theorem can be extended to our case, because
the involved kernels have singularities which are not worse than in the
Allen--Cahn case. The aim of the discussion that follows is thus only to
determine which parameters of the renormalisation group introduced above are
really needed to ensure convergence. 

The first step is to observe that if $\varrho_\eps$ is the mollifyer, i.e., if
$\xi^\eps=\varrho_\eps*\xi$, noting $K_\eps=K*\varrho_\eps$ we have 
\begin{equation}
 \label{eq:rFHN05}
 (\Pi_{\bar z}^\eps \RSI)(z) = (K_\eps * \xi)(z)
 = \int K_\eps(z-z_1)\xi(z_1)\6z_1 =: \chi_\eps(z)\;, 
\end{equation} 
which is independent of $\bar z$ and belongs to the first Wiener chaos 
with 
\begin{equation}
 \label{eq:rFHN06} 
 \bigpar{(\widehat\cW^{(\eps;1)}\RSI)(z)}(z_1) = K_\eps(z-z_1)\;.
\end{equation} 
As a consequence, we simply set 
\begin{equation}
 \label{eq:rFHN06a}
 (\widehat \Pi_{\bar z}^{(\eps)}\RSI)(z) = \chi_\eps(z)\;, 
 \qquad
 (\widehat \Pi_{\bar z}\RSI)(z) = \chi_0(z) := (K*\xi)(z)\;.
\end{equation} 
Due to translation invariance, it will henceforth be sufficient to evaluate
models in $\bar z=0$. As already seen in the proof of
Proposition~\ref{prop:fixed_point_equiv}, we have 
\begin{equation}
\label{eq:rFHN05b}
 (\Pi_0^\eps \RSoI)(z) = (\KQ * \xi^\eps)(z) 
 = \int \KQ(z-z_1)\xi^\eps(z_1)\6z_1  =: \chiQ_\eps(z)\;.
\end{equation} 
This can be rewritten in the form 
\begin{equation}
 \label{eq:rFHN05c} 
 \chiQ_\eps(z) = (\KQ_\eps * \xi)(z)
 = \int \KQ_\eps(z-z_1)\xi(z_1)\6z_1\;, 
\end{equation}
where we have set 
\begin{equation}
 \label{eq:rFHN05d}
 \KQ_\eps(t,x) = \int_0^{2T} Q(s)K_\eps(t-s,x)\6s\;.
\end{equation} 
It follows that we can set 
\begin{equation}
 \label{eq:rFHN05e}
 (\widehat \Pi_0^{(\eps)}\RSoI)(z) = \chiQ_\eps(z)\;, 
 \qquad
 (\widehat \Pi_0\RSoI)(z) = \chiQ_0(z) := (\KQ*\xi)(z)\;.
\end{equation}
The second step is to consider products of terms $\RSI$ and $\RSoI$. To this
end, we observe that by definition of the canonical model (see~\eqref{eq:Pi_x}),
\begin{equation}
 \label{eq:rFHN07}
 (\Pi_0^\eps \RSV)(z) = \chi_\eps(z)^2 
 = \iint K_\eps(z-z_1)K_\eps(z-z_2)\xi(z_1)\xi(z_2) \6z_1\6z_2\;,
\end{equation} 
which belongs to the second inhomogeneous Wiener chaos. Renormalisation is
needed because the right-hand side is known to diverge in the limit $\eps\to0$.
However, if one introduces the Wick product 
\begin{equation}
 \label{eq:rFHN08}
 \xi(z_1)\diamond\xi(z_2) = \xi(z_1)\xi(z_2) - \delta(z_1-z_2)\;, 
\end{equation} 
then one obtains 
\begin{equation}
 \label{eq:rFHN09}
 (\Pi_0^\eps \RSV)(z) 
 = \iint K_\eps(z-z_1)K_\eps(z-z_2)\xi(z_1)\diamond\xi(z_2) \6z_1\6z_2 
 + \int K_\eps(z-z_1)^2\6z_1\;.
\end{equation} 
The first term on the right-hand side belongs to the second homogeneous Wiener
chaos, with 
\begin{equation}
 \label{eq:rFHN10} 
 \bigpar{(\widehat\cW^{(\eps;2)}\RSV)(z)}(z_1,z_2) =
K_\eps(z-z_1)K_\eps(z-z_2)\;,
\end{equation} 
and is known to converge as $\eps\to0$. We can thus define 
\begin{equation}
 \label{eq:rFHN11}
 (\widehat\Pi_0^{(\eps)} \RSV)(z) 
 = \iint K_\eps(z-z_1)K_\eps(z-z_2)\xi(z_1)\diamond\xi(z_2) \6z_1\6z_2
 =: \chi_\eps^{\diamond 2}(z)\;,
\end{equation} 
which is indeed compatible with~\eqref{eq:PiM-AC-V},
provided we set 
\begin{equation}
 \label{eq:rFHN12}
 C_1(\eps) = \int K_\eps(z-z_1)^2\6z_1 = \int K_\eps(-z_1)^2\6z_1\;.
\end{equation} 
The limiting model $(\widehat\Pi_0 \RSV)(z)$ is then naturally defined by the
same expression as in~\eqref{eq:rFHN11}, but with $K_\eps$ replaced by $K$.

In order to deal with the term $\RSW$, the relevant Wick product formula is 
\begin{align}
\nonumber
\xi(z_1)\diamond\xi(z_2)\diamond\xi(z_3) 
 {}={} & \xi(z_1)\xi(z_2)\xi(z_3) \\
 & {}- \xi(z_1)\delta(z_2-z_3) - \xi(z_2)\delta(z_3-z_1) 
 - \xi(z_3)\delta(z_1-z_2)\;.
\label{eq:rFHN13} 
\end{align} 
In view of~\eqref{eq:PiM-AC-W}, this is indeed compatible
with the definition 
\begin{equation}
 \label{eq:rFHN14}
 (\widehat\Pi_0^{(\eps)} \RSW)(z) 
 = \iiint K_\eps(z-z_1)K_\eps(z-z_2)K_\eps(z-z_3)
 \xi(z_1)\diamond\xi(z_2)\diamond\xi(z_3)
 \6z_1\6z_2\6z_3\;,
\end{equation} 
without the need to introduce a further renormalisation constant. 

Before being able to deal with terms such as $\RSWV$, we have to compute the
model associated with $\RSY$. Applying~\eqref{eq:Pi_xI} (in which the sum only
contains the term $\ell=0$) and the last relation in~\eqref{eq:f_x}, we obtain 
\begin{align}
\nonumber
 (\Pi_0^\eps \RSY)(z)
 &= \int \bigbrak{K(z-z_1) - K(-z_1)} (\Pi_0^\eps \RSV)(z_1)\6z_1 \\
\nonumber
 &= \int \bigbrak{K(z-z_1) - K(-z_1)} \chi_\eps^{\diamond 2}(z_1)\6z_1 \\
 &=: (K*\chi_\eps^{\diamond 2})(z) - (K*\chi_\eps^{\diamond 2})(0)\;,
 \label{eq:rFHN15}
\end{align} 
where we have used the fact that the integrals of $K(z-\cdot)$ and $K$ are
equal to cancel out the term $C_1$. It thus follows from~\eqref{eq:PiM-AC-WV}
that 
\begin{equation}
 \label{eq:rFHN16}
 (\widehat\Pi_0^{(\eps)} \RSWV)(z)
 = \bigbrak{(K*\chi_\eps^{\diamond 2})(z) - (K*\chi_\eps^{\diamond 2})(0)}
 \chi_\eps^{\diamond 2}(z) - C_2(\RSWV)\;.
\end{equation} 
Similar expressions hold for the other terms of $\cF_2$.
In order to compute the Wiener chaos decomposition of expressions such
as~\eqref{eq:rFHN16}, we will need the identity 
\begin{align}
\nonumber
\bigpar{\xi(z_1)\diamond\xi(z_2)} \bigpar{\xi(z_3)\diamond\xi(z_4)} 
 {}={} & \xi(z_1)\diamond\xi(z_2)\diamond\xi(z_3)\diamond\xi_4 \\
\nonumber
 & {}+ \xi(z_1)\diamond\xi(z_3)\delta(z_2-z_4) 
 + \xi(z_1)\diamond\xi(z_4)\delta(z_2-z_3)\\
\nonumber
 & {}+ \xi(z_2)\diamond\xi(z_3)\delta(z_1-z_4)
 + \xi(z_2)\diamond\xi(z_4)\delta(z_1-z_3)\\ 
 & {}+  \delta(z_1-z_3)\delta(z_2-z_4) + \delta(z_1-z_4)\delta(z_2-z_3)\;,
\label{eq:rFHN18} 
\end{align} 
which follows from~\cite[Lemma~10.3]{Hairer2014} and the Wick rule 
$I_k(f)\diamond I_\ell(g) = I_{k+\ell}(f\otimes g)$. 
It turns out that the important term regarding renormalisation is the
contribution to the zeroth Wiener chaos, resulting from the last line
in~\eqref{eq:rFHN18}, which is given in the case $\tau=\RSWV$ by 
\begin{align}
\nonumber
(\widehat \cW^{(\eps;0)}\RSWV)(z) 
{}={} & 2\iiint K(z-\bar z)K_\eps(\bar z-z_1)K_\eps(\bar z-z_2)
K_\eps(z-z_1)K_\eps(z-z_2)\6\bar z\6z_1\6z_2 \\
\nonumber
& {}+ 2\iiint K(-\bar z)K_\eps(\bar z-z_1)K_\eps(\bar z-z_2)
K_\eps(z-z_1)K_\eps(z-z_2)\6\bar z\6z_1\6z_2\\
& {}- C_2(\RSWV)\;.
\label{eq:rFHN19} 
\end{align}
As shown in~\cite[Thm.~10.22]{Hairer2014}, only the first integral on the
right-hand side (which is independent of $z$ by translation invariance) diverges
as $\eps\to0$, which imposes a choice for $C_2(\RSWV)$. In order to represent
this quantity and related ones appearing when dealing with the other terms, we
introduce the notations 
\begin{align}
\nonumber
\cQ_0^\eps(z) &= \int K_\eps(z_1)K_\eps(z_1-z)\6z_1\;, \\
\nonumber
\cQ_1^\eps(z) &= \int K_\eps(z_1)\KQ_\eps(z_1-z)\6z_1\;, \\
\cQ_2^\eps(z) &= \int \KQ_\eps(z_1)\KQ_\eps(z_1-z)\6z_1\;.
\label{eq:rFHN20} 
\end{align}
In particular, we have 
\begin{equation}
C_2(\RSWV)  = 2 \int K(z)\cQ_0^\eps(z)^2\6z\;. 
\label{eq:rFHN21}
\end{equation} 
Before proceeding to the convergence result, we collect a number of bounds on
integrals involving the kernels $K_\eps$ and $\KQ_\eps$. In what follows we will
sometimes write $K_0$ in place of $K$ as well as $\KQ_0$ in place of $\KQ$,
while $\abs{x}=\abs{x_1}+\abs{x_2}+\abs{x_3}$ denotes the $\ell^1$-norm of
$x\in\R^3$, so that $\norm{(t,x)}_{\fraks} = \abs{t}^{1/2}+\abs{x}$,
cf~\eqref{eq:parabolic_norm}. Note however that $\abs{x}$ and the Euclidean norm
$\norm{x}$ are equivalent. 

\begin{lemma}
\label{lem:KKQ}
The following bounds hold for all $z=(t,x)\in\R^4$ and all $\eps\in[0,1]$:
\begin{equation}
 \label{eq:KKQ01}
 \abs{K_\eps(z)} \lesssim \frac{1}{(\norm{z}_\fraks+\eps)^3}\;, 
 \qquad
 \abs{\KQ_\eps(z)} \lesssim \frac{1}{\abs{x}+\eps}\;, 
\end{equation} 
and 
\begin{equation}
 \label{eq:KKQ02}
 \int K_\eps(z)^2\6z \lesssim \frac1\eps\;, \qquad
 \int K_\eps(z)\KQ_\eps(z)\6z \lesssim 1\;, \qquad
 \int \KQ_\eps(z)^2\6z \lesssim 1\;. 
\end{equation} 
Furthermore, one has 
\begin{equation}
 \label{eq:KKQ03}
 \abs{\cQ_0^\eps(z)} \lesssim \frac{1}{\norm{z}_\fraks+\eps}\;, 
 \qquad
 \abs{\cQ_1^\eps(z)} \lesssim 1\;, 
 \qquad
 \abs{\cQ_2^\eps(z)} \lesssim 1\;. 
\end{equation}
Finally, for $i,j\in\set{0,1,2}$ let 
\begin{equation}
 \label{eq:KKQ04}
 I_{ij}(\eps) = \int K(z) \cQ_i^\eps(z)\cQ_j^\eps(z)\6z\;.  
\end{equation} 
Then $I_{00}(\eps) \asymp \abs{\log\eps}$, while all other $I_{ij}(\eps)$
are bounded uniformly in $\eps$. 
\end{lemma}

The proof is given in Appendix~\ref{app_proof_KKQ}. It is important to emphasize
that Lemma \ref{lem:KKQ} is the key result to show that certain terms do not
have to be renormalised. For example, consider the difference between the first
bound in \eqref{eq:KKQ02} and the second and third bounds. This explains why
certain symbols involving the $\cE$-operator do not have to be renormalised
while the same symbol without an additional $\cE$-operator has to be
renormalised; see also Appendix~\ref{app_proof_KKQ} for a detailed computation.

\begin{lemma}
\label{lem:KQeps}
For any $\eps,\theta\in(0,1]$ and any $(t,x)\in\R^4$, one has the bound  
\begin{equation}
\label{eq:KQeps01} 
  \bigabs{\KQ(t,x) - \KQ_\eps(t,x)}
 \lesssim \frac{\eps^\theta}{\abs{x}^{1+\theta}}\;.
\end{equation} 
As a consequence, 
\begin{equation}
\label{eq:KQeps02} 
\int \bigpar{\KQ_\eps(z-z_1) - \KQ(z-z_1)} 
\bigpar{\KQ_\eps(\bar z-z_1) - \KQ(\bar z-z_1)} \6z_1  
\lesssim \eps^{2\theta}
\end{equation}
holds for all  $\theta\in(0,\frac12)$ and all $z, \bar z\in\R^4$.
\end{lemma}

We give the proof in Appendix~\ref{app_proof_KQeps}. 

We can now state the main result of this section, which is an adaptation
of~\cite[Thm.~10.22]{Hairer2014}. 

\begin{prop}[Convergence of the renormalised models]
\label{prop:convergence_renorm}
There exists a random model $\Zhat$, independent of the choice of mollifier
$\varrho$, and a family of renormalisation transformations
$M_\eps\in\mathfrak{R}$ such that for any $\theta < -\frac52 -\alpha_0 =
\kappa$ (recall from~\eqref{eq:rs04} the definition of the noise regularity
$\alpha_0$), any compact set $\fraK$ and any $\gamma<\zeta$, one has 
\begin{equation}
 \label{eq:cvr01}
 \E \seminormff{M_\eps Z^\eps}{\Zhat}_{\gamma;\fraK} \lesssim \eps^\theta
\end{equation} 
uniformly over $\eps\in(0,1]$. Here $\zeta$ is such that~\eqref{eq:K02}
holds for all polynomials $P$ of parabolic degree less or equal $\zeta$. 
\end{prop}
\begin{proof}
As already stated, the proof follows along the lines
of~\cite[Thm.~10.22]{Hairer2014}, with some changes due to the presence of the
additional integration operator $\cE$. Thus we will only comment on those parts
of the proof that need to be adapted. 

The result will follow if we are able to choose the constants $C_1(\tau)$ and
$C_2(\tau)$ in~\eqref{eq:rgen08} for each $\eps\in(0,1]$ in such a way that the
bounds~\eqref{eq:rFHN02} hold for all $\tau\in\cF_-$; equivalently, the
bounds~\eqref{eq:rFHN04c} should hold for all these $\tau$. The limiting model
$\Zhat$ will be constructed step by step during this process.

The first step is to set 
\[
 (\Pihat_0\Xi)(z) = \xi(z)\;,
\]
and the fact that the required bounds are satisfied is a consequence of
Proposition~9.5 in~\cite{Hairer2014} (see also Lemma~10.2). 

In the case of $\tau=\RSI$, the limiting model is defined
by~\eqref{eq:rFHN06a}, and the conclusion follows as shown
in~\cite[Thm.~10.22]{Hairer2014} by applying the bounds on convolutions 
in~\cite[Lemmas~10.14 and~10.17]{Hairer2014}. We thus turn to the case
$\tau=\RSoI$. By~\eqref{eq:rFHN05e}, we have 
\begin{align*}
 \bigpar{(\widehat\cW^{(\eps;1)}\RSoI)(z)}(z_1) &= \KQ_\eps(z-z_1)\;, \\ 
 \bigpar{(\widehat\cW^{(1)}\RSoI)(z)}(z_1) &= \KQ(z-z_1)\;. 
\end{align*}
The first bound in~\eqref{eq:rFHN04c} then follows from the fact that 
\[
 \bigpscal{(\widehat\cW^{(1)}\RSoI)(z)}{(\widehat\cW^{(1)}\RSoI)(\bar z)}
 = \int \KQ(z-z_1)\KQ(\bar z-z_1)\6z_1 = \cQ^0_2(z-\bar z)
\]
(via the change of variables $z_1\mapsto z-z_1$). The right-hand side being
bounded as shown in Lemma~\ref{lem:KKQ}, the bound~\eqref{eq:rFHN04c} is
satisfied for $\lambda=0$ provided $\bar\kappa\leqs -2\abss{\RSoI} =
1+2\kappa$. In order to prove the second bound in~\eqref{eq:rFHN04c}, we use
the bound~\eqref{eq:KQeps02} in Lemma~\ref{lem:KQeps}, which yields 
\[
 \bigpscal{(\delta\widehat\cW^{(\eps;1)}\RSoI)(t,x)}
 {(\delta\widehat\cW^{(\eps;1)}\RSoI)(\bar t,\bar x)}
 \lesssim \eps^{2\theta}
\]
for any $\theta\in(0,\frac12)$.
Next we consider the term $\RSVo$. Then a decomposition analogous to the one
given in~\eqref{eq:rFHN09} shows that we have 
\begin{align*}
\bigpar{(\widehat\cW^{(\eps;2)}\RSVo)(z)}(z_1,z_2) &=
K_\eps(z-z_1)\KQ_\eps(z-z_2)\;, \\
(\widehat\cW^{(\eps;0)}\RSVo)(z) &=
\int K_\eps(-z_1)\KQ_\eps(-z_1)\6z_1 - C_1'(\eps)\;,
\end{align*}
where we have used the change of variables $z_1\mapsto z_1-z$.
Note that the second line is independent of $z$, as a consequence of
translation invariance, so that we can drop the argument $z$ from the notation.
The integral in the second line is equal to $\cQ^\eps_1(0)$, which is bounded
uniformly in $\eps$ by Lemma~\ref{lem:KKQ}. We may thus simply choose
$C_1'(\eps)=0$, which amounts to setting 
\begin{align*}
 (\Pihat_0^{(\eps)}\RSVo)(z) 
 &= \iint K_\eps(z-z_1)\KQ_\eps(z-z_2) \xi(z_1)\diamond\xi(z_2) \6z_1\6z_2 
 + \widehat\cW^{(\eps;0)}\RSVo \\
 &= \chi_\eps(z) \chiQ_\eps(z)\;.
\end{align*}
Indeed, the first bound in~\eqref{eq:rFHN04c} for $k=0$ is satisfied with
$\lambda=0$, provided $\bar\kappa\leqs-2\abss{\RSVo}=2+4\kappa$. Note that
being
allowed to set $C_1'(\eps)=0$ here is precisely the point where we use the fact
that certain terms do not have to be renormalised due to the bounds in
Lemma~\ref{lem:KKQ}. Of course, another possibility would be to set
$C_1'(\eps)=\widehat\cW^{(\eps;0)}\RSVo$, which would result in $\Pihat_0\RSVo$
belonging to the second \emph{homogeneous} Wiener chaos. Since $\cQ^\eps_1(0)$
is bounded, however, the choice of $C_1'(\eps)$ does not really matter, since it
only results in a shift of the random variable defining $\Zhat$. Whatever the
choice, the bounds~\eqref{eq:rFHN04c} can be checked by similar arguments as in
the previous case. 

The symbols $\RSVoo, \RSWo, \RSWoo, \RSWooo$ as well as $\RSVo X_i$ and $\RSVoo
X_i$ can be treated in pretty much the same way. It thus remains to consider
the symbols in the families of $\RSWV$, $\RSVW$ and $\RSWW$. As shown in the
proof of~\cite[Thm.~10.22]{Hairer2014}, the symbol $\RSVW$ already fulfils the
bounds~\eqref{eq:rFHN04c} as a consequence of the choice of $C_1(\eps)$. Since
the kernel $\KQ$ is less singular than $K$, the same conclusion holds for all
symbols of the same family. 

The symbol $\RSWV$ requires the introduction of a further renormalisation
constant $C_2(\RSWV)$ which has order $\abs{\log\eps}$. We thus have to check
what happens to the other symbols in the same family. Consider for instance the
case of $\RSWVo$. A computation similar to the one made in~\eqref{eq:rFHN15}
and~\eqref{eq:rFHN16} yields the expression 
\[
  (\widehat\Pi_0^\eps \RSWVo)(z)
 = \bigbrak{(K*(\chi_\eps\diamond\chiQ_\eps))(z) -
(K*(\chi_\eps\diamond\chiQ_\eps))(0)}
 \chi_\eps^{\diamond 2}(z) - C_2(\RSWVo)\;.
\]
Writing out the integrals and applying~\eqref{eq:rFHN18}, we can decompose this
expression into a sum of terms in the Wiener chaoses of order $0, 2$ and $4$.
In particular, the term in the $0$th Wiener chaos is given by 
\begin{align*}
(\widehat \cW^{(\eps;0)}\RSWVo)(z) 
{}={} & 2\iiint K(z-\bar z)K_\eps(\bar z-z_1)\KQ_\eps(\bar z-z_2)
K_\eps(z-z_1)K_\eps(z-z_2)\6\bar z\6z_1\6z_2 \\
& {}+ 2\iiint K(-\bar z)K_\eps(\bar z-z_1)\KQ_\eps(\bar z-z_2)
K_\eps(z-z_1)K_\eps(z-z_2)\6\bar z\6z_1\6z_2\\
& {}- C_2(\RSWVo)\;.
\end{align*}
The first integral on the right-hand side does not depend on $z$, as can be seen
by the change of variables $(\bar z,z_1,z_2)\mapsto(\bar z+z,z_1+z,z_2+z)$. As a
consequence, it is equal to $2I_{01}(\eps)$, and is thus uniformly bounded in
$\eps$ as shown in Lemma~\ref{lem:KKQ}. The second integral is also bounded, so
that we may choose $C_2(\RSWVo) = 0$ (or any other finite value). The terms in
the other two Wiener chaoses are already bounded in the case of $\RSWV$, so
they remain bounded in the present case. It is now quite obvious that the other
symbols of the family do not need to be renormalised either, as the relevant
integrals are obtained from the above ones by substituting the appropriate
kernels $K$ by $\KQ$. 

Finally, we have to deal with symbols in the family of $\RSWW$. However, it is
shown in the proof of~\cite[Thm.~10.22]{Hairer2014} that the constant
$C_2(\RSWV)$ suffices to renormalise $\RSWW$ as well, so the situation is
completely analogous for the other symbols of the family.  
\end{proof}

%%%%%%%%%%%%%%%%%%%%%%%%%%%%%%%%%%%%%%%%%%%%%%%%%%%%%%%%%%%%%%%%%%%%%%%%%%%%%%

\subsection{Computation of the renormalised equations}
\label{ssec_renorm_compute}

In order to complete the proof of the main results, it remains to derive the
classical equation with mollified noise that corresponds to the renormalised
models. In other words, for a given $M=M_\eps\in\mathfrak{R}$, we have to
characterise the solution map $\bar\cS^L_M$ satisfying 
\begin{equation}
\label{eq:rc01} 
 \cR^M \cS^L_M(u_0,v_0,MZ^\eps) = \bar\cS^L_M(u_0,v_0,\xi^\eps)\;,
\end{equation}
where $\cR^M$ is the reconstruction operator of the model
$\Pi^M$ ($=\widehat\Pi^{(\eps)}$) and $\cS^L_M$ is the map giving the fixed
point of the renormalised equation in terms of initial data and forcing term.
Relation~\eqref{eq:rc01} can also be represented as the following commutative
diagram:

\begin{center}
\begin{tikzpicture}[>=stealth']
\matrix (m) [matrix of math nodes, row sep=3em,
column sep=3.5em, text height=1.75ex, text depth=0.5ex]
{ (u_0,v_0,MZ^\eps) & (U^*_M,V^*_M) \\
 (u_0,v_0,\xi^\eps) & (\hat u^\eps,\hat v^\eps) \\};
\path[|->]
(m-1-1) edge node[auto] {$\cS_M^L$} (m-1-2)
(m-2-1) edge node[auto] {$M\Psi$} (m-1-1)
(m-1-2) edge node[auto] {$\cR^M$} (m-2-2)
(m-2-1) edge node[below] {$\bar\cS^L_M$} (m-2-2);
\end{tikzpicture}
\end{center}

In order to determine the classical equation, we will need the following
algebraic result.

\begin{lemma}
\label{lem:PizM}
For all $\tau\in\cF_0$, one has 
\begin{equation}
 \label{eq:rc02}
 (\Pi^M_z\tau)(z) = (\Pi_z M\tau)(z)\;.
\end{equation} 
\end{lemma}
\begin{proof}
The claim can be checked directly by comparing the expressions given
in~\eqref{eq:M-AC} and~\eqref{eq:PiM-AC}. More generally, the first relation
in~\eqref{eq:rgen04} shows that~\eqref{eq:rc02} holds for all $\tau$ such that
$\Delta^M(\tau) = M\tau\otimes\unit$, which happens to be the case for all
symbols excepts those of the families $\RSWV, \RSVW$ and $\RSWW$. In the
exceptional cases, the relation has to be checked by an explicit computation; in
general it holds only when the model is evaluated at the base point, i.e.\ we
may have $(\Pi^M_z\tau)(\bar z) \neq (\Pi_z M\tau)(\bar z)$ if $z\neq\bar z$.
For instance, in the case $\tau=\RSVoWoo$, the action of the renormalisation
map is given by 
\[
 M\RSVoWoo = \RSVoWoo -2C_1'\RSVoIo -C_1''\RSVoI\;.
\]
One can then check that in the case $\tau=\RSWoo$, \eqref{eq:rgen04b} is
satisfied by setting $\Delta^M\RSWoo=M\RSWoo\otimes\unit$ and $\hat
M\cJ(\RSWoo)=\cJ(M\RSWoo)$. Using this and the definition of the map $\Delta$,
a somewhat lengthy computation shows that 
\[
 \Delta^M \RSVoWoo = M\RSVoWoo\otimes\unit + 2C_1'\RSoI
X_i\otimes\cJ_i(\RSoI) + C_1''\RSoI X_i\otimes\cJ_i(\RSI)
\]
satisfies the second relation in~\eqref{eq:rgen04b} when $\tau=\RSVoWoo$. 
By~\eqref{eq:rgen04}, this implies 
\[
 \Pi_z^M\RSVoWoo = (\Pi_z\otimes f_z)\Delta^M\RSVoWoo
 = \Pi_z M\RSVoWoo + 2C_1'\Pi_z(\RSoI X_i)\pscal{f_z}{\cJ_i\RSoI}
 + C_1''\Pi_z(\RSoI X_i)\pscal{f_z}{\cJ_i\RSI}\;,
\]
where we used the fact that $\pscal{f_z}{\unit}=1$.
Since $(\Pi_z\RSoI X_i)(\bar z) = (\Pi_z\RSoI)(\bar z)(\bar x_i-x_i)$
by~\eqref{eq:Pi_x}, we see that the last two terms on the right-hand side indeed
vanish when $\bar z=z$.
\end{proof}

As a consequence of~\eqref{eq:rc02}, when $f$ is a function, as is the case for
$\eps>0$, we have 
\begin{equation}
 \label{eq:rc03}
 (\cR^M f)(z) = (\Pi^M_z f(z))(z) = (\Pi_z Mf(z))(z) = (\cR Mf)(z)\;.
\end{equation} 
This directly implies the following result, which is the equivalent of 
Proposition~\ref{prop:fixed_point_equiv} for the renormalised equation. 

\begin{prop}[Renormalised equation]
\label{prop:fixed_point_renorm}
Assume that $\eps>0$ and that~\eqref{eq:fix01} admits a fixed point $(U,V)$
where $U, V\in\cD^{\gamma,\eta}$ for some $\gamma, \eta\in\R$. Then
$(\hat u,\hat v)=(\cR^M U, \cR^M V)$ satisfies the classical SPDE 
\begin{align}
\nonumber
\partial_t \hat u &= \Delta \hat u + \widehat F(\hat u,\hat v) + \xi^\eps\;, \\
\partial_t \hat v &= A_1 \hat u + A_2 \hat v\;,
\label{eq:rc04}
\end{align} 
provided $\widehat F$ is such that  
\begin{equation}
\label{eq:rc05}
\widehat F(MU,MV) = MF(U,V) + \varrho(U,V)\;, 
\end{equation} 
with $\varrho$ having only components of strictly positive homogeneity. 
\end{prop}
\begin{proof}
Applying $M$ to the expansions of $U$ and $V$ given in~\eqref{eq:fix03}, we see
that $MU$ differs from $U$ and $MV$ differs from $V$ only by components of
homogeneity $\frac32-\kappa$ at least. Applying $\cR$ to $MU$ and $MV$ and
using~\eqref{eq:rc03}, we obtain that as in the proof
of Proposition~\ref{prop:fixed_point_equiv} 
\begin{align*}
\hat u(z) &= \chi_\eps(z) + \varphi(z)\;, \\
\hat v(z) &= \chi^Q_\eps(z) + \psi(z)\;.
\end{align*}
On the other hand, applying $M$ to the right-hand side of~\eqref{eq:fix01a} and
using the fact that $M$ commutes with $\cI$ and $\cE$ and leaves polynomials
invariant, we get 
\begin{align*}
 MU &= \cI\Rplus \bigbrak{\Xi + MF(U,V)} + \widehat R(MU,MV) + Gu_0\;,\\
 MV &= \cE\Rplus MU_- + \cQ\Rplus MU_+ + \widehat Qv_0\;. 
\end{align*}
Applying the reconstruction operator, we thus obtain 
\begin{align*}
 \hat u(t,x) &= G * \Rplus \bigbrak{\xi^\eps(t,x)+(\cR MF(U,V))(t,x)} +
Gu_0(t,x)\;, \\
 \hat v(t,x) &= \int_0^t Q(t-s)\hat v(s,x)\6s + \widehat Qv_0(t,x)\;.
\end{align*}
The claim then follows from the fact that~\eqref{eq:rc05} implies 
\[
(\cR MF(U,V))(z) 
= (\cR \widehat F(MU,MV))(z)
= \widehat F(\hat u(z),\hat v(z))\;, 
\]
where the last equality follows from the same argument as the one used in the
proof of Proposition~\ref{prop:fixed_point_equiv}. 
\end{proof}

Finally, we compute the expression of $\widehat F$. To simplify the
computations, we assume that $C_1=C_1(\RSV)$ and $C_2=C_2(\RSWV)$ are the
only non-zero renormalisation constants, which is a possible choice as seen in
the proof of Proposition~\ref{prop:convergence_renorm}. This is also the place
where the assumption $\gamma_2=0$ turns out to be necessary in order to obtain a
simple renormalised equation. 

\begin{prop}[Expression for the renormalised $F$]
\label{prop:Fhat}
Let $F$ be the cubic polynomial given in~\eqref{eq:fix02} with
either $d=2$, or $d=3$ and $\gamma_2=0$, and
let $M$ be the renormalisation transformation defined in~\eqref{eq:rgen08},
where the only nonzero constants are $C_1=C_1(\RSV)$ and $C_2=C_2(\RSWV)$. Then 
\begin{equation}
 \label{eq:cubic01}
 \widehat F(u,v) = 
 F(u,v) - c_0(\eps) - c_1(\eps)u {}- c_2(\eps)v\;,
% - c_2(\eps)v\;,
\end{equation}
where the coefficients are given by
\begin{align}
\nonumber
c_0(\eps) &= \beta_1 (C_1 + 3\gamma_1 C_2)\;, \\
\nonumber
c_1(\eps) &= 3\gamma_1 (C_1 + 3\gamma_1 C_2)\;, \\
c_2(\eps) &= \gamma_2 C_1\;.
\label{eq:cubic02}
\end{align}
\end{prop}
\begin{proof}
The fact that $(U,V)$ differs from $(MU,MV)$ only by terms of order
$\frac32-\kappa$ implies 
\[
 \widehat F(U,V) = MF(U,V) + \varrho'(U,V)\;,
\]
where all components of $\varrho'(U,V)$ have strictly positive homogeneity.
Thus we only have to compute the difference $MF(U,V) - F(U,V)$, which
reduces to determining the terms of $F(U,V)$ which are modified by the
renormalisation map, or, more precisely, which terms of homogeneity less or
equal $0$ appear when applying $M$. 

The computation is straightforward, though somewhat lengthy. In practice, it
suffices to expand the quadratic terms of $F(U,V)$ up to order
$\frac12-\kappa$, and cubic terms up to order $0-2\kappa$. To explain where the
condition $\gamma_2=0$ arises if $d=3$, we allow for general $\gamma_2\in\R$ for
the time being. 

Consider first the quadratic part $F_2(U,V) = \beta_1U^2 + \beta_2UV +
\beta_3V^2$ of $F(U,V)$. It turns out that when the renormalisation map $M$ is
applied, only the term proportional to $\RSV$ of $F_2(U,V)$ yields a
contribution of not strictly positive homogeneity. Namely, one has 
\[
 M F_2(U,V) - F_2(U,V) =
-\beta_1C_1\unit
+\varrho_2(U,V)\;,
\]
where $\varrho_2(U,V)$ has strictly positive homogeneity.
It remains to determine the contribution of the cubic part $F_3(U,V) =
\gamma_1U^3+\gamma_2U^2V+\gamma_3UV^2+\gamma_4V^3$. Substituting $U$ and $V$ by
their expansion~\eqref{eq:fix03} and sorting terms by homogeneity yields an
expression of the form 
\[
 F_3(U,V) = F_{3;-\frac32}(U,V) + F_{3;-1}(U,V) + F_{3;-\frac12}(U,V) +
F_{3;0}(U,V) + \varrho_3(U,V)\;, 
\]
where each $F_{3;\beta}$ contains only terms of homogeneity
$\beta-\Order{\kappa}$ and the remainder $\varrho_3$ has strictly positive
homogeneity. We treat all of these terms separately. For instance, we have 
$F_{3;-\frac32}(U,V) =
\gamma_1\RSW+\gamma_2\RSWo+\gamma_3\RSWoo+\gamma_4\RSWoo$, where the
first two terms have to be renormalised, yielding 
\[
 M F_{3;-\frac32}(U,V) - F_{3;-\frac32}(U,V) =
-3\gamma_1C_1\RSI 
-\gamma_2C_1\RSoI
+\varrho_{3;-\frac32}(U,V)\;.
\]
The term $F_{3;-1}(U,V)$ yields a contribution 
\[
 M F_{3;-1}(U,V) - F_{3;-1}(U,V) =
-3\gamma_1C_1\varphi\unit 
-\gamma_2C_1\psi\unit
+\varrho_{3;-1}(U,V)
\]
stemming from the terms $\RSV$. 
Consider next the term $F_{3;0}(U,V)$. It contains $2$ terms proportional to 
$\RSWV$, whose renormalisation yields the contribution 
\[
 M F_{3;0}(U,V) - F_{3;0}(U,V) = -\bigbrak{3\gamma_1 b_1 
 + \gamma_2 \hat b_1}C_2 \unit+\varrho_{3;0}(U,V)\;. 
\]
Finally, in the term $F_{3;-\frac12}(U,V)$, it is the terms in 
$\RSWW$ and $\RSWWo$ that need to be renormalised. One obtains 
\begin{align*}
 M F_{3;-\frac12}(U,V) - F_{3;-\frac12}(U,V) ={}&
-\bigbrak{9\gamma_1a_1 + 3 \gamma_2 \hat a_1}C_2\RSI \\
&{}-\bigbrak{9\gamma_1a_2 + 3\gamma_2\hat a_2}C_2\RSoI
+\varrho_{3;-\frac12}(U,V)\;.
\end{align*}
In the case $\gamma_2=0$, using $a_i=\gamma_i$ and replacing the $b_i$ by their
expressions~\eqref{eq:fix04}, and adding the last four expressions, we see that
we can always factor out the expression $\varphi\unit+\RSI$, which coincides
with $U$ up to terms of strictly positive homogeneity. 
The result follows. 

In case $\gamma_2\neq 0$, there are additional terms such as $\gamma_2 C_2 \hat
b_1\unit$, which do not admit a simple expression in terms of $U$ and $V$, so
that the renormalised equation has no simple closed-form expression. 
However when $d=2$, then we can choose $C_2=0$ because the symbol $\RSWV$
has strictly positive homogeneity and does not have to be renormalised. Then we
can also factor out the quantity $\psi\unit+\RSoI$, yielding the expression
for $c_2(\eps)$ in~\eqref{eq:cubic02}.
\end{proof}

%%%%%%%%%%%%%%%%%%%%%%%%%%%%%%%%%%%%%%%%%%%%%%%%%%%%%%%%%%%%%%%%%%%%%%%%%%%%%%

\section{Proof of the main results}
\label{sec_proof}

Since Theorem~\ref{thm:FHN} is a particular case of Theorem~\ref{thm:cubic}, we
proceed directly to proving the latter result. 

\begin{proof}[Proof of Theorem~\ref{thm:cubic}]
Consider first the situation on a fixed, sufficiently small time interval
$[0,T]$. Fix $\eps>0$ and let 
\begin{align*}
(U^\eps,V^\eps) &= \cS_T(u_0,v_0,M_\eps Z^\eps)\;, \\
(U,V) &= \cS_T(u_0,v_0,\Zhat)\;,
\end{align*}
where $\cS_T$ is the solution map given in Proposition~\ref{prop:fixed_point}
(we have suppressed the dependence on $W$, which is the same for both
fixed points). For any $\delta>0$, the bound~\eqref{eq:prop_fixed_point02},
Markov's inequality and Proposition~\ref{prop:convergence_renorm} imply 
\[
 \Bigprob{\seminormff{U^\eps}{U}_{\gamma,\eta;T} + 
 \seminormff{V^\eps}{V}_{\gamma,\eta;T} > C_0\delta}
 \leqs \frac{1}{\delta} \E \seminormff{M_\eps Z^\eps}{\Zhat}_{2\gamma+\alpha;O}
 \lesssim \frac{\eps^\theta}{\delta}\;.
\]
It follows that $(U^\eps,V^\eps)$ converges to $(U,V)$ in probability as
$\eps\to0$. 

Next, let $(u,v) = (\cR U, \cR V)$ and $(u^\eps,v^\eps) = (\cR U^\eps, \cR
V^\eps)$ (where we do not indicate the model-dependence of the reconstruction
operators $\cR$). The fact that $(u^\eps,v^\eps)$ converges to $(u,v)$ in
probability as $\eps\to0$ is a consequence of the
bound~\cite[(3.4)]{Hairer2014} in the reconstruction theorem, combined with the
definitions~\eqref{eq:defmod1} of $\norm{\Pi}_{\gamma;\fraK}$
and~\eqref{eq:norm_gamma_eta2} of $\normDgamma{\cdot}_{\gamma,\eta;\fraK}$. 

The convergence result can be extended from a fixed time interval $[0,T]$ to
the random time interval up to the first exit from a ball of radius $L$ in
exactly the same way as in~\cite[Section~7.3]{Hairer2014}. In particular, the
required continuous dependence on the model is proved
in~\cite[Cor.~7.12]{Hairer2014}. 

Proposition~\ref{prop:fixed_point_renorm} shows that $(u^\eps,v^\eps)$
satisfies the system~\eqref{eq:rcubic03}, where the explicit expression of
$\widehat F$ is a consequence of Proposition~\ref{prop:Fhat}. The
expressions~\eqref{eq:FHN05} and~\eqref{eq:rcubic04} for the renormalisation
constants in dimension $d=3$ follow from~\eqref{eq:cubic02}, where the
only nonzero terms are $C_1(\eps)$, defined in~\eqref{eq:rFHN12}, and
$C_2(\eps)$, defined in~\eqref{eq:rFHN21}. 

In the case $d=2$, Table~\ref{tab:FF_AllenCahn} shows that only the terms
$\RSI$, $\RSV$ and $\RSW$ need to be renormalised, to that the
expression~\eqref{eq:FHN07} for the renormalisation constant follows by taking
$C_2(\eps)=0$ in~\eqref{eq:cubic02}. 
\end{proof}

\begin{proof}[Proof of Theorem~\ref{thm:vectorial}]
The only difference with the previous case is that we have to take into account
the presence of the $n$ different operators $\cE_i$, $i=1,\dots,n$, introduced
in Remark~\ref{rem:E_vectorial}. This means that the fixed-point equation for
$V$ in~\eqref{eq:fix01} is replaced by $n$ different equations of the form 
\[
 V_i = \cK^{Q_i}_\gamma \Rplus U + \widehat Q_iv_0\;, 
\qquad i=1,\dots,n\;.
\]
Since $n$ is finite, it is straightforward to check that the proofs of
Proposition~\ref{prop:fixed_point_equiv} and Proposition~\ref{prop:fixed_point}
carry over to this situation. 

It remains to check the expressions~\eqref{eq:rvect05} for the renormalisation
constants. To compute the diverging part of $\widehat F(U,V) - F(U,V)$, it is
in fact sufficient to determine the coefficients of $\RSV$, $\RSWV$ and $\RSWW$
in $F(U,V)$, since these are the only ones that need to be renormalised. These
can only result from the terms $\beta_1 U^2 + \gamma_1 U^3$ in $F(U,V)$, since
all other terms will contain at least one of the $\cE_i$. It is now sufficient
to check that $U$ has the same expansion as in~\eqref{eq:fix03}, except that the
first relation in~\eqref{eq:fix04} becomes 
\[
 b_1 = \beta_1 + 3\varphi\gamma_1 + \sum_{i=1}^n \psi_i \gamma_{2,i}\;,
\]
where $\psi_i$ is the coefficient of $\unit$ in $V_i$. The result then follows 
as in the proof of Proposition~\ref{prop:Fhat}, taking into account the
modified value of $b_1$. 
\end{proof}

%%%%%%%%%%%%%%%%%%%%%%%%%%%%%%%%%%%%%%%%%%%%%%%%%%%%%%%%%%%%%%%%%%%%%%%%%%%%%%

\appendix

%%%%%%%%%%%%%%%%%%%%%%%%%%%%%%%%%%%%%%%%%%%%%%%%%%%%%%%%%%%%%%%%%%%%%%%%%%%%%%

\section{Proof of Lemma~\ref{lem_trans_inv}}
\label{app_proof_translation}

The proof is by induction on the size of the regularity structure, which is
constructed recursively, starting with $\Xi$, $\unit$ and the $X_i$, and
applying the operators $\cI$ and $\cE$ as well as multiplication to already
existing symbols. We show that the properties~\eqref{eq:translation}
hold on the polynomial regularity structure and for $\tau=\Xi$, and that they
are stable under the algebraic operations extending this structure. 

\begin{enum}
\item 	The claim certainly holds on the polynomial part of
the regularity structure, since 
\begin{align*}
(\Pi^\eps_{z+h}X^k)(\bar z+h) &= (\bar z+h-z-h)^k = (\Pi^\eps_z X^k)(\bar z)\;,
\\
\Gamma^\eps_{z+h,\bar z+h}X^k &= (X-\bar z-h+z+h)^k = \Gamma^\eps_{z\bar
z}X^k\;, 
\end{align*}
and 
\begin{align*}
 \Gamma^\eps_{z+h,\bar z+h}X^k &= (\Id\otimes\gamma^\eps_{z+h,\bar z+h})\Delta
X^k
 = X^k + \unit\pscal{\gamma^\eps_{z+h,\bar z+h}}{X^k}\;, \\
 \Gamma^\eps_{z,\bar z}X^k &= (\Id\otimes\gamma^\eps_{z,\bar z})\Delta X^k
 = X^k + \unit\pscal{\gamma^\eps_{z,\bar z}}{X^k}\;.
\end{align*}

\item 	Using the fact that $\Delta(\Xi)=\Xi\otimes\unit$, it is also immediate
to check that~\eqref{eq:translation} holds for $\tau=\Xi$ (this reflects
translation invariance of the noise).

\item	We show that if the relations~\eqref{eq:translation} hold for $\tau$,
then they also hold with $\tau$ and $\bar\tau$ replaced by $\cI\tau$ and
$\cJ_k\tau$. Indeed, 
\begin{align*}
\pscal{f^\eps_{z+h}}{\cJ_k\tau} 
&= -\int D^k K(z+h-\bar z) (\Pi^\eps_{z+h}\tau)(\bar z)\6\bar z \\
&= -\int D^k K(z-\tilde z) (\Pi^\eps_{z+h}\tau)(\tilde z+h)\6\tilde z \\
&= -\int D^k K(z-\tilde z) (\Pi^\eps_z\tau)(\tilde z)\6\tilde z \\
&= \pscal{f^\eps_z}{\cJ_k\tau}\;,
\end{align*}
and 
\begin{align*}
(\Pi^\eps_{z+h}\cI\tau)(\bar z+h) 
&= \int K(\bar z+h-z') (\Pi^\eps_{z+h}\tau)(z')\6z' +
\sum_\ell \frac{(\bar z-z)^\ell}{\ell!}\pscal{f^\eps_{z+h}}{\cJ_\ell\tau} \\
&= \int K(\bar z-\tilde z) (\Pi^\eps_z\tau)(\tilde z)\6\tilde z +
\sum_\ell \frac{(\bar z-z)^\ell}{\ell!}\pscal{f^\eps_z}{\cJ_\ell\tau} \\
&= (\Pi^\eps_z\cI\tau)(\bar z)\;.
\end{align*}
Furthermore, it is shown in the proof of~\cite[Thm.~8.24]{Hairer2014} that  
\[
 \Gamma^\eps_{z+h,\bar z+h}(\cI\tau) = 
 (\Id\otimes\gamma^\eps_{z+h,\bar z+h})(\cI\otimes\Id)\Delta\tau
 + \sum_\ell \frac{(X-z_\gamma)^\ell}{\ell!} \pscal{\gamma^\eps_{z+h,\bar
z+h}}{\cJ_\ell\tau}\;,
\]
where $z_\gamma\in\R^{d+1}$ has components
$(z_\gamma)_i=-\pscal{\gamma^\eps_{z+h,\bar z+h}}{X_i}$. 
The first term on the right-hand side is equal to 
$\cI \Gamma^\eps_{z+h,\bar z+h}(\tau)$. 
The induction hypothesis implies that all terms are equal to their value for
$h=0$, proving the required identity for $\Gamma^\eps_{z\bar z}(\cI\tau)$. The
identity for $\gamma^\eps_{z\bar z}(\cJ_k\tau)$ follows in a similar manner. 

\item 	In a similar way, using the definition~\eqref{eq:E01} of
$\Pi^\eps_z(\cE\tau)$, it is straightforward to check that if $\tau$
satisfies~\eqref{eq:translation}, then $\cE\tau$
satisfies~\eqref{eq:translation} as well. 

\item 	If $\tau,\sigma$ satisfy~\eqref{eq:translation}, then 
\[
 (\Pi^\eps_{z+h}\tau\sigma)(\bar z+h) 
 = (\Pi^\eps_{z+h}\tau)(\bar z+h)(\Pi^\eps_{z+h}\sigma)(\bar z+h)
 = (\Pi^\eps_z\tau)(\bar z)(\Pi^\eps_z\sigma)(\bar z)
 = (\Pi^\eps_z\tau\sigma)(\bar z)\;.
\]
Furthermore, 
\[
 \Gamma^\eps_{z+h,\bar z+h}(\tau\sigma) 
 = (\Id\otimes\gamma^\eps_{z+h,\bar z+h})\Delta(\tau\sigma)
 = (\Id\otimes\gamma^\eps_{z\bar z})\Delta(\tau\sigma)
 = \Gamma^\eps_{z\bar z}(\tau\sigma)\;.
\]
Here we have used the fact that the induction hypothesis applies because
the second component of $\Delta(\tau)$ always depends only on already computed
quantities, cf.~Table~1. The translation invariance of $\gamma^\eps_{z\bar
z}(\tau\sigma)$ follows. 
\qed
\end{enum}

%%%%%%%%%%%%%%%%%%%%%%%%%%%%%%%%%%%%%%%%%%%%%%%%%%%%%%%%%%%%%%%%%%%%%%%%%%%%%%

\section{Proof of Lemma~\ref{lem:KKQ}}
\label{app_proof_KKQ}

It follows from the definition of $K$ that it satisfies $\abs{K(z)}\lesssim
\norm{z}_{\fraks}^{-3}$. Computing the derivatives of $K$, one sees that
it is singular of order $-3$ in the sense of~\cite[Def.~10.12]{Hairer2014}.
Thus the first bound in~\eqref{eq:KKQ01} follows
from~\cite[Lemma~10.17]{Hairer2014}. 

The second bound in~\eqref{eq:KKQ01} is a crucial estimate, 
which allows us to avoid renormalisation of certain symbols involving the 
operator $\cE$. Therefore, we provide additional details in the proof of 
this bound. Recalling the definitions of $K^Q$ and $Q$ we have 
\[
K^Q(t,x)=\int_\R Q(s,x)K_\eps(t-s,x)\6s = \int_0^{2T\wedge (t+\eps^2)} 
Q(s,x)K_\eps(t-s,x)\6s
\]
since $K_\eps(t,x)$ is supported on $\set{t>-\eps^2}$. By a change of
variable $s\mapsto t-s$ in the last integral and using the boundedness of $Q$ on $[0,T]$
for a final time $T>0$, it follows that 
\[
K^Q(t,x)\asymp  \int_{-\eps^2}^{t} K_\eps(s,x)\6s,
\]
where we recall that the notation $f\asymp g$ indicates that we have both 
$f\lesssim g$ and $g\lesssim f$. Therefore, combining the last characterisation 
of $K^Q(t,x)$ with the first bound in~\eqref{eq:KKQ01} yields
\begin{align}
\nonumber
\abs{\KQ_\eps(t,x)} 
 &\lesssim \int_{-\eps^2}^t \frac{1}{(\abs{s}^{1/2}+\abs{x}+\eps)^{3}} \6s \\
\nonumber
 &\lesssim \int_{-\eps^2}^{(\abs{x}+\eps)^2}
 \frac{1}{(\abs{x}+\eps)^{3}} \6s
 + \int_{(\abs{x}+\eps)^2}^t \frac{1}{\abs{s}^{3/2}} \6s \\
 &\lesssim \frac{1}{\abs{x}+\eps}\;.
 \nonumber
\end{align} 
and this concludes the proof of the second bound in~\eqref{eq:KKQ01}.
The first bound in~\eqref{eq:KKQ02} follows from 
\[
 \int K_\eps(z)^2 \6z 
 \lesssim \int \frac{1}{(\norm{z_1}_\fraks+\eps)^6}\6z_1 
\asymp \int_0^1 \int_0^1 \frac{r^2\6r}{(t^{1/2}+r+\eps)^6} \6t 
\asymp \frac{1}{\eps}\;. 
\]
The second bound in~\eqref{eq:KKQ02} is again important, since it leads to
terms which involve $\cE$ but which do not have to be renormalised, so we
provide more details in the proof of this bound. Note that 
\begin{align*}
 \int K_\eps(t,x)\KQ_\eps(t,x) \6x 
 &\lesssim \int \frac{1}{(|t|^{1/2}+|x|+\eps)^3}\frac{1}{(|x|+\eps)} \6 x\;, \\
 &\lesssim \int \frac{1}{(|t|^{1/2}+\|x\|+\eps)^3}\frac{1}{(\|x\|+\eps)} \6 x\; \\
 &\lesssim \iint \frac{1}{(|t|^{1/2}+r+\eps)^3}\frac{1}{(r+\eps)} \sin(\theta_1)
r^2 \6 \theta\6 r   \;,
\end{align*}
using the equivalence of norms and spherical coordinates $(r,\theta)=(r,\theta_1,\theta_2)$
in $\R^3$. Therefore, we find the estimate 
\[
 \int K_\eps(t,x)\KQ_\eps(t,x) \6x 
 \lesssim \int_0^1 \frac{r^2\6r}{(r+\abs{t}^{1/2}+\eps)^3(r+\eps)}
 \asymp \frac{1}{\abs{t}^{1/2}+\eps}\; 
\]
Integrating over time then yields
\[
 \iint K_\eps(t,x)\KQ_\eps(t,x) \6x \6t 
 \lesssim \int_{-\eps^2}^1 \frac{1}{\abs{t}^{1/2}+\eps} \6t 
 \lesssim 1\;.
\]
which proves the second bound in~\eqref{eq:KKQ02}.
The third bound in~\eqref{eq:KKQ03} follows from the fact that 
\[
 \int \KQ_\eps(t,x)^2 \6x 
 \lesssim \int_0^1 \frac{r^2\6r}{(r+\eps)^2} \asymp 1\;.
\]
To obtain the bound on $\cQ_0$, we distinguish the regimes
$\norm{z}_{\fraks}\leqs\eps$ and $\norm{z}_{\fraks}>\eps$. If
$\norm{z}_\fraks\leqs\eps$, the same computation as the one made for the
integral of $K_\eps(z)^2$ shows that $\abs{\cQ_0(z)}\lesssim 1/\eps$, while for
$\norm{z}_\fraks>\eps$, we can use~\cite[Lemma~10.14]{Hairer2014} to
obtain 
\[
 \abs{\cQ_0(z)} \lesssim \int \frac{1}{\norm{z_1}_\fraks^3}
\frac{1}{\norm{z-z_1}_\fraks^3}\6z_1 
\lesssim \frac{1}{\norm{z}_\fraks}\;.
\]
Using spherical coordinates, one can show that for all $a,b\geqs0$ and
all $n\geqs1$, one has 
\[
 I_n := \int \frac{1}{(\abs{x_1}+a)^n} \frac{1}{\abs{x-x_1}+b} \6x_1 
 \asymp \int_0^1 \frac{r^2\6r}{(r+a)^n(r+\abs{x}+b)}\;.
\]
From this we easily get that whenever $a\geqs b\geqs0$,  
\[
 I_1 \lesssim 1 + \abs{x} + a^2\;, \qquad 
 I_3 \lesssim \frac{1+\abs{\log a}}{\abs{x}+a}\;,
\] 
and thus 
\begin{align*}
 \abs{\cQ_1^\eps(t,x)} &\lesssim \int_0^1
 \frac{1+\abs{\log(t^{1/2}+\eps)}}{\abs{x}+t^{1/2}+\eps} \6t \lesssim 1\;, \\
 \abs{\cQ_2^\eps(t,x)} &\lesssim 1+\abs{x}+\eps^2 \lesssim 1\;.
\end{align*} 
It remains to bound the terms $I_{ij}(\eps)$. It is in fact
sufficient to check that 
\begin{align*}
 I_{00}(\eps) 
 &\lesssim \int_0^1\int_0^1
\frac{r^2\6r}{(t^{1/2}+r)^3(t^{1/2}+r+\eps)^2}\6t
\lesssim \abs{\log\eps}\;, \\
 I_{01}(\eps) 
 &\lesssim \int_0^1\int_0^1
\frac{r^2\6r}{(t^{1/2}+r)^3(t^{1/2}+r+\eps)}\6t
\lesssim 1\;, 
\end{align*}
since the other terms can be bounded above by $I_{01}(\eps)$. Elementary
computations show that these bounds indeed hold true.
\qed

%%%%%%%%%%%%%%%%%%%%%%%%%%%%%%%%%%%%%%%%%%%%%%%%%%%%%%%%%%%%%%%%%%%%%%%%%%%%%%

\section{Proof of Lemma~\ref{lem:KQeps}}
\label{app_proof_KQeps}

Let us recall Definition~10.12 in~\cite{Hairer2014}. Given a scaling
$\fraks$, a smooth function $K: \R^{d+1}\setminus\set{0}\to\R$ is said to be
(singular) of order $\zeta$ if for every sufficiently small multiindex $k$, one
has 
\[
 \abs{D^k K(z)} \lesssim \norm{z}_\fraks^{\zeta - \abs{k}_\fraks}
 \qquad
 \forall z \text{ such that } \norm{z}_\fraks \leqs 1\;.
\]
More precisely, we require that there exists $m\geqs0$ such that 
\[
 \normDgamma{K}_{\zeta;m} := \sup_{k\colon\abs{k}_\fraks \leqs m} 
 \sup_{z\in\R^{d+1}} \norm{z}_\fraks^{\abs{k}_\fraks - \zeta} \abs{D^k K(z)} <
\infty\;.
\]
Note that in~\cite[Def~10.12]{Hairer2014}, the second supremum indeed runs over
the whole space, but in practice we will only apply it to compactly supported
functions. 
Lemma~10.17 in~\cite{Hairer2014} states that if $K$ is of order
$\zeta\in(-\abs{\fraks},0)$, then $K_\eps=K*\varrho_\eps$ has  bounded
derivatives of all orders, satisfying 
\[
 \abs{D^k K_\eps(z)} \lesssim (\norm{z}_\fraks + \eps)^{\zeta-\abs{k}_\fraks}
\normDgamma{K}_{\zeta;\abs{k}_\fraks}\;.
\]
Furthermore, for all $\bar\zeta\in[\zeta-1,\zeta)$ and $m\geqs0$, one has 
\begin{equation}
 \label{eq:K*eps2}
 \normDgamma{K-K_\eps}_{\bar\zeta;m} 
 \lesssim \eps^{\zeta-\bar\zeta} \normDgamma{K}_{\zeta;m+\max\set{\fraks_i}}\;. 
\end{equation} 
Consider now the quantity 
\[
 \KQ(t,x) - \KQ_\eps(t,x) 
 = \int_0^t \bigbrak{K(s,x) - K_\eps(s,x)}Q(t-s)\6s 
\]
(in fact, the lower bound can be replaced by $0\vee(t-2T)$). We have 
\[
 \bigabs{\KQ(t,x) - \KQ_\eps(t,x)}
 \lesssim \int_0^t \bigabs{K(s,x) - K_\eps(s,x)}\6s\;. 
\]
We already know that $K$ is singular of order $-3$, and thus 
$\normDgamma{K}_{-3;m}$ is bounded for any $m$. In particular, we have
by~\eqref{eq:K*eps2} with $m=0$ the bound 
\[
 \normDgamma{K-K_\eps}_{\bar\zeta;0} 
 \lesssim \eps^{-3-\bar\zeta} \normDgamma{K}_{-3;2} 
\]
for any $\bar\zeta\in[-4,-3)$, 
which implies 
\[
 \bigabs{K(s,x) - K_\eps(s,x)}
 \leqs \frac{\normDgamma{K-K_\eps}_{\bar\zeta;0}}{\norm{z}_\fraks^{-\bar\zeta}}
 \lesssim \frac{\eps^{-3-\bar\zeta}}{\norm{z}_\fraks^{-\bar\zeta}} 
 \normDgamma{K}_{-3;2}\;.
\]
Taking $\bar\zeta=-3-\theta$ for some $\theta\in(0,1]$, we obtain 
\[
 \bigabs{K(s,x) - K_\eps(s,x)}
 \lesssim \frac{\eps^\theta}{\norm{z}_\fraks^{3+\theta}} 
 \normDgamma{K}_{-3;2}\;,
\]
and thus, since $\normDgamma{K}_{-3;2}\lesssim 1$, 
\[
 \bigabs{\KQ(t,x) - \KQ_\eps(t,x)}
 \lesssim \eps^\theta\int_0^t \frac{\6s}{s^{(3+\theta)/2} +
\abs{x}^{3+\theta}}
 \lesssim \frac{\eps^\theta}{\abs{x}^{1+\theta}}\;,
\]
proving~\eqref{eq:KQeps01}. To prove~\eqref{eq:KQeps02}, we use spherical
coordinates as in the proof of Lemma~\ref{lem:KKQ} to get, for any constant
$c$ of order $1$,
\begin{align*}
\int_{\abs{x_1}\leqs c} \frac{\6x_1}{\abs{x_1}^{1+\theta}
\abs{x-x_1}^{1+\theta}} 
\asymp&{} \int_0^1 \frac{r^2}{r^{1+\theta}} \int_0^\pi 
\frac{\sin\phi\6\phi}{(r^2+\abs{x}^2-2r\abs{x}\cos\phi)^{(1+\theta)/2}} \6r \\
\asymp&{} \int_0^1 \frac{(r+\abs{x})^{1-\theta} - \bigabs{r-\abs{x}}^{1-\theta}}
{2\abs{x}r^\theta}\6r \\
\lesssim&{} \int_0^{2\abs{x}} \frac{\abs{x}^{1-\theta}}{\abs{x}r^\theta}\6r \\
&{}+ \int_{2\abs{x}}^1 \frac{r^{1-\theta}}{\abs{x}r^\theta} 
\biggbrak{\biggpar{1+\frac{\abs{x}}{r}}^{1-\theta} -
\biggpar{1-\frac{\abs{x}}{r}}^{1-\theta}} \6r  \\
\asymp&{} \abs{x}^{1-2\theta} + \int_{2\abs{x}}^1 \frac{\6r}{r^{2\theta}} \\
\lesssim&{} 1
\end{align*}
for $\abs{x}\lesssim 1$, provided $2\theta < 1$. The result follows since
$t\mapsto\KQ(t,x)$ has compact support. 
\qed

%%%%%%%%%%%%%%%%%%%%%%%%%%%%%%%%%%%%%%%%%%%%%%%%%%%%%%%%%%%%%%%%%%%%%%%%%%%%%%

\newpage

%\vfill

\small
\bibliography{BK}
\bibliographystyle{abbrv}               

\newpage
\tableofcontents

\goodbreak

\vfill

\bigskip\bigskip\noindent
{\small 
Nils Berglund \\ 
Universit\'e d'Orl\'eans, Laboratoire {\sc Mapmo} \\
{\sc CNRS, UMR 7349} \\
F\'ed\'eration Denis Poisson, FR 2964 \\
B\^atiment de Math\'ematiques, B.P. 6759\\
45067~Orl\'eans Cedex 2, France \\
{\it E-mail address: }{\tt nils.berglund@univ-orleans.fr}\\

\noindent 
Christian Kuehn \\
Vienna University of Technology \\ 
Institute for Analysis and Scientific Computing\\
1040 Vienna, Austria\\
{\it E-mail address: }{\tt ck274@cornell.edu}
}

%%%%%%%%%%%%%%%%%%%%%%%%%%%%%%%%%%%%%%%%%%%%%%%%%%%%%%%%%%%%%%%%%%%%%%%%%%%%%%
\end{document}